%% file: manuscript.tex
\documentclass[onecolumn]{IEEEtran}

\IEEEoverridecommandlockouts                              
\usepackage{graphicx}
\usepackage{amsmath}
\usepackage{amssymb}
\usepackage[font=footnotesize]{caption} 
\usepackage{subfig} 
\usepackage{cite} 
\usepackage{color}
\usepackage{algorithm} 
\usepackage{algpseudocode} 
\usepackage{enumerate}
\usepackage{hyperref}

\usepackage{setspace}
\usepackage{tikz}
\usetikzlibrary{calc} 

\newtheorem{theorem}{Theorem}[section]
\newtheorem{lemma}[theorem]{Lemma}
\newtheorem{proposition}[theorem]{Proposition}

\newtheorem{example}[theorem]{Example}
\newtheorem{definition}[theorem]{Definition}
\newtheorem{remark}[theorem]{Remark}

\newtheorem{problem}[theorem]{Problem}

\graphicspath{{figures/}}

\newcommand{\tr}{\mathrm{tr\,}}

\begin{document}

\title{\LARGE \bf
Optimal Sensor Placement for Target Localization and Tracking\\ in 2D and 3D
}

\author{Shiyu Zhao,
        Ben M. Chen
        and Tong H. Lee 
\thanks{S. Zhao, B. M. Chen and T. H. Lee are with the Department of Electrical and Computer Engineering, National University of Singapore, Singapore 117576, Singapore
    {\tt\small \{shiyuzhao, bmchen, eleleeth\}@nus.edu.sg}}
}

\maketitle

\begin{abstract}
This paper analytically characterizes optimal sensor placements for target localization and tracking in 2D and 3D.
Three types of sensors are considered: bearing-only, range-only, and received-signal-strength.
The optimal placement problems of the three sensor types are formulated as an identical parameter optimization problem and consequently analyzed in a unified framework.
Recently developed frame theory is applied to the optimality analysis.
We prove necessary and sufficient conditions for optimal placements in 2D and 3D.
A number of important analytical properties of optimal placements are further explored.
In order to verify the analytical analysis, we present a gradient control law that can numerically construct generic optimal placements.
\end{abstract}


\IEEEpeerreviewmaketitle

\section{Introduction}
Target localization and tracking using networked mobile sensor platforms has become an active research area in recent years.
When localizing a target from noisy measurements of multiple sensors, the sensor placement can significantly affect the estimation accuracy of any localization algorithms.
The term sensor placement as used here refers to the relative sensor-target geometry.
This paper will address the \emph{optimal} sensor placement that can minimize the target localization uncertainty.

In the literature, there are generally two kinds of formulations for optimal sensor placement problems.
One is optimal control \cite{Jarurat07,Andrew08,Oshman1999AES,Millerbook} and the other is parameter optimization \cite{bishop10,bishop07part2,Bishop09RSS,dogancy08,Sonia06,Zhang95,dogancay07AES,Jason09,Moreno2011}.

The optimal control formulation is usually adopted for cooperative path planning \cite{Jarurat07,Andrew08,Oshman1999AES}, where the aim is to estimate the target position on one hand and plan the path of sensor platforms to minimize the estimation uncertainty on the other hand.
This problem is also referred as simultaneous-localization-and-planning (SLAP) \cite{Andrew08}.
In a SLAP problem, target motion and sensor measurement models are considered as process and measurement models, respectively.
The Kalman filter usually is applied to estimate the target position and to characterize the estimation covariance.
In order to minimize the estimation covariance, an optimal control problem will be formulated.
The disadvantage of this kind of formulation is that the optimal control with various constraints generally can only be solved by numerical methods. Analytical properties usually cannot be obtained.

To avoid complicated optimal control problems, many studies including our work in this paper formulate the optimal placement problem as a parameter optimization problem.
This kind of formulation has a long history and has been investigated extensively in \cite{bishop10,bishop07part2,Bishop09RSS,dogancy08,Sonia06,Zhang95,dogancay07AES,Jason09,Moreno2011}, to name a few.
The parameter optimization formulation is based on the assumption that a rough estimation of the target position has already been obtained in other ways.
By using this rough estimation, the sensor positions are the parameters to be optimized such that the consequent target localization based on the optimized placement will be more accurate.
The objective function in the parameter optimization formulation usually involves the Fisher information matrix (FIM).
The FIM is the inverse of the Cramer Rao lower bound (CRLB), which is the minimum achievable estimation covariance.
An unbiased estimator that achieves the CRLB is called efficient.
An optimal placement, which maximizes a function (such as the determinant) of the FIM, can be interpreted as maximizing the target information gathered by the sensors or minimizing the estimation covariance of any efficient estimators.

In contrast to the optimal control formulation, the parameter optimization formulation can be solved analytically.
The analytical solutions are important for us to get insight to the effect of sensor placement on target localization uncertainty.
It is notable that the numerical results based on the optimal control formulation generally are consistent with the analytical results based on the parameter optimization formulation.
For example, the numerical simulations in \cite{Jarurat07,Andrew08} show that the final optimal angle subtended by two sensors at the target is 90 degrees.
In \cite{bishop10,dogancy08,Sonia06}, it is also analytically proved that the placement of two sensors is optimal if the angle subtended by the two sensors at the target is 90 degrees.

In this paper, we will investigate optimal sensor placement by adopting the parameter optimization formulation.
Our aim is to determine the optimal sensor-target geometry based on an initial estimation of the target position.
Optimal sensor placement is of not only theoretical interest but also significantly practical value.
Many studies have shown that target tracking performance can be improved when sensors are steered to form an optimal placement.
In this paper, we only focus on determining optimal placements and will not address target tracking.
One may refer to \cite{Sonia06} for a comprehensive example that illustrates the application optimal sensor placements to cooperative target tracking.

Until now, most of the existing studies only consider 2D optimal sensor placements
\cite{bishop10,bishop07part2,Bishop09RSS,dogancy08,Sonia06,Zhang95,Jason09}    .
Very few works in the literature have tackled 3D cases \cite{Moreno2011}.
Analytical characterization of generic optimal sensor placements in 3D is still an open problem.
In this paper we will extend the results in \cite{bishop10,dogancy08,Sonia06,Bishop09RSS} from 2D to 3D.
The extension will be non-trivial.
Maximizing the determinant of the FIM has been widely adopted as the criterion for optimal placements in 2D.
This criterion, however, cannot be directly applied to 3D cases because the determinant of the FIM is hardly analytically tractable in 3D.
Motivated by this, we will propose a new criterion for optimal placement, which enables us to analytically characterize optimal placements in 2D and 3D.
The existing analysis of 2D cases can be regarded as a special case of our general analysis for both 2D and 3D cases.

The existing work on optimal sensor placement has addressed many sensor types including bearing-only \cite{bishop10,dogancy08,zhao2012}, range-only \cite{Sonia06,bishop10,MITPaper}, received-signal-strength (RSS) \cite{Bishop09RSS}, time-of-arrival (TOA) \cite{bishop10,bishop07part2}, and time-difference-of-arrival (TDOA) \cite{bishop10,Jason09}.
These sensor types are analyzed individually in the literature.
One of the contributions of this paper is to unify the analyses of bearing-only, range-only, and RSS sensors.
Based on our proposed optimality criterion, we will show the objective functions for the three sensor types are exactly the same.
Hence their optimal placement can be analyzed in a unified way.
Since the measurement models and FIMs of TOA and TDOA sensors are significantly different from those of the three, we will not consider TOA and TDOA sensors in this paper.

By employing recently developed frame theory, we prove necessary and sufficient conditions for optimal placements in 2D and 3D.
Frames provide a redundant and robust way for representing signals and are widely used in signal processing.
We refer to \cite{Jelena07,Jelena07partII} for an introduction to frames.
It might be interesting to ask why frames arise in optimal sensor placement problems.
This question can be loosely answered from the redundancy point of view.
As mentioned in \cite{Jelena07}, one would use frames when redundancy is a must.
In our work, the redundancy can be interpreted as the ratio between the number of sensors and the space dimension.
When the sensor number equals the dimension, there is no redundancy in the system, then we will show that the necessary and sufficient condition of optimal placement can be proved without using frames.
But when the sensor number is larger than the space dimension, our optimality analysis will heavily rely on frame theory.

The paper is organized as follows.
Section \ref{section_frameTheory} introduces preliminaries to frame theory.
Section \ref{section_problemFormulation} presents a unified formulation for optimal placement problems of bearing-only, range-only, and RSS sensors in 2D and 3D.
In Section \ref{section_optimalityanalysis}, necessary and sufficient conditions for optimal placement in 2D and 3D are proved.
Section \ref{section_analyticalproperties} further explores a number of important properties of optimal placements.
In Section~\ref{section_control}, a gradient control law is proposed to numerically verify our analytical analysis.
Conclusions are drawn in Section \ref{section_conclusion}.

\section{Preliminaries to Frame Theory}\label{section_frameTheory}
Frames can be defined in any Hilbert space.
Here we are only interested in $d$-dimensional Euclidean space $\mathbb{R}^d$.
Let $\|\cdot\|$ be the Euclidean norm of a vector or the Frobenius norm of a matrix.
As shown in \cite{Benedetto03,Casazza06,Jelena07,Jelena07partII}, a set of vectors $\{\varphi_i\}_{i=1}^n$ in $\mathbb{R}^d$ ($n\ge d$) is called a frame if there exist constants $0<a\le b<+\infty$ so that for all $x\in\mathbb{R}^d$
\begin{align}\label{eq_frameDefinition}
    a\|x\|^2 \le \sum_{i=1}^n \langle x, \varphi_i\rangle^2 \le b\|x\|^2,
\end{align}
where $\langle \cdot,\cdot \rangle$ denotes the inner product of two vectors.
The constants $a$ and $b$ are called the \emph{frame bounds}.
A frame $\{\varphi_i\}_{i=1}^n$ is called \emph{unit-norm} if $\|\varphi_i\|=1$ for all $i\in\{1,\dots,n\}$.
Denote $\Phi=[\varphi_1,\dots,\varphi_n]\in\mathbb{R}^{d\times n}$.
Because $\langle x, \varphi_i\rangle ^2 =(x^T\varphi_i)^2 = x^T\varphi_i\varphi_i^Tx$,
inequality $\eqref{eq_frameDefinition}$ can be rewritten as
\begin{align*}
    a\|x\|^2\le x^T \Phi\Phi^T x\le b\|x\|^2,
\end{align*}
where the matrix $\Phi\Phi^T = \sum_{i=1}^n \varphi_i\varphi_i^T$ is called the \emph{frame operator}.
The frame bounds $a$ and $b$ obviously are the smallest and largest eigenvalues of $\Phi\Phi^T$, respectively.
Since $a>0$, $\Phi\Phi^T$ is positive definite.
It is well known that $d$ vectors in $\mathbb{R}^d$ form a basis if the vectors span $\mathbb{R}^d$.
Frame essentially is a generalization of the concept of basis.
The frame $\{\varphi_i\}_{i=1}^n$ can also span $\mathbb{R}^d$ because $\Phi\Phi^T$ is positive definite and hence $\Phi$ is of full row rank.
But compared to a basis, a frame have $n-d$ redundant vectors.
The constant $n/d$ is referred as the \emph{redundancy} of the system. When $n/d=1$, the frame would degenerate to a basis of $\mathbb{R}^d$.

Tight frame is a very important concept in frame theory.
A frame is called \emph{tight} when $a=b$.
From \eqref{eq_frameDefinition}, it is easy to see the frame $\{\varphi_i\}_{i=1}^n$ is tight when
\begin{align}\label{eq_tightframe}
    \sum_{i=1}^n \varphi_i\varphi_i^T = a I_d.
\end{align}
Taking trace on both sides of \eqref{eq_tightframe} yields $a=\sum_{i=1}^n \|\varphi_i\|^2/d$.
It is an important and fundamental problem in frame theory to construct a tight frame $\{\varphi_i\}_{i=1}^n$ that solves \eqref{eq_tightframe} with specified norms.
This problem is also recognized as notoriously difficult \cite{Casazza2012}.
One approach to this problem is to characterize tight frames as the minimizers of the \emph{frame potential}
\begin{align}\label{eq_framePotential}
    \mathrm{FP}\left(\{\varphi_i\}_{i=1}^n\right)=\sum_{i=1}^n\sum_{j=1}^n \left(\varphi_i^T\varphi_j\right)^2.
\end{align}
Frame potential is first proposed in \cite{Benedetto03} for unit-norm frames, and then generalized in \cite{Casazza06} for frames with arbitrary norms.

The following concept \emph{irregularity} \cite{Casazza06,Jelena07} is crucial for characterizing the minimizers of the frame potential.

\begin{definition}[Irregularity]\label{definition_irregularity}
    For any positive non-increasing sequence $\{c_i\}_{i=1}^n$ with $c_1\ge \dots \ge c_n>0$, and any integer $d$ satisfying $1\le d \le n$, denote $k_0$ as the smallest nonnegative integer $k$ for which
    \begin{align}
         \label{eq_regularPart} c_{k+1}^2 \le \frac{1}{d-k}\sum_{i=k+1}^n c_i^2.
    \end{align}
    The integer $k_0$ is called the irregularity of $\{c_i\}_{i=1}^n$ with respect to $d$.
\end{definition}

\begin{remark}
    The irregularity of a sequence is evaluated with respect to a particular positive integer.
    The irregularity of a sequence may be different when evaluated with respect to different positive integers.
    In this paper, we will omit mentioning this integer when the context is clear.
\end{remark}

Because the index $k=d-1$ always makes \eqref{eq_regularPart} hold, the irregularity $k_0$ always exists and satisfies
\begin{align*}
0\le k_0 \le d-1.
\end{align*}
When $k_0=0$, inequality \eqref{eq_regularPart} degenerates to the \emph{fundamental inequality} \cite{Casazza06}
\begin{align}\label{eq_fundamentalInequality}
        \underset{j=1,\dots,n}{\max} c_j^2\le \frac{1}{d}\sum_{i=1}^n c_i^2.
\end{align}
In this paper we call the sequence $\{c_i\}_{i=1}^n$ \emph{regular} when $k_0=0$,
and \emph{irregular} when $1\le k_0 \le d-1$.
The fundamental inequality \eqref{eq_fundamentalInequality} intuitively means that: a sequence is regular when no element is much larger than the others. Next we show several examples to illustrate the concept of irregularity.

\begin{example}\label{example_irregularity_equalcase}
Consider a sequence $\{c_i\}_{i=1}^n$ with $c=c_1=\dots=c_n$ and any $d\le n$. Because ${1}/{d}\sum_{i=1}^n c_i^2=nc^2/d\ge c^2$, the fundamental inequality \eqref{eq_fundamentalInequality} holds.
Thus this sequence is regular with respect to any integer $d\le n$.
\end{example}

\begin{example}
Consider a sequence $\{c_i\}_{i=1}^4=\{10,1,1,1\}$ and $d=3$.
Note the feature of this sequence is that one element is much larger than the others.
Because $10^2>1/3(10^2+1+1+1)$, the sequence is irregular with respect to $d=3$.
In order to determine the irregularity $k_0$, we need to further check if $\{c_i\}_{i=2}^4=\{1,1,1\}$ is regular with respect to $d-1=2$.
Since the elements of $\{c_i\}_{i=2}^4$ equal to each other, $\{c_i\}_{i=2}^4$ is regular with respect to $2$ as shown in Example \ref{example_irregularity_equalcase}.
Hence the irregularity of $\{c_i\}_{i=1}^4$ with respect to $d=3$ is $k_0=1$.
This example shows that a sequence would be irregular if certain element is much larger than the others.
\end{example}

\begin{example}
Consider a sequence $\{c_i\}_{i=1}^4=\{10,10,1,1\}$ and $d=2$ or $3$.
When $d=2$, we have $10^2<1/2(10^2+10^2+1+1)$.
Hence $\{c_i\}_{i=1}^4$ is regular with respect to $d=2$.
When $d=3$, we have $10^2>1/3(10^2+10^2+1+1)$, $10^2>1/2(10^2+1+1)$ and $1<1/1(1+1)$.
Hence $\{c_i\}_{i=1}^4$ is irregular with respect to $d=3$ and the irregularity is $k_0=2$.
This example shows that a sequence might be regular with respect to one integer but irregular with respect to another.
\end{example}

The minimizers of the frame potential in \eqref{eq_framePotential} are characterized by the following lemma \cite{Casazza06}, which will be used to prove the necessary and sufficient conditions for optimal placement.
\begin{lemma}
    \label{lemma_irregularframes}
    In $\mathbb{R}^d$, given a positive non-increasing sequence $\{c_i\}_{i=1}^n$ with irregularity as $k_0$, if the norms of the frame $\{\varphi_i\}_{i=1}^n$ are specified as $\|\varphi_i\|=c_i$ for all $i\in\{1,\dots,n\}$, any minimizer of the frame potential in \eqref{eq_framePotential} is of the form
    \begin{align*}
        \{\varphi_i\}_{i=1}^n=\{\varphi_i\}_{i=1}^{k_0} \cup \{\varphi_i\}_{i=k_0+1}^n,
    \end{align*}
    where $\{\varphi_i\}_{i=1}^{k_0}$ is an orthogonal set, and $\{\varphi_i\}_{i=k_0+1}^n$ is a tight frame in the orthogonal complement of the span of $\{\varphi_i\}_{i=1}^{k_0}$. Any local minimizer is also a global minimizer.
\end{lemma}

From Lemma \ref{lemma_irregularframes}, a minimizer of the frame potential consists of an orthogonal set $\{\varphi_i\}_{i=1}^{k_0}$ and a tight frame $\{\varphi_i\}_{i=k_0+1}^n$. The partition of the two sets is determined by the irregularity of the specified frame norms $\{c_i\}_{i=1}^n$.
When the irregularity $k_0=0$, it is clear that a minimizer of the frame potential is a tight frame.
As a corollary of Lemma \ref{lemma_irregularframes}, the following result \cite{Casazza06} gives the existence condition of the solutions to \eqref{eq_tightframe}.

\begin{lemma}
    \label{lemma_fundamentalInequality}
    In $\mathbb{R}^d$, given a positive sequence $\{c_i\}_{i=1}^n$, there exists a tight frame $\{\varphi_i\}_{i=1}^n$ with $\|\varphi_i\|=c_i$ for all $i\in\{1,\dots,n\}$ solving \eqref{eq_tightframe} if and only if $\{c_i\}_{i=1}^n$ is regular.
\end{lemma}


\section{Problem Formulation}\label{section_problemFormulation}
Consider one target and $n$ sensors in $\mathbb{R}^d$ ($d=2$ or 3 and $n\ge d$).
As shown in Table \ref{table:differentSensorTypes}, we consider three types of sensors: bearing-only, range-only, and RSS.
Suppose $n$ sensors involve only one sensor type. Sensor networks with mixed sensor types are not addressed.
Following \cite{bishop10,dogancy08,Sonia06,Bishop09RSS}, a rough estimation $p \in \mathbb{R}^d$ of the target position is assumed to be obtained in other ways.
Since $p$ is the only available information of the target, the optimal placement will be determined based on this estimation.
As illustrated in \cite{Sonia06}, the rough estimation of the target position can be obtained from a Kalman filter in practice and the optimal placement can be applied to improve target tracking performance.
It should be noted that this paper only focuses on determining optimal sensor placements and will not discuss their application to target localization or tracking.
Denote the position of sensor $i$ as $s_i\in\mathbb{R}^d$, $i\in\{1,\dots,n\}$.
Then $r_i=s_i-p$ denotes the position of sensor $i$ relative to the target.
The relative sensor-target placement can be fully described by $\{r_i\}_{i=1}^n$.
Our aim is to determine the optimal $\{r_i\}_{i=1}^n$ such that certain objective function can be optimized.
The distance between sensor $i$ and the target is given by $\|r_i\|$.
The unit-length vector $g_i=r_i/\|r_i\|$ represents the orientation of sensor $i$ relative to the target.

\subsection{Sensor Measurement Model and FIM}
For any sensor type in Table \ref{table:differentSensorTypes}, the measurement model of sensor $i$ is expressed as
\begin{align*}
    z_i=h_i(r_i)+v_i,
\end{align*}
where $z_i\in\mathbb{R}^m$ denotes the measurement of sensor $i$, the function $h_i(r_i): \mathbb{R}^d \rightarrow \mathbb{R}^m$ is determined by the type of the sensor, and
$v_i\in\mathbb{R}^m$ is the additive measurement noise.
We assume $v_i$ to be a zero-mean Gaussian noise with covariance as $\Sigma_i=\sigma_i^2 I_m\in\mathbb{R}^{m\times m}$, where $I_m$ denotes the $m\times m$ identity matrix.
By further assuming the measurement noises of different sensors are uncorrelated, the FIM given by $n$ sensors is expressed as
\begin{align}\label{eq_FIM}
    F
    = \sum_{i=1}^n \left(\frac{\partial h_i}{\partial p}\right)^T \Sigma_i^{-1} \frac{\partial h_i}{\partial p},
\end{align}
where $\partial h_i/\partial p$ denotes the Jacobian of $h_i(r_i)=h_i(s_i-p)$ with respect to $p$.
We refer to \cite{bishop10,dogancy08,Bishop09RSS,Sonia06,xielihuabook} for a detailed derivation of the FIM formula in \eqref{eq_FIM}.
\begin{table*}  
      \caption{Measurement models and FIMs of the three sensor types.} \label{table:differentSensorTypes}
      \centering
      \begin{tabular}{lllll}
        \hline
        Sensor type  & Measurement model                              & FIM                           & Coefficient                 &  Optimality criterion  \\
        \hline

        Bearing-only & $\displaystyle h_i(r_i)=\frac{r_i}{\|r_i\|}$ & $\displaystyle F=\sum_{i=1}^n c_i^2 (I_d-g_i g_i^T)$ & $\displaystyle c_i=\frac{1}{\sigma_i\|r_i\|}$ & $\displaystyle \min \left\|\sum_{i=1}^n c_i^2 g_i g_i^T\right\|^2$\\
        \hline

        Range-only   & $h_i(r_i)=\|r_i\|$                             & $\displaystyle F=\sum_{i=1}^n c_i^2 g_i g_i^T$ &     $\displaystyle c_i=\frac{1}{\sigma_i}$    & $\displaystyle \min \left\|\sum_{i=1}^n c_i^2 g_ig_i^T\right\|^2$\\
        \hline

        RSS          & $h_i(r_i)=\ln \|r_i\|$                         & $\displaystyle F=\sum_{i=1}^n c_i^2 g_ig_i^T$ &     $\displaystyle c_i=\frac{1}{\sigma_i\|r_i\|}$    & $\displaystyle \min \left\|\sum_{i=1}^n c_i^2 g_ig_i^T\right\|^2$ \\
        \hline
      \end{tabular}
\end{table*}

The measurement models of bearing-only, range-only, and RSS sensors are given in Table~\ref{table:differentSensorTypes}.
The measurement of a bearing-only sensor is conventionally modeled as one angle (azimuth) in 2D or two angles (azimuth and altitude) in 3D.
The drawback of this kind of model is that the model complexity increases dramatically as the dimension increases.
As shown in Table~\ref{table:differentSensorTypes}, we model the measurement of a bearing-only sensor as a unit-length vector pointing from the target to the sensor.
A unit-length vector essentially characterizes a bearing and is very suitable to represent a bearing-only measurement.
This model, which was proposed in our previous work \cite{zhao2012}, enables us to easily formulate optimal bearing-only placement in $\mathbb{R}^2$ and $\mathbb{R}^3$.
The measurement model of range-only sensors given in Table \ref{table:differentSensorTypes} is the same as the one in \cite{bishop10}.
The measurement model of RSS sensors given in Table \ref{table:differentSensorTypes} is a modified version of the one in \cite{Bishop09RSS}. Without loss of generality, we simplify the model in \cite{Bishop09RSS} by omitting certain additive and multiplicative constants.


The FIMs of the three sensor types are shown in Table \ref{table:differentSensorTypes}.
The FIM can be calculated by substituting $h(r_i)$ in to \eqref{eq_FIM}.
The calculation is straightforward and omitted here.
As will be shown later, the coefficients $\{c_i\}_{i=1}^n$ in the FIMs fully determine the optimal placements.
Following \cite{bishop10,dogancy08,Bishop09RSS,Sonia06}, we assume the coefficient $c_i$ to be \emph{arbitrary but fixed}.
(i) For bearing-only or RSS sensors, because $c_i=1/(\sigma_i\|r_i\|)$, both $\sigma_i$ and $\|r_i\|$ are assumed to be fixed.
Otherwise, if $\|r_i\|$ is unconstrained, the placement will be optimal when $\|r_i\|$ approaches zero.
To avoid this trivial solution, it is reasonable to assume $\|r_i\|$ to be fixed.
(ii) For range-only sensors, because $c_i=1/\sigma_i$, only $\sigma_i$ is assumed to be fixed.
Hence $\|r_i\|$ will have no influence on the optimality of the placement for range-only sensors.

To end this subsection, we would like to point out that the FIMs given in Table \ref{table:differentSensorTypes} are consistent with the ones given in \cite{bishop10,dogancy08,Bishop09RSS,Sonia06} in 2D cases.
To verify that, we can substitute $g_i=[\cos \theta_i, \sin \theta_i]^T\in\mathbb{R}^2$ into the FIMs in Table \ref{table:differentSensorTypes}.

\subsection{New Criterion for Optimal Placement}\label{subsetction_newcriterion}
The conventional criterion for optimal placement is to maximize the determinant of the FIM, i.e., $\det F$.
But $\det F$ is hardly analytically tractable in $\mathbb{R}^3$.
In order to analytically characterize optimal placements in $\mathbb{R}^2$ and $\mathbb{R}^3$, we will consider a new objective function.
Let $\{\lambda_i\}_{i=1}^d$ be the eigenvalues of $F$ and $\bar{\lambda}=1/d\sum_{i=1}^d \lambda_i$.
The new objective function considered in this paper is $\|F-\bar{\lambda}I_d\|^2$.
Compared to the conventional one $\det F$, the new objective function $\|F-\bar{\lambda}I_d\|^2$ is of strong analytical tractability.
We next formally state the optimal sensor placement problem that we are going to solve.

\begin{problem}\label{problem_problemDefinition}
    Consider one target and $n$ sensors in $\mathbb{R}^d$ ($d=2$ or $3$ and $n\ge d$).
    The sensors involve only one of the three sensor types in Table \ref{table:differentSensorTypes}.
    Given arbitrary but fixed positive coefficients $\{c_i\}_{i=1}^n$, find the optimal placement $\{g_i^*\}_{i=1}^n$ such that
    \begin{align}\label{eq_problemdefinition1}
        \{g_i^*\}_{i=1}^n=\underset{\{g_i\}_{i=1}^n \subset \mathbb{S}^{d-1}}{\arg\min} \|F-\bar{\lambda}I_d\|^2,
    \end{align}
    where $\mathbb{S}^{d-1}$ denotes the unit sphere in $\mathbb{R}^d$.
\end{problem}

\begin{remark}
    The relative sensor-target placement can be fully described by $\{r_i\}_{i=1}^n$.
    Recall $\|r_i\|$ is assumed to be fixed for bearing-only or RSS sensors, and $\|r_i\|$ has no effect on the placement optimality for range-only sensors.
    Thus for any sensor type, the optimal sensor placement can also be fully described by $\{g_i\}_{i=1}^n$.
    That means we only need to determine the optimal relative sensor-target bearings $\{g_i^*\}_{i=1}^n$ to obtain the optimal placement.
\end{remark}


Although the FIMs of different sensor types may have different formulas as shown in Table~\ref{table:differentSensorTypes},
the following result shows that substituting the FIMs of the three sensor types into \eqref{eq_problemdefinition1} will lead to an identical objective function.
The following result is important because it enables us to \emph{unify} the formulations of optimal placement for the three sensor types.

\begin{lemma}\label{lemma_problemdefinition2}
    Consider one target and $n$ sensors in $\mathbb{R}^d$ ($d=2$ or $3$ and $n\ge d$).
    The sensors involve only one of the three sensor types in Table \ref{table:differentSensorTypes}.
    The problem defined in \eqref{eq_problemdefinition1} is equivalent to
    \begin{align}\label{eq_problemdefinition2}
        \{g_i^*\}_{i=1}^n = \underset{\{g_i\}_{i=1}^n\subset\mathbb{S}^{d-1}}{\arg\min} \left\| G \right\|^2,
    \end{align}
    where $G=\sum_{i=1}^n c_i^2 g_i g_i^T$.
\end{lemma}
\begin{proof}
    If all sensors are bearing-only, the FIM is $F=\sum_{i=1}^n c_i^2 (I_d-g_i g_i^T)$ and then $\bar{\lambda}={1}/{d}\sum_{j=1}^d \lambda_j=\tr F /d={(d-1)}/{d}\sum_{i=1}^n c_i^2$. Hence
    \begin{align*}
        F-\bar{\lambda}I_d
        &=\sum_{i=1}^n c_i^2 (I_d-g_i g_i^T) -\frac{d-1}{d}\sum_{i=1}^n c_i^2 I_d \\
        &=-\sum_{i=1}^n c_i^2 g_i g_i^T + \frac{1}{d}\sum_{i=1}^n c_i^2I_d.
    \end{align*}
    If all sensors are range-only or RSS, the FIM is $F=\sum_{i=1}^n c_i^2 g_ig_i^T$ and then $\bar{\lambda}={1}/{d}\sum_{j=1}^d \lambda_j=\tr F/d={1}/{d}\sum_{i=1}^n c_i^2$. Hence
    \begin{align*}
        F-\bar{\lambda}I_d =\sum_{i=1}^n c_i^2 g_i g_i^T-\frac{1}{d}\sum_{i=1}^n c_i^2I_d.
    \end{align*}
    Therefore, for any one of the three sensor types given in Table~\ref{table:differentSensorTypes}, the new objective function can be rewritten as
    \begin{align}\label{eq_objectivefunctionvalue1}
        \|F-\bar{\lambda}I_d\|^2
        &= \left\|\sum_{i=1}^n c_i^2 g_ig_i^T-\frac{1}{d}\sum_{i=1}^n c_i^2 I_d\right\|^2 \nonumber \\
        &= \left\|G\right\|^2 -\frac{1}{d}\left(\sum_{i=1}^n c_i^2\right)^2.
    \end{align}
    Because $1/d(\sum_{i=1}^n c_i^2)^2$ is constant, minimizing $\|G\|^2$ is equivalent to minimizing $\|F-\bar{\lambda}I_d\|^2$.
\end{proof}

One primary task of this paper is to solve the parameter optimization problem \eqref{eq_problemdefinition2}.
It should be noted that we must clearly know the sensor type that we work with, such that the coefficients $\{c_i\}_{i=1}^n$ in $G$ can be calculated correctly according to the sensor type.
Once $\{c_i\}_{i=1}^n$ have been calculated, the sensor types will be transparent to us.
As a consequence, the analysis of optimal sensor placement in the sequel of the paper will apply to all the three sensor types.


\subsection{Relationship between the New and Conventional Criterions}

The new criterion for optimal placement is to minimize $\|F-\bar{\lambda}I_d\|^2$, while the widely adopted conventional one is to maximize $\det F$.
We will show that the new criterion has a very close connection to the conventional one.

\begin{lemma}\label{lemma_newoldequivalence}
    For any one of the three sensor types given in Table \ref{table:differentSensorTypes}, we have
    \begin{align*}
        \det F\le \bar{\lambda}^d,
    \end{align*}
    where the equality holds if and only if
    \begin{align}
        \|F-\bar{\lambda}I_d\|^2=0.
    \end{align}
\end{lemma}

\begin{proof}
For any one of the three sensor types, the FIM $F$ is symmetric positive (semi) definite. Hence $\lambda_j$ is real and nonnegative.
From the FIMs shown in Table \ref{table:differentSensorTypes}, we have $\sum_{j=1}^d \lambda_j=\tr F=(d-1)\sum_{i=1}^n c_i^2$ for bearing-only sensors, and $\sum_{j=1}^d \lambda_j=\tr F=\sum_{i=1}^n c_i^2$ for range-only or RSS sensors.
Note $\{c_i\}_{i=1}^n$ is assumed to be fixed. Hence $\sum_{j=1}^d \lambda_j$ is an invariant quantity for any one of the three sensor types.
By the inequality of arithmetic and geometric means, the conventional objective function $\det F$ satisfies
\begin{align}\label{eq_detF_upperbound}
    \det F
    =\prod_{j=1}^d \lambda_j
    \le \left(\frac{1}{d}\sum_{j=1}^d \lambda_j\right)^d
    =\bar{\lambda}^d,
\end{align}
where the equality holds if and only if $\lambda_j=\bar{\lambda}$ for all $j\in\{1,\dots,d\}$, which means
\begin{align}\label{eq_detF_upperbound_condition}
    F=\bar{\lambda}I_d \Leftrightarrow \|F-\bar{\lambda}I_d\|^2=0.
\end{align}
To sum up, $\det F$ is maximized to its upper bound $\bar{\lambda}^d$ if and only if $\|F-\bar{\lambda}I_d\|^2=0$.
\end{proof}

Based on Lemma \ref{lemma_newoldequivalence}, we next further examine the relationship between the new and conventional criterions case by case.
(i) In $\mathbb{R}^2$,
we have $\det F = {1}/{2}((\tr\mathcal{F})^2 - \tr(F^2)) = {1}/{2} \left(4\bar{\lambda}^2 - \|F\|^2\right)$
and $\|F-\bar{\lambda}I_2\|^2=\tr {(F-\bar{\lambda}I_2)^2}= \|F\|^2-2\bar{\lambda}^2$, which suggest
\begin{align*}
    \|F-\bar{\lambda}I_2\|^2=-2\det F + 2\bar{\lambda}^2.
\end{align*}
Because $2\bar{\lambda}^2$ is constant, minimizing $\|F-\bar{\lambda}I_2\|^2$ is rigorously equivalent to maximizing $\det F$ in $\mathbb{R}^2$.
Thus the new criterion is rigorously equivalent to the conventional one for all 2D cases.
As a consequence, our analysis based on the new criterion will be consistent with the 2D results in \cite{bishop10,dogancy08,Sonia06,Bishop09RSS}.
(ii) In $\mathbb{R}^3$, if $\|F-\bar{\lambda}I_3\|^2$ is able to achieve zero, then $\det F$ can be maximized to its upper bound as shown in Lemma \ref{lemma_newoldequivalence}.
In this case the new criterion is still rigorously equivalent to the conventional one.
(iii) In $\mathbb{R}^3$, $\|F-\bar{\lambda}I_3\|^2$ is not able to reach zero in certain \emph{irregular} cases
(see Section~\ref{section_optimalityanalysis} for the formal definition of irregular).
In these cases $\det F$ and $\|F-\bar{\lambda}I_3\|^2$ may not be optimized simultaneously.
But as will be shown later, the analysis of irregular cases in $\mathbb{R}^3$ based on the new criterion is a also reasonable extension of the analysis of irregular cases in $\mathbb{R}^2$.

\subsection{Equivalent Placements}
Before solving \eqref{eq_problemdefinition2}, we identify a group of placements that result in the same value of $\|G\|^2$.

\begin{proposition}\label{propsition_eigenvalueInvariance}
    The objective function $\|G\|^2$ is invariant to the sign of $g_i$ for all $i\in\{1,\dots,n\}$ and any orthogonal transformations over $\{g_i\}_{i=1}^n$.
\end{proposition}
\begin{proof}
    First, $g_i g_i^T=(-g_i )(-g_i )^T$ for all $i\in\{1,\dots,n\}$, hence $\|G\|^2$ is invariant to the sign of $g_i$.
    Second, let $U\in \mathbb{R}^{d\times d}$ be an orthogonal matrix satisfying $U^TU=I_d$. Applying $U$ to $\{g_i\}_{i=1}^n$ yields $\{g_i'=Ug_i\}_{i=1}^n$. Then we have $G'= \sum_{i=1}^n c_i^2 g_i'(g_i')^T =\sum_{i=1}^n c_i^2 (Ug_i)(Ug_i)^T = UGU^T$.
    Since $G$ and $G'$ are both symmetric, we have $\|G'\|^2=\tr(UGU^TUGU^T)=\tr(G^2)=\|G\|^2$.
\end{proof}


Geometrically speaking, changing the sign of $g_i$ means flipping sensor $i$ about the target, and an orthogonal transformation represents a rotation, reflection or both combined operation over all sensors.
It is noticed that the invariance to the sign change of $g_i$ was originally recognized in \cite{dogancy08} for bearing-only sensors.
By Proposition \ref{propsition_eigenvalueInvariance}, we define the following equivalence relationship.

\begin{definition}[Equivalent placements]\label{definition_equivalence}
    Given arbitrary but fixed coefficients $\{c_i\}_{i=1}^n$, two placements $\{g_i\}_{i=1}^n$ and $\{g_i'\}_{i=1}^n$ are called equivalent if they are differed by indices permutation, flipping any sensors about the target, or any global rotation, reflection or both combined over all sensors.
\end{definition}

Due to the equivalence, there always exist an infinite number of equivalent optimal placements minimizing $\|G\|^2$.
If two optimal placements are equivalent, they lead to the same objective function value.
But the converse statement is not true in general.
In Section \ref{subsection_uniqueness}, we will give the condition under which the converse is true.
Examples of 2D equivalent placements are given in Figure~\ref{fig_EquivalentPlacements}.

\begin{figure}
  \centering
  \includegraphics[]{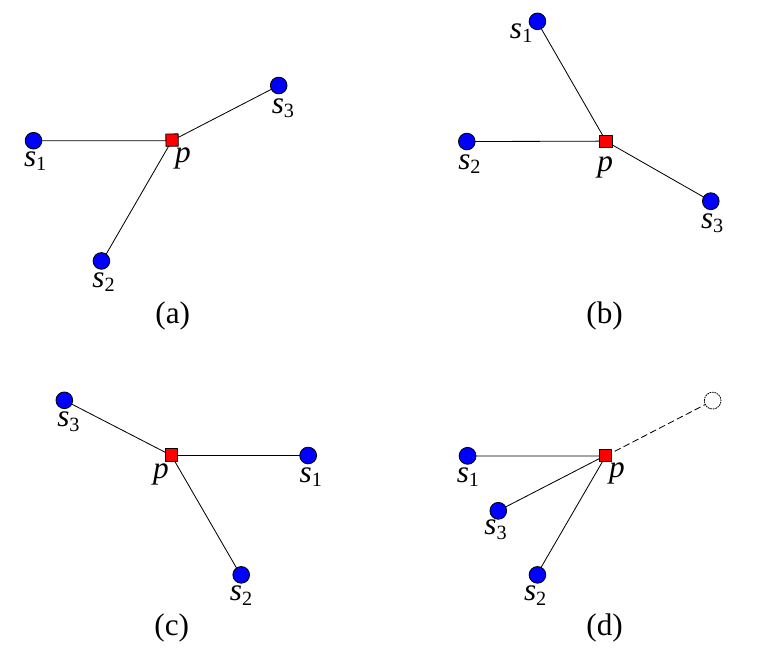}
  \caption{Examples of equivalent placements ($d=2,n=3$): (a) Original placement. (b) Rotate all sensors about the target 60 degrees clockwise. (c) Reflect all sensors about the vertical axis. (d) Flipping the sensor $s_3$ about the target.}
  \label{fig_EquivalentPlacements}
\end{figure}

\section{Necessary and Sufficient Conditions for Optimal Placement}\label{section_optimalityanalysis}
In this section, we will prove the necessary and sufficient conditions for optimal placements solving \eqref{eq_problemdefinition2}.
Recall $G=\sum_{i=1}^n c_i^2g_ig_i^T$. Then we have
\begin{align*}
    \|G\|^2
    &=\sum_{i=1}^n \sum_{j=1}^n (c_ic_jg_i^Tg_j)^2 \\
    &=\sum_{i=1}^n \sum_{j=1}^n (\varphi_i\varphi_j)^2,
\end{align*}
where $\varphi_i=c_i g_i$ for $i\in\{1,\dots,n\}$.
The vectors $\{\varphi_i\}_{i=1}^n$ form a frame in $\mathbb{R}^d$.
The objective function $\|G\|^2$ is exactly the frame potential of the frame $\{\varphi_i\}_{i=1}^n$ as shown in \eqref{eq_framePotential}.
The matrix $G$ is the frame operator.
Note here $\|\varphi_i\|=c_i$.
The coefficient sequence $\{c_i\}_{i=1}^n$ will fully determine the minimizers of $\|G\|^2$.
According to the irregularity of $\{c_i\}_{i=1}^n$, optimal placements are categorized as regular and irregular as shown below.

When $\{c_i\}_{i=1}^n$ is regular, the necessary and sufficient condition of optimal placement is given below.
The 2D version of the following result has been proposed in \cite{bishop10,dogancy08,Bishop09RSS}.

\begin{theorem}[Regular optimal placement]\label{theorem_regularOptimalPlacement}
    In $\mathbb{R}^d$ with $d=2$ or $3$, 
    if the positive coefficient sequence $\{c_i\}_{i=1}^n$ is regular, then the objective function $\|G\|^2$ satisfies
    \begin{align}\label{eq_lowerboundRegular}
        \|G\|^2 \ge \frac{1}{d} \left( \sum_{i=1}^n c_i^2 \right)^2.
    \end{align}
    The lower bound of $\|G\|^2$ is achieved if and only if
    \begin{align}\label{eq_RegularOptimalPlacement}
        \sum_{i=1}^n c_i^2g_ig_i^T=\frac{1}{d}\sum_{i=1}^n c_i^2 I_d.
    \end{align}
\end{theorem}
\begin{proof}
    Let $\{\mu_j\}_{j=1}^d$ be the eigenvalues of $G$.
    Then $\sum_{j=1}^d \mu_j=\tr G=\sum_{i=1}^n c_i^2$ is constant.
    Let $\bar{\mu}={1}/{d}\sum_{j=1}^d \mu_j={1}/{d}\sum_{i=1}^n c_i^2$.
    It is obvious that
    \begin{align*}
        \|G\|^2
        =\sum_{j=1}^d \mu_j^2
        \ge d\bar{\mu}^2
        =\frac{1}{d} \left( \sum_{i=1}^n c_i^2 \right)^2.
    \end{align*}
    The lower bound of $\|G\|^2$ is achieved if and only if $\mu_j=\bar{\mu}$ for all $j\in\{1,\dots,d\}$, which implies
    $    G=\bar{\mu}I_d$,
    i.e., the equation \eqref{eq_RegularOptimalPlacement}.
    By denoting $\varphi_i=c_ig_i$, \eqref{eq_RegularOptimalPlacement} becomes $\sum_{i=1}^n \varphi_i\varphi_i^T={1}/{d}\sum_{i=1}^n c_i^2 I_d$ which is the same as \eqref{eq_tightframe}.
    Thus a regular optimal placement solving \eqref{eq_RegularOptimalPlacement} actually corresponds to a tight frame.
    Because $\{c_i\}_{i=1}^n$ is regular, by Lemma \ref{lemma_fundamentalInequality} there exist optimal placements solving \eqref{eq_RegularOptimalPlacement}.
\end{proof}

We call a placement \emph{regular} when its coefficient sequence is regular, and \emph{regular optimal} when it solves \eqref{eq_RegularOptimalPlacement}.
To obtain a regular optimal placement, we still need to solve \eqref{eq_RegularOptimalPlacement}.
Details of the solutions to \eqref{eq_RegularOptimalPlacement} will be given in Section \ref{subsection_explicitconstruction}.

\begin{remark}
    When $\|G\|^2$ reaches its lower bound ${1}/{d} ( \sum_{i=1}^n c_i^2 )^2$, we have $\|F-\bar{\lambda}I_d\|^2=0$ by \eqref{eq_objectivefunctionvalue1}.
    By Lemma \ref{lemma_newoldequivalence}, the conventional objective function $\det F$ will be maximized to its upper bound.
    Therefore, regular optimal placements not only minimize the new objective functions $\|G\|^2$ and $\|F-\bar{\lambda}I_d\|^2$ but also maximize the conventional one $\det F$ in $\mathbb{R}^2$ and $\mathbb{R}^3$.
\end{remark}

When $\{c_i\}_{i=1}^n$ is irregular, \eqref{eq_RegularOptimalPlacement} will have no solution.
Then the the necessary and sufficient condition of optimal placement is given below.
The 2D version of the following result has been proposed in \cite{bishop10,dogancy08,Bishop09RSS}.

\begin{theorem}[Irregular optimal placement]\label{theorem_irregularOptimalPlacement}
    In $\mathbb{R}^d$ with $d=2$ or $3$, if the positive coefficient sequence $\{c_i\}_{i=1}^n$ is irregular with irregularity as $k_0\ge1$, without loss of generality $\{c_i\}_{i=1}^n$ can be assumed to be a non-increasing sequence, and then the objective function $\|G\|^2$ satisfies   \begin{align}\label{eq_lowerboundIrregular}
        \|G\|^2 \ge \sum_{i=1}^{k_0} c_i^4 + \frac{1}{d-k_0}\left( \sum_{i=k_0+1}^n c_i^2 \right)^2.
    \end{align}
    The lower bound of $\|G\|^2$ is achieved if and only if
    \begin{align}\label{eq_irregularOptimalPlacement}
        \{g_i\}_{i=1}^n=\{g_i\}_{i=1}^{k_0} \cup \{g_i\}_{i=k_0+1}^n,
    \end{align}
    where $\{g_i\}_{i=1}^{k_0}$ is an orthogonal set, and $\{g_i\}_{i=k_0+1}^n$ forms a regular optimal placement in the $(d-k_0)$-dimensional orthogonal complement of $\{g_i\}_{i=1}^{k_0}$.
\end{theorem}
\begin{proof}
    Recall $\|G\|^2$ is the frame potential of the frame $\{\varphi_i\}_{i=1}^n$ where $\varphi_i=c_i g_i$.
    From Lemma \ref{lemma_irregularframes}, the minimizer of $\|G\|^2$ is of the following form: $\{c_ig_i\}_{i=1}^{k_0}$ is an orthogonal set, and $\{c_ig_i\}_{i=k_0+1}^n$ is a tight frame (i.e., a regular optimal placement) in the orthogonal complement of $\{c_ig_i\}_{i=1}^{k_0}$.

    Let $\Phi_1=[\varphi_1,\dots,\varphi_{k_0}]\in\mathbb{R}^{d\times k_0}$, $\Phi_2=[\varphi_{k_0+1},\dots,\varphi_n]$ $\in\mathbb{R}^{d\times (n-k_0)}$, and
    $\Phi=[\Phi_1, \Phi_2]\in\mathbb{R}^{d\times n}$.
    When $\{g_i\}_{i=1}^n$ is of the form in \eqref{eq_irregularOptimalPlacement}, the columns of $\Phi_1$ are orthogonal to those of $\Phi_2$. Then
    \begin{align*}
        \|G\|^2= \tr (\Phi^T\Phi)^2 = \tr (\Phi_1^T\Phi_1)^2 + \tr (\Phi_2^T\Phi_2)^2.
    \end{align*}
    Because $\{g_i\}_{i=1}^{k_0}$ is an orthogonal set, we have $\tr (\Phi_1^T\Phi_1)^2=\sum_{i=1}^{k_0}\|\varphi_i\|^4=\sum_{i=1}^{k_0}c_i^4$.
    Because $\{g_i\}_{i=k_0+1}^n$ is a regular optimal placement in a $(d-k_0)$-dimensional subspace, we have $\tr (\Phi_2^T\Phi_2)^2={1}/{(d-k_0)}(\sum_{i=k_0+1}^n c_i^2)^2$ by Theorem \ref{theorem_regularOptimalPlacement}. Therefore, when $\{g_i\}_{i=1}^n$ is of the form in \eqref{eq_irregularOptimalPlacement}, the objective function $\|G\|^2$ reaches its lower bound as shown in \eqref{eq_lowerboundIrregular}.
\end{proof}

\begin{figure}
  \centering
  \includegraphics[]{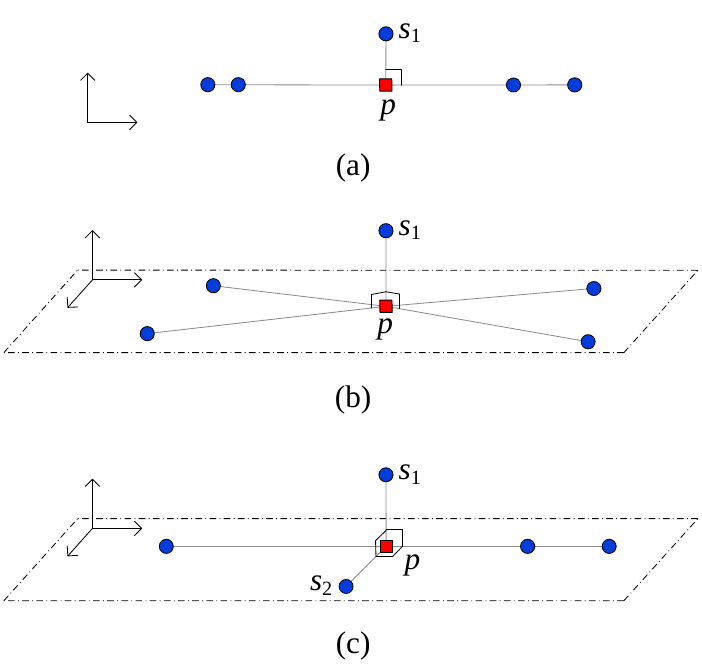}
  \caption{
  An illustration of the three kinds of irregular optimal placements in $\mathbb{R}^2$ and $\mathbb{R}^3$. (a) $d=2$, $k_0=1$; (b) $d=3$, $k_0=1$; (c) $d=3$, $k_0=2$.
  }
  \label{fig_irregularPlacements}
\end{figure}

We call a placement \emph{irregular} when its coefficient sequence is irregular, and \emph{irregular optimal} when it is of the form in \eqref{eq_irregularOptimalPlacement}.
In Theorem \ref{theorem_irregularOptimalPlacement}, because $\{g_i\}_{i=k_0+1}^n$ is a regular optimal placement in a $(d-k_0)$-dimensional space, an irregular optimal placement problem will eventually convert to a regular one in a lower dimensional subspace.

Apparently the irregularity of $\{c_i\}_{i=1}^n$ plays a key role in determining optimal placements.
Recall the irregularity $k_0$ of an irregular sequence with respect to $d$ satisfies $1\le k_0 \le d-1$.
In $\mathbb{R}^2$, we have $d=2$ and hence $k_0=1$; in $\mathbb{R}^3$, we have $d=3$ and hence $k_0=1$ or $2$.
Thus there exist only \emph{three} kinds of irregular optimal placements in $\mathbb{R}^2$ and $\mathbb{R}^3$.
By Theorem \ref{theorem_irregularOptimalPlacement}, these three kinds of irregular optimal placements can be intuitively described as below.

\begin{enumerate}[(i)]
  \item Irregular optimal placement in $\mathbb{R}^2$ with irregularity $k_0=1$:
    the vector $g_1$ is orthogonal to $\{g_i\}_{i=2}^n$, and $\{g_i\}_{i=2}^n$ are collinear. See an illustration in Figure \ref{fig_irregularPlacements} (a).
  \item Irregular optimal placement in $\mathbb{R}^3$ with irregularity $k_0=1$:
    the vector $g_1$ is orthogonal to $\{g_i\}_{i=2}^n$, and $\{g_i\}_{i=2}^n$ form a regular optimal placement in the 2D plane perpendicular to $g_1$. See an illustration in Figure \ref{fig_irregularPlacements} (b).
  \item Irregular optimal placement in $\mathbb{R}^3$ with irregularity $k_0=2$:
    the vectors $g_1$, $g_2$ and $\{g_i\}_{i=3}^n$ are mutually orthogonal, and $\{g_i\}_{i=3}^n$ are collinear. See an illustration in Figure \ref{fig_irregularPlacements} (c).
\end{enumerate}

We now intuitively explain why irregular optimal placements are the way as described above.
The sequence $\{c_i\}_{i=1}^n$ is irregular when certain $c_i$ is much larger than the others.
The coefficient $c_i$ actually is the weight for sensor $i$.
The larger the weight $c_i$ is, the more sensor $i$ contributes to the FIM.
Recall $c_i=1/(\sigma_i\|r_i\|)$ or $c_i=1/\sigma_i$.
If sensor $i$ is very close to the target (i.e., $\|r_i\|$ is small) and its measurement is very accurate (i.e., $\sigma_i$ is small), then $c_i$ will be very large and sensor $i$ may dominate all the others.
Thus the dominant sensor can measure the target sufficiently accurately in one dimension, and all the other sensors should measure the target in the orthogonal complement in order to improve the overall measurement accuracy.

To make our analysis more general, we do not assume $\sigma_i$'s to be identical. It is also meaningful to check the special case, $\sigma_i=\sigma_j$ for all $i\ne j$, which often arises in practice.
For bearing-only or RSS-based sensors, the coefficient is $c_i=1/(\sigma_i\|r_i\|)$. So when $\sigma_i=\sigma_j$ for all $i\ne j$,
a regular sequence $\{c_i\}_{i=1}^n$ implies the following equation according to the fundamental inequality \eqref{eq_fundamentalInequality}:
\begin{align} \label{eq_fundamentalInequalityRanges}
    \underset{j=1,\dots,n}{\max}\frac{1}{\|r_j\|^2} \le \frac{1}{d}\sum_{i=1}^n \frac{1}{\|r_i\|^2},
\end{align}
which geometrically means no sensor is much closer to the target than the others.
Conversely, if certain sensors are much closer to the target than the others, then the sequence $\{c_i\}_{i=1}^n$ will be irregular.
The 2D version of inequality \eqref{eq_fundamentalInequalityRanges} has been proposed in \cite{bishop10,dogancy08,Bishop09RSS}.
For range-only sensors, the coefficient is $c_i=1/\sigma_i$. If $\sigma_i=\sigma_j$ for all $i\ne j$, then $c_i=c_j$ and $\{c_i\}_{i=1}^n$ is regular with respect to any $d\le n$.

We next consider an important special case $n=d$, i.e., the sensor number equals to the dimension of the space.
This case is important because the optimal placement will be independent to the coefficients $\{c_i\}_{i=1}^n$ in this case.
The optimal placement in the case of $n=d=2$ has been solved by \cite{bishop10,dogancy08,Bishop09RSS,Sonia06}.

\begin{theorem}\label{theorem_d=n}
    In $\mathbb{R}^d$ with $d=2$ or $3$, if $n=d$, the objective function $\|G\|^2$ satisfies
    \begin{align*}
        \|G\|^2 \ge \sum_{i=1}^d c_i^4.
    \end{align*}
    The lower bound of $\|G\|^2$ is achieved if and only if $\{g_i\}_{i=1}^d$ is an orthogonal basis of $\mathbb{R}^d$.
\end{theorem}
\begin{proof}
     Since $G=\sum_{i=1}^d c_i^2 g_i g_i^T$ and $g_i^Tg_i=1$ for all $i\in\{1,\dots,n\}$, we have
    \begin{align*}
        \|G\|^2
        &=\tr (G^2) \\
        &=\sum_{i=1}^d\sum_{j=1}^d c_i^2 c_j^2 (g_i^Tg_j)^2 \\
        &=\sum_{i=1}^d\sum_{j=1,j\ne i}^d c_i^2 c_j^2 (g_i^Tg_j)^2 + \sum_{i=1}^d c_i^4 \\
        &\ge \sum_{i=1}^d c_i^4,
    \end{align*}
    where the equality holds if and only if $g_i^Tg_j=0$ for all $i,j\in\{1,\dots,d\}$ and $i\ne j$.
\end{proof}

Theorem \ref{theorem_d=n} can also be proved as a corollary of Theorem \ref{theorem_regularOptimalPlacement} and Theorem \ref{theorem_irregularOptimalPlacement}. But as shown above, we can also directly prove it without using frame theory.
This can be explained from the redundancy point of view.
Recall the constant $n/d$ reflects the redundancy of the system.
There will be no redundancy when $n/d=1$.
Then frames are no longer necessary for the optimality analysis.

\section{Analytical Properties of Optimal Placements}\label{section_analyticalproperties}
In this section, we further explore a number of analytical properties of optimal placements in 2D and 3D. Theorem \ref{theorem_irregularOptimalPlacement} implies that an irregular optimal placement problem can be eventually converted to a regular one in a lower dimensional space. Hence we will only focus on regular optimal placements.

\subsection{Explicit Construction}\label{subsection_explicitconstruction}
The work in \cite{bishop10,dogancy08,Bishop09RSS,Sonia06} has proposed explicit construction methods for some special 2D optimal placements.
The construction of generic optimal placements in 2D or 3D is still an open problem.
From Theorem \ref{theorem_regularOptimalPlacement}, we know a regular optimal placement is essentially a tight frame.
In fact, construction of tight frames in arbitrary dimensions has been extensively studied.
Thus by referring to \cite{dejun2006,Casazza2006,FICKUSPreprint2,Casazza2012}, to name a few, we are able to construct optimal placements with an arbitrary number of sensors and arbitrary but fixed coefficients in 2D and 3D.

We here present a proof of the necessary and sufficient condition for 2D regular optimal placements \emph{without} using frame theory. In the meantime, we propose an explicit construction method for arbitrary 2D regular optimal placements.
The necessary and sufficient condition for 2D optimal placements has already been given in \cite{bishop10,dogancy08,Bishop09RSS}, where the sufficiency, however, is not proved. We will prove the sufficiency by construction.
The following lemma can be found in \cite{bishop10,dogancy08,Bishop09RSS,Sonia06,Benedetto03,Goyal01,FickusDissertation01}.

\begin{lemma}\label{lemma_2DEquivalentProblem}
    In $\mathbb{R}^2$, the unit-length vector $g_i$ can be written as $g_i=[\cos\theta_i, \sin\theta_i]^T$. Then \eqref{eq_RegularOptimalPlacement} is equivalent to
    \begin{align}\label{eq_2DEquivalentProblem}
        \sum_{i=1}^n c_i^2\bar{g}_i=0,
    \end{align}
    where $\bar{g}_i=\left[\cos 2\theta_i, \sin 2\theta_i\right]^T$.
\end{lemma}
\begin{proof}
     Substituting $g_i=[\cos\theta_i, \sin\theta_i]^T$ into \eqref{eq_RegularOptimalPlacement} gives
     \begin{align*}
        \sum_{i=1}^n c_i^2 \left[
                             \begin{array}{cc}
                               \frac{1}{2} \cos 2\theta_i   &   \frac{1}{2} \sin 2\theta_i \\
                               \frac{1}{2} \sin 2\theta_i   &   -\frac{1}{2} \cos 2\theta_i \\
                             \end{array}
                           \right]=0,
    \end{align*}
    which is equivalent to \eqref{eq_2DEquivalentProblem}.
\end{proof}

By Lemma \ref{lemma_2DEquivalentProblem}, the matrix equation \eqref{eq_RegularOptimalPlacement} is simplified to a vector equation \eqref{eq_2DEquivalentProblem}.
In order to construct $\{g_i\}_{i=1}^n$ solving \eqref{eq_RegularOptimalPlacement}, we can first construct $\{\bar{g}_i\}_{i=1}^n$ solving \eqref{eq_2DEquivalentProblem}.
Once $\bar{g}_i=[\cos2\theta_i, \sin2\theta_i]^T$ is obtained, $g_i$ can be retrieved as $g_i=\pm[\cos\theta_i, \sin\theta_i]^T$. Note the sign changes of $g_i$ give equivalent optimal placement as mentioned in Definition \ref{definition_equivalence}.

\begin{theorem}\label{theorem_2DNecessaryandSufficient}
    In $\mathbb{R}^2$, given a positive sequence $\{c_i\}_{i=1}^n$, there exists $\{\bar{g}_i\}_{i=1}^n$ with $\|\bar{g}_i\|=1$ solving \eqref{eq_2DEquivalentProblem} if and only if
    \begin{align}\label{eq_fundamentalInequalityFor2D}
        \max_{j=1,\dots, n} c_j^2\le\frac{1}{2}\sum_{i=1}^n c_i^2.
    \end{align}
\end{theorem}
\begin{proof}
    Necessity: If $\sum_{i=1}^n c_i^2\bar{g}_i=0$, then $c_j^2\bar{g}_j=\sum_{i\ne j} c_i^2\bar{g}_i$ for all $j\in\{1,\dots, n\}$.
    Hence $c_j^2=\|c_j^2\bar{g}_j\|=\|\sum_{i\ne j} c_i^2\bar{g}_i\| \le \sum_{i\ne j} \|c_i^2\bar{g}_i\| = \sum_{i\ne j} c_i^2$.
    Then adding $c_j^2$ on both sides of the inequality gives $2c_j^2\le\sum_{i=1}^n c_i^2$.

    Sufficiency: If $c_j^2\le 1/2 \sum_{i=1}^n c_i^2$ for all $j\in\{1,\dots, n\}$, it is obvious that there always exists an index $n_0$ ($2\le n_0\le n$) such that
    \begin{align}
        \label{eq_2DTrangleEquation1} c_1^2+\dots+c_{n_0-1}^2 \le \frac{1}{2}\sum_{i=1}^n c_i^2,\\
        \label{eq_2DTrangleEquation2} c_1^2+\dots+c_{n_0-1}^2+c_{n_0}^2 \ge \frac{1}{2}\sum_{i=1}^n c_i^2.
    \end{align}
    When $n_0<n$, denote
    \begin{align}\label{eq_l1l2l3}
        \ell_1 &= c_1^2+\dots+c_{n_0-1}^2,\nonumber\\
        \ell_2 &= c_{n_0}^2, \nonumber\\
        \ell_3 &= c_{n_0+1}^2+\dots+c_n^2.
    \end{align}
    Obviously $\ell_1+\ell_2+\ell_3=\sum_{i=1}^n c_i^2$.
    From \eqref{eq_fundamentalInequalityFor2D}, $c_{n_0}\le1/2\sum_{i=1}^n c_i^2$ and hence $\ell_1+\ell_3\ge \ell_2$.
    From \eqref{eq_2DTrangleEquation1}, $\ell_1\le1/2\sum_{i=1}^n c_i^2$ and hence $\ell_2+\ell_3\ge \ell_1$.
    From \eqref{eq_2DTrangleEquation2}, $\ell_1+\ell_2\ge1/2\sum_{i=1}^n c_i^2$ and hence $\ell_1+\ell_2\ge \ell_3$.
    Therefore, $\ell_1$, $\ell_2$ and $\ell_3$ satisfy the triangle inequality and can form a triangle.
    Choose $\bar{g}_1=\dots=\bar{g}_{n_0-1}$. Then $\sum_{i=1}^{n_0-1} c_i^2\bar{g}_i=\ell_1\bar{g}_1$.
    Choose $\bar{g}_{n_0+1}=\dots=\bar{g}_{n}$. Then $\sum_{i=n_0+1}^{n} c_i^2\bar{g}_i=\ell_3\bar{g}_n$.
    Then \eqref{eq_2DEquivalentProblem} becomes
    \begin{align}\label{eq_triangleEuqation}
        \ell_1\bar{g}_1+\ell_2\bar{g}_{n_0}+\ell_3\bar{g}_n=0.
    \end{align}
    We can choose $\bar{g}_1, \bar{g}_{n_0}$ and $\bar{g}_n$ that align with the three sides of the triangle with side length as $\ell_1$, $\ell_2$ and $\ell_3$, respectively (see Figure~\ref{fig_2DConstruction}). Then \eqref{eq_triangleEuqation} and consequently \eqref{eq_2DEquivalentProblem} can be solved.
    When $n_0=n$, the above proof is still valid. In this case, we have $\ell_3=0$ and $\ell_1=\ell_2$, and \eqref{eq_triangleEuqation} becomes $\bar{g}_1+\bar{g}_{n_0}=0$.
\end{proof}

\begin{figure}
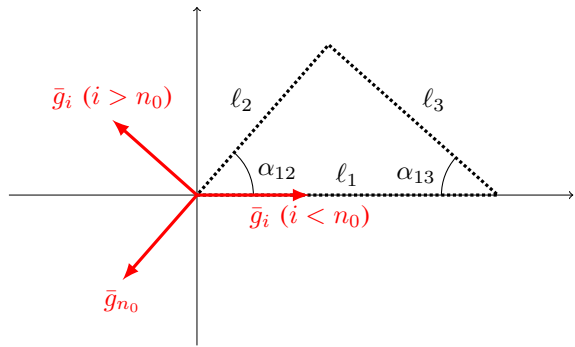

      \centering
      \def\myScale{0.5}
      \include{figures//figure_tikz_2Dconstruction}
      \caption{Geometric illustration for the 2D construction.}
      \label{fig_2DConstruction}
\end{figure}

\begin{algorithm}[t]
\caption{Construction of 2D regular optimal placements $\{g_i\}_{i=1}^n$ with coefficients $\{c_i\}_{i=1}^n$.}\label{corrollary_construction2D}
    \begin{algorithmic}[1] 
        \State Choose $n_0$ satisfying \eqref{eq_2DTrangleEquation1} and \eqref{eq_2DTrangleEquation2}. Then compute $\ell_1$, $\ell_2$ and $\ell_3$ in \eqref{eq_l1l2l3}.
        \State Compute interior angles $\alpha_{12}$ and $\alpha_{13}$ of the triangle with side lengths as $\ell_1$, $\ell_2$ and $\ell_3$ (See Figure \ref{fig_2DConstruction}).
        \State Choose $g_i=[1,0]^T$ for $i\in\{1,\dots,n_0-1\}$, $g_{n_0}=\left[\cos((\pi+\alpha_{12})/2), \sin((\pi+\alpha_{12})/2)\right]^T$, and $g_{i}=\left[\cos((\pi-\alpha_{13})/2), \sin((\pi-\alpha_{13})/2)\right]^T$ for $i\in\{n_0+1,\dots,n\}$.
    \end{algorithmic}
\end{algorithm}

From the proof of Theorem \ref{theorem_2DNecessaryandSufficient}, an explicit construction of 2D regular optimal placements can be summarized in Algorithm \ref{corrollary_construction2D}.
The following example illustrates the construction method in Algorithm \ref{corrollary_construction2D}.

\begin{example}\label{example_2DConstruction}
    In $\mathbb{R}^2$, given six bearing-only sensors with sensor-target ranges respectively as $\|r_1\|=5$, $\|r_2\|=6$, $\|r_3\|=7$, $\|r_4\|=8$, $\|r_5\|=9$, and $\|r_6\|=10$.
    The measurement noise variance is $\sigma_i=1$ for all $i\in\{1,\dots,6\}$. Recall $c_i=1/(\sigma_i\|r_i\|)$ for bearing-only sensors.
    Then $c_1^2=0.0400$, $c_2^2=0.0278$, $c_3^2=0.0204$, $c_4^2=0.0156$, $c_5^2=0.0123$, $c_6^2=0.0100$, and $1/2\sum_{i=1}^6 c_i^2=0.0631$.
    It is easy to check the sequence $\{c_i\}_{i=1}^6$ is regular.
    Because $c_1^2<1/2\sum_{i=1}^6 c_i^2$ and $c_1^2+c_2^2>1/2\sum_{i=1}^6 c_i^2$, choose $n_0=2$.
    Hence $\ell_1=0.0400$, $\ell_2=0.0278$, and $\ell_3=0.0584$. Then $\alpha_{12}=2.0560\,\mathrm{rad}$ and $\alpha_{13}=0.4344\,\mathrm{rad}$. As instructed in Algorithm \ref{corrollary_construction2D}, choose $g_1=[1, 0]^T, g_2=[0.8563, -0.5165]^T, g_3=\dots=g_6=[0.2155, 0.9765]^T$. It can be verified that $\sum_{i=1}^6 c_i^2g_ig_i^T=1/2\sum_{i=1}^6 c_i^2 I_2$.
\end{example}


\subsection{Equally-weighted Optimal Placements}\label{subsection_equallyweightedplacement}
The coefficient $c_i$ actually is the weight of sensor $i$. A placement is called \emph{equally-weighted} if $c_1=\dots=c_n$.
Since $\{c_i\}_{i=1}^n$ is regular with respect to any $d\le n$ if $c_1=\dots=c_n$, equally-weighted placements must be regular.
For bearing-only or RSS sensors, equally-weighted means $\sigma_i=\sigma_j$ and $\|r_i\|=\|r_j\|$ for all $i\ne j$ as $c_i=1/(\sigma_i\|r_i\|)$.
The corresponding geometry is that all sensors are restricted on a 2D circle or a 3D sphere centered at the target.
For range-only sensors, equally-weighted means $\sigma_i=\sigma_j$ for all $i\ne j$ as $c_i=1/\sigma_i$.

Equally-weighted placements are important because they often arise in practice and have some important special properties.
In the equally-weighted case, \eqref{eq_RegularOptimalPlacement} is simplified to $\sum_{i=1}^n g_ig_i^T={n}/{d}I_d$,
which implies that an equally-weighted optimal placement is essentially a unit-norm tight frame \cite{Benedetto03,FickusDissertation01}. 
In $\mathbb{R}^2$, an equally-weighted placement is optimal if $n$ ($n\ge3$) sensors are located at the vertices of an $n$-side regular polygon \cite{Benedetto03,FickusDissertation01,bishop10,dogancy08,Bishop09RSS,Sonia06} as shown in Figure \ref{fig_regularpolygon}.
In $\mathbb{R}^3$, an equally-weighted placement is optimal if $n$ sensors are located at the vertices of a Platonic solid \cite{Benedetto03,FickusDissertation01}. There are only five Platonic solids as shown in Figure \ref{fig_platonicsolids}.
It should be noted that equally-weighted optimal placements are not limited to regular polygons or Platonic solids.
In Section \ref{subsection_distributedconstruction} we will show more examples of equally-weighted optimal placements.

\begin{figure}
  \centering
  \subfloat[]{\includegraphics[width=0.2\linewidth]{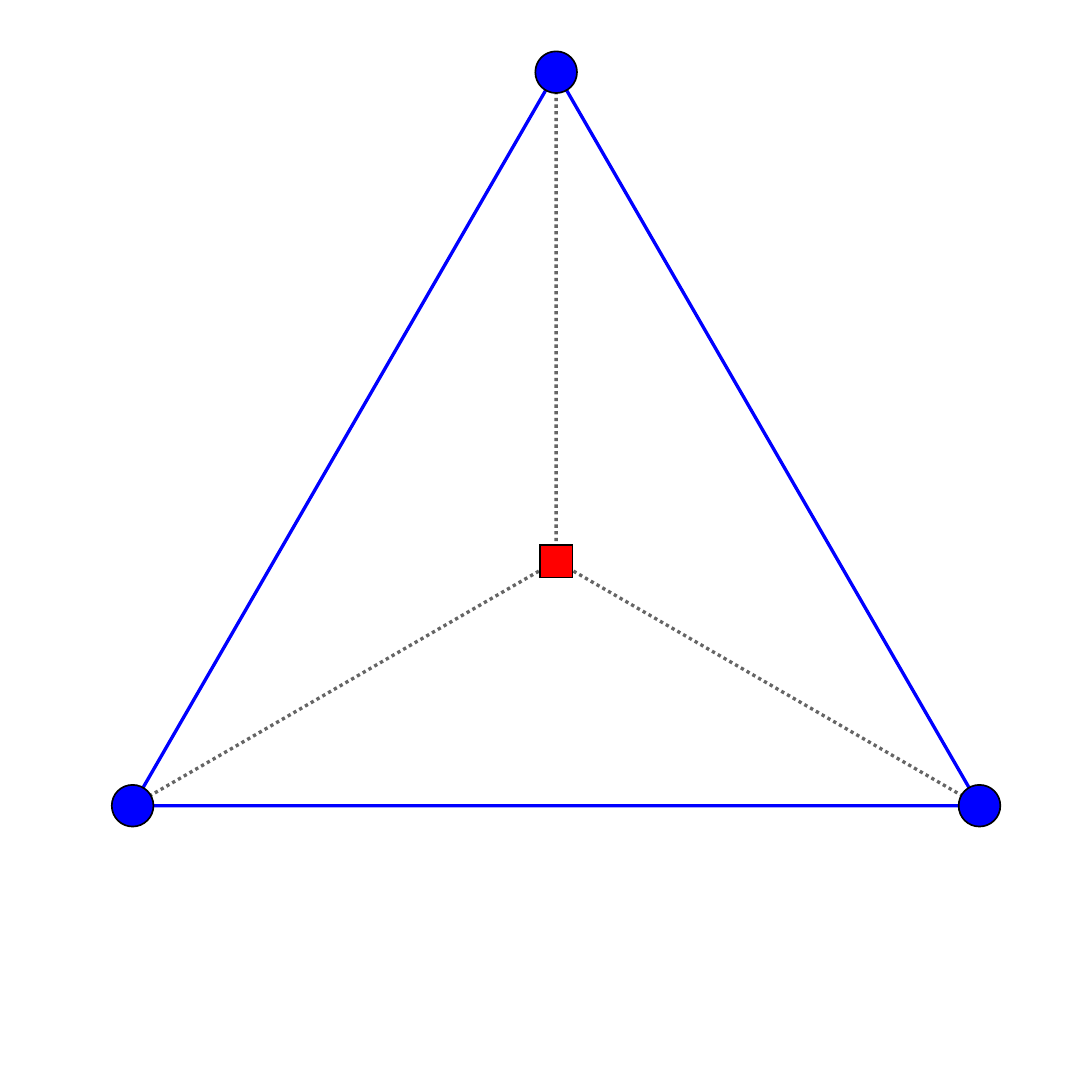}}
  \subfloat[]{\includegraphics[width=0.2\linewidth]{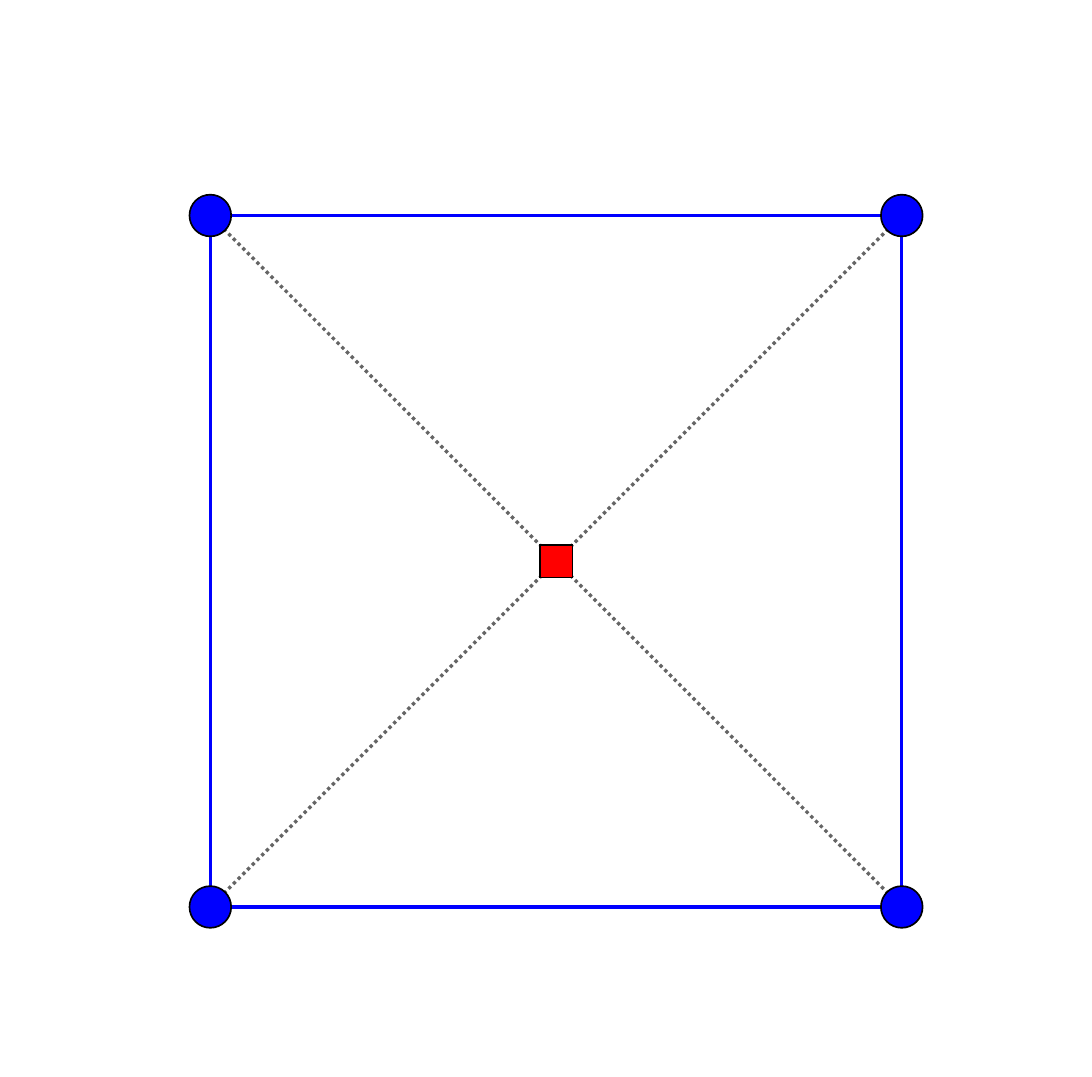}}
  \subfloat[]{\includegraphics[width=0.2\linewidth]{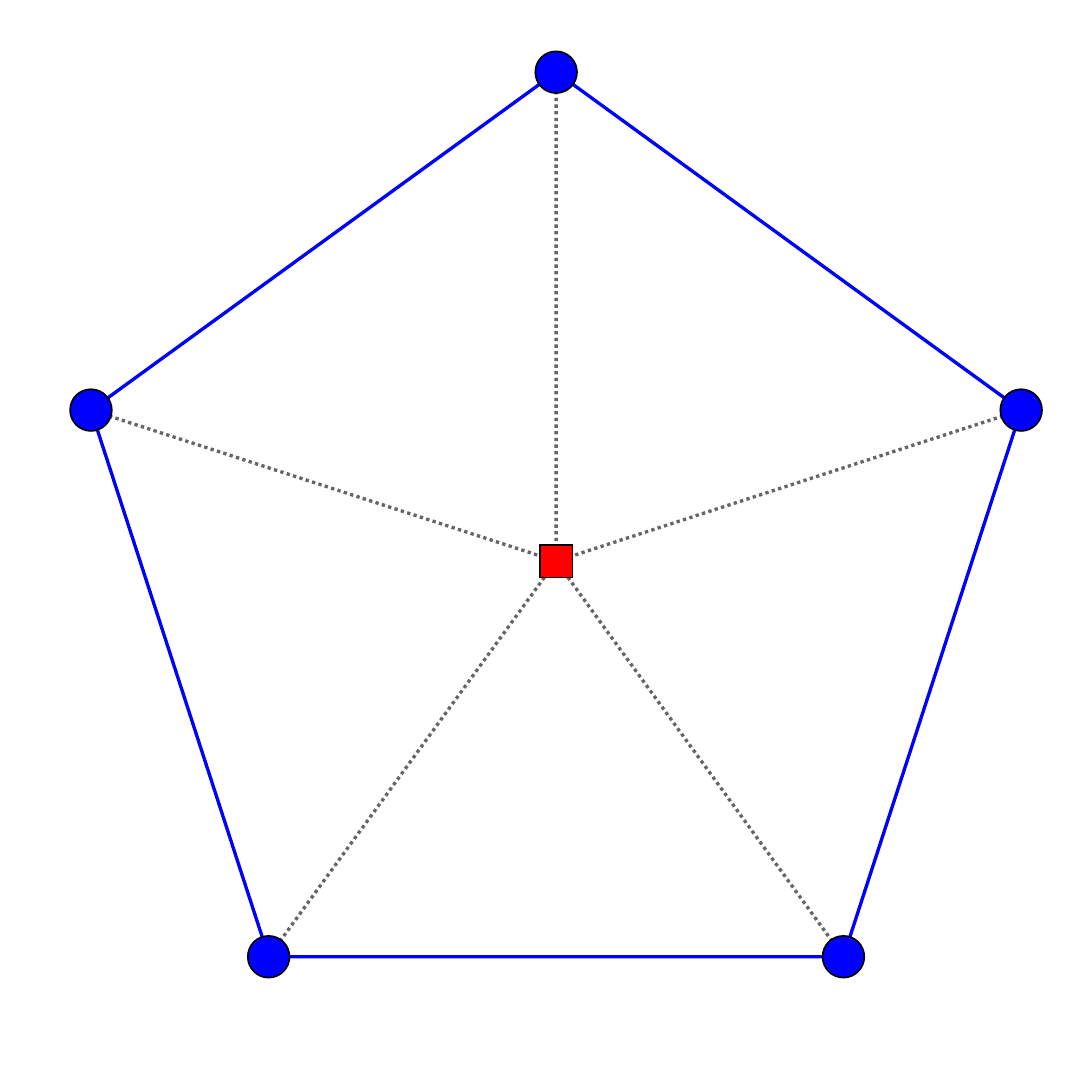}}
  \subfloat[]{\includegraphics[width=0.2\linewidth]{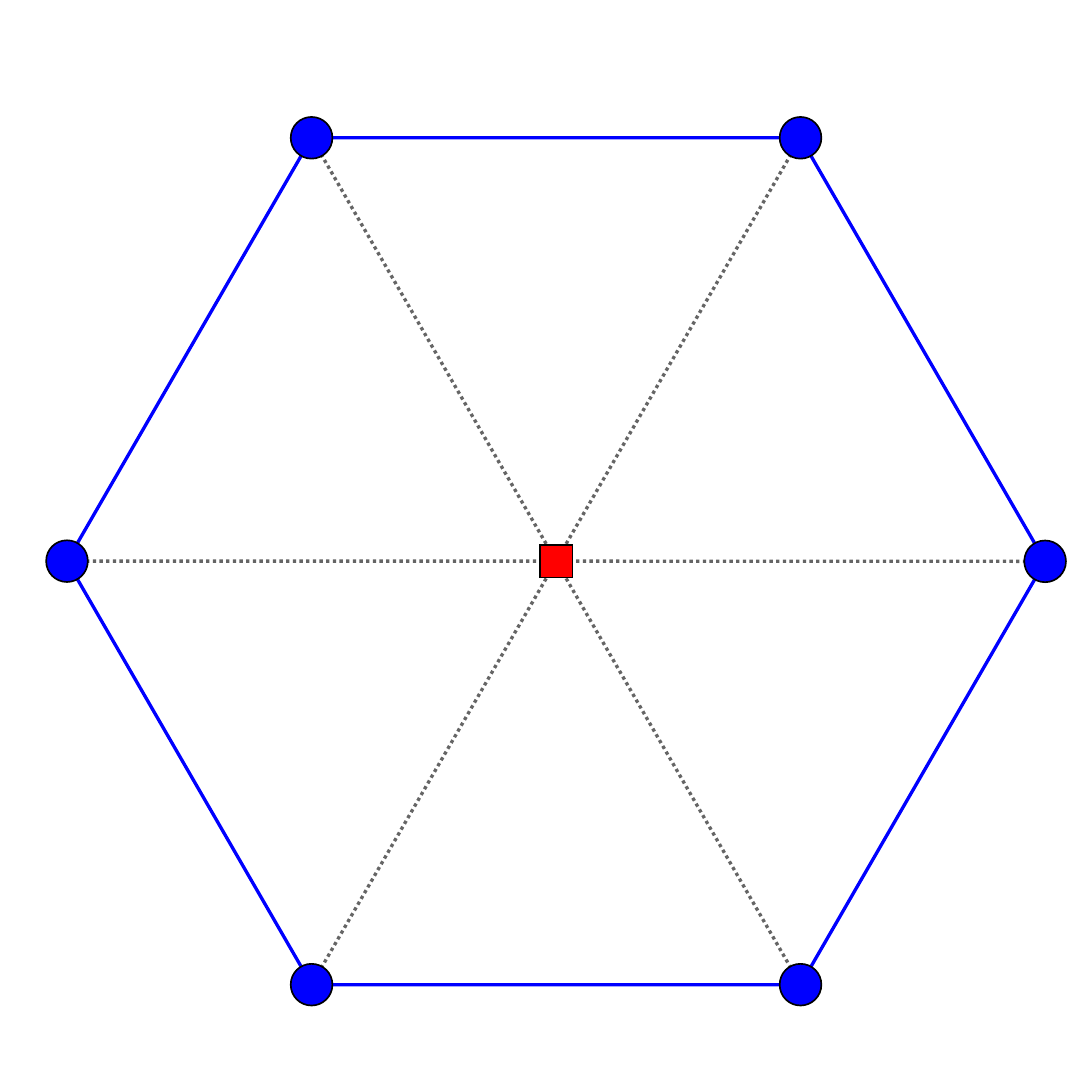}}
  \subfloat[]{\includegraphics[width=0.2\linewidth]{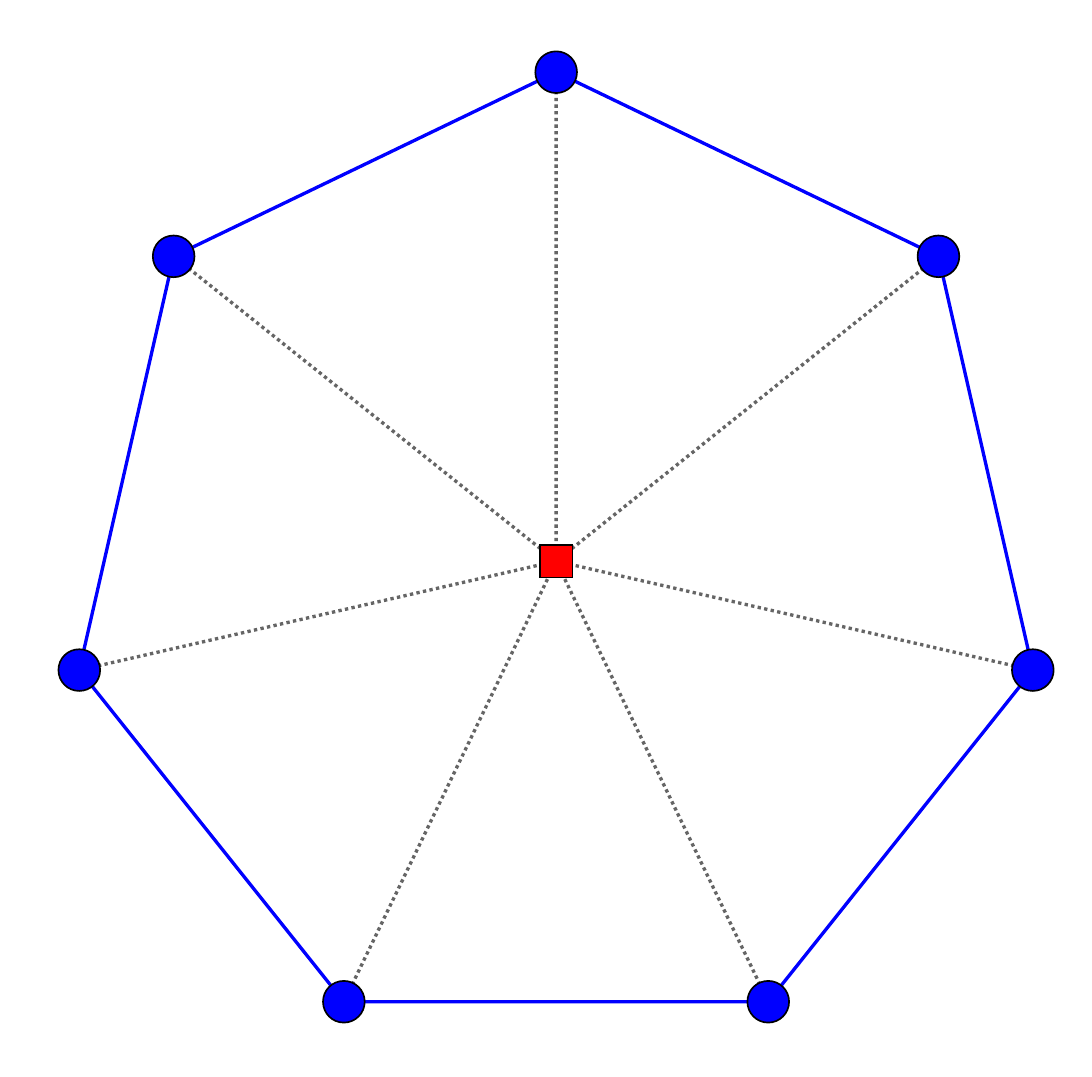}}
  \caption{Examples of 2D equally-weighted optimal placements: regular polygons. Red square: target; blue dots: sensors.}
  \label{fig_regularpolygon}
\end{figure}
\begin{figure}
  \centering
  \subfloat[]{\includegraphics[width=0.2\linewidth]{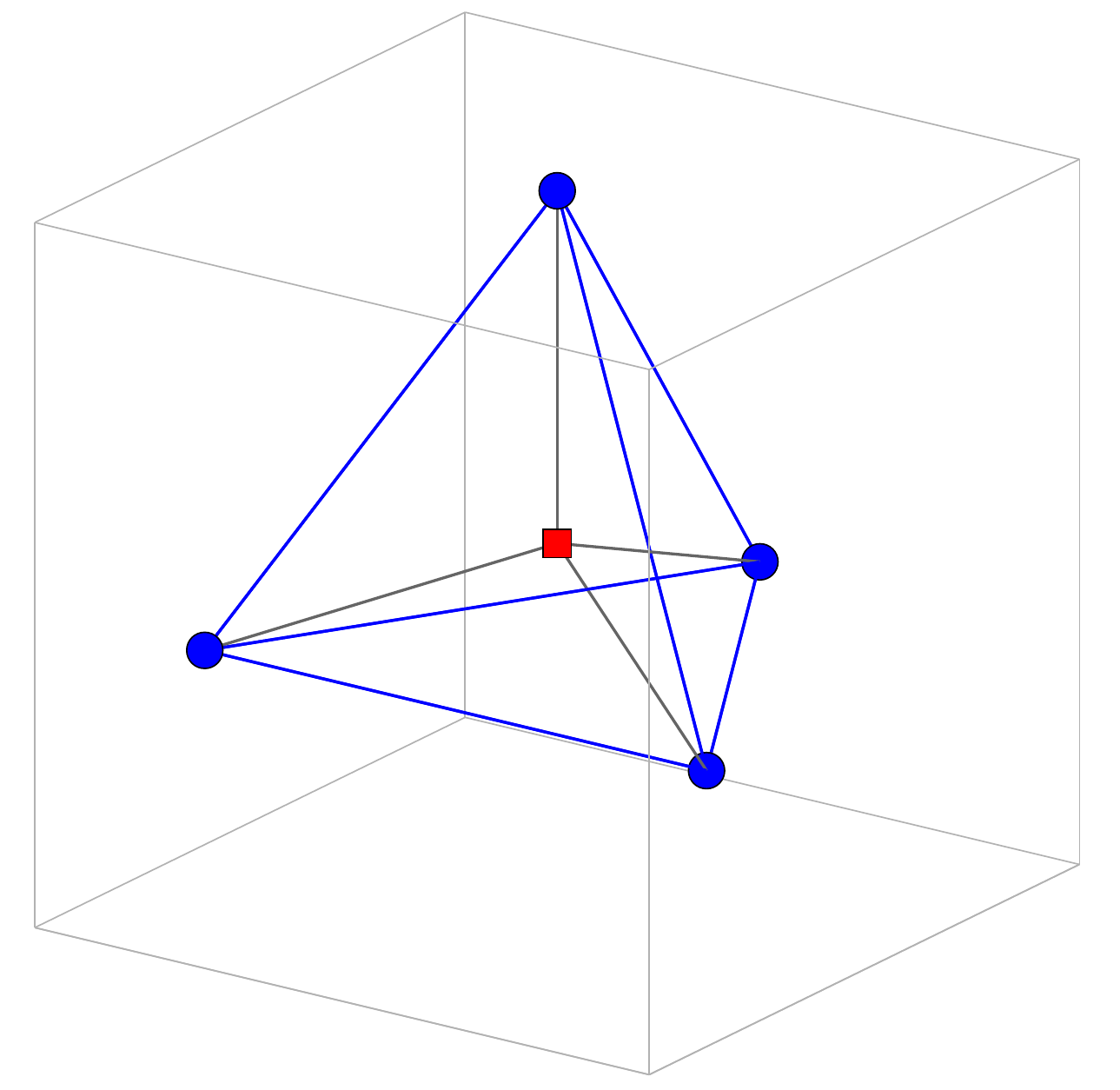}}
  \subfloat[]{\includegraphics[width=0.2\linewidth]{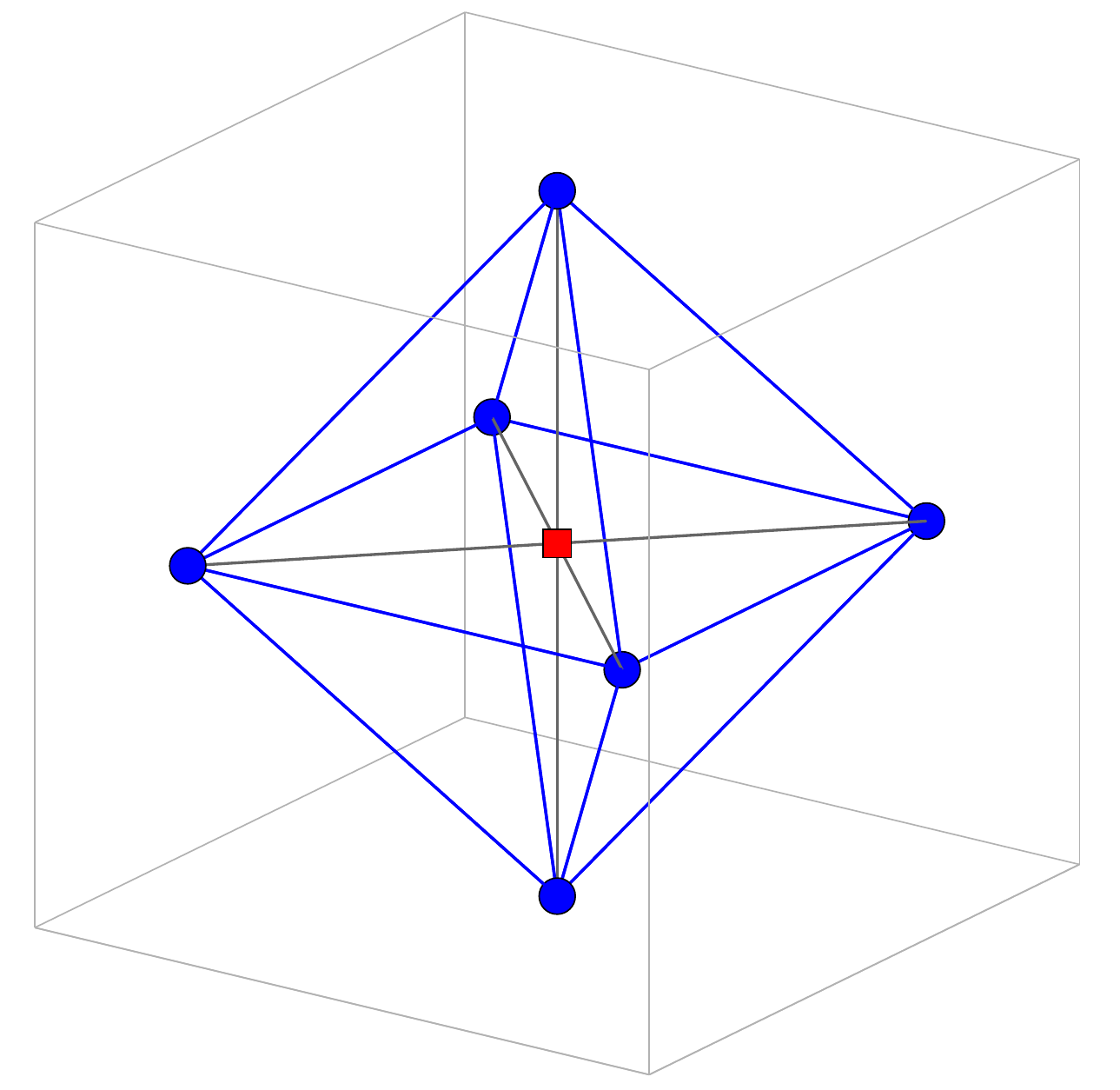}}
  \subfloat[]{\includegraphics[width=0.2\linewidth]{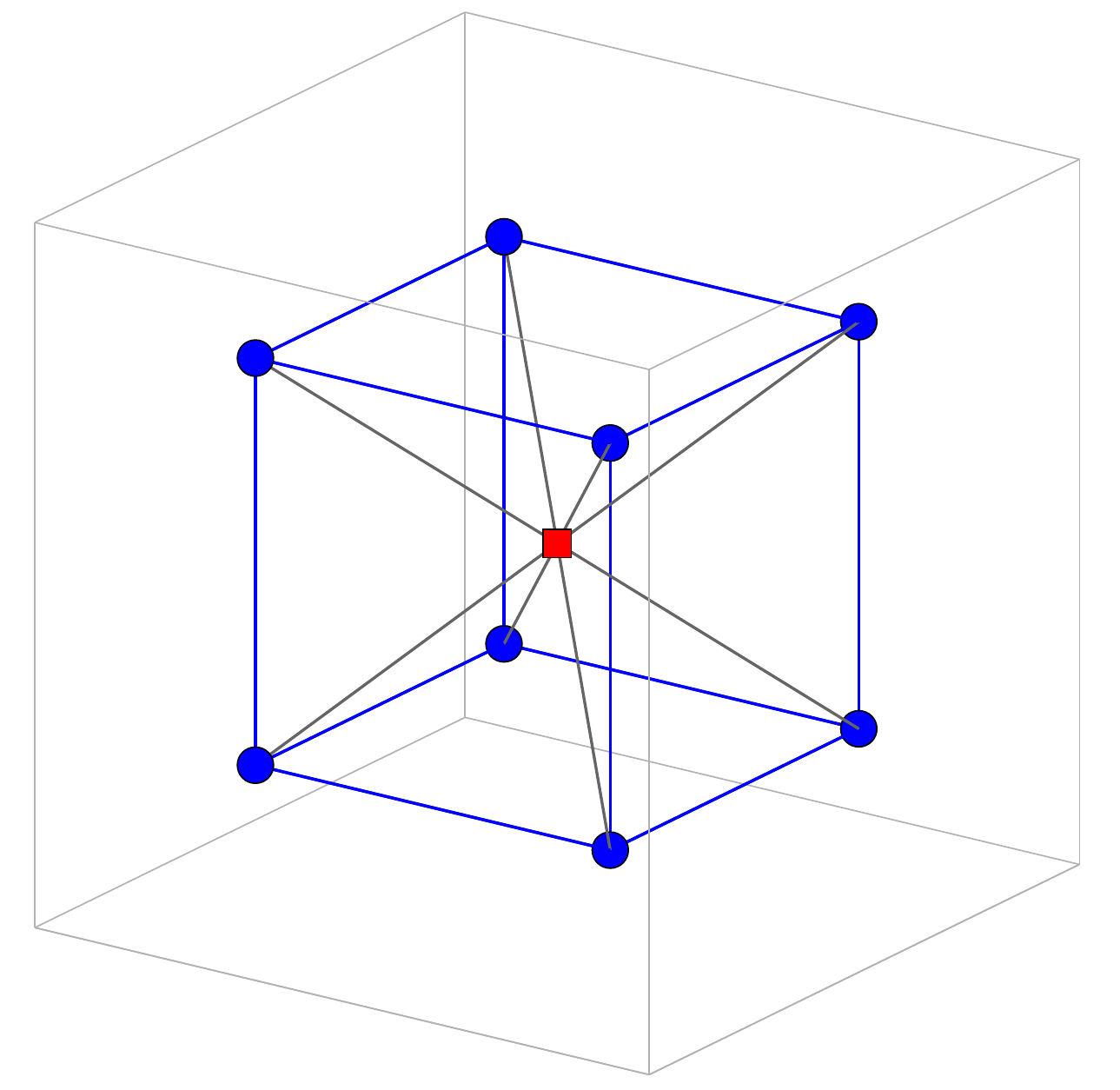}}
  \subfloat[]{\includegraphics[width=0.2\linewidth]{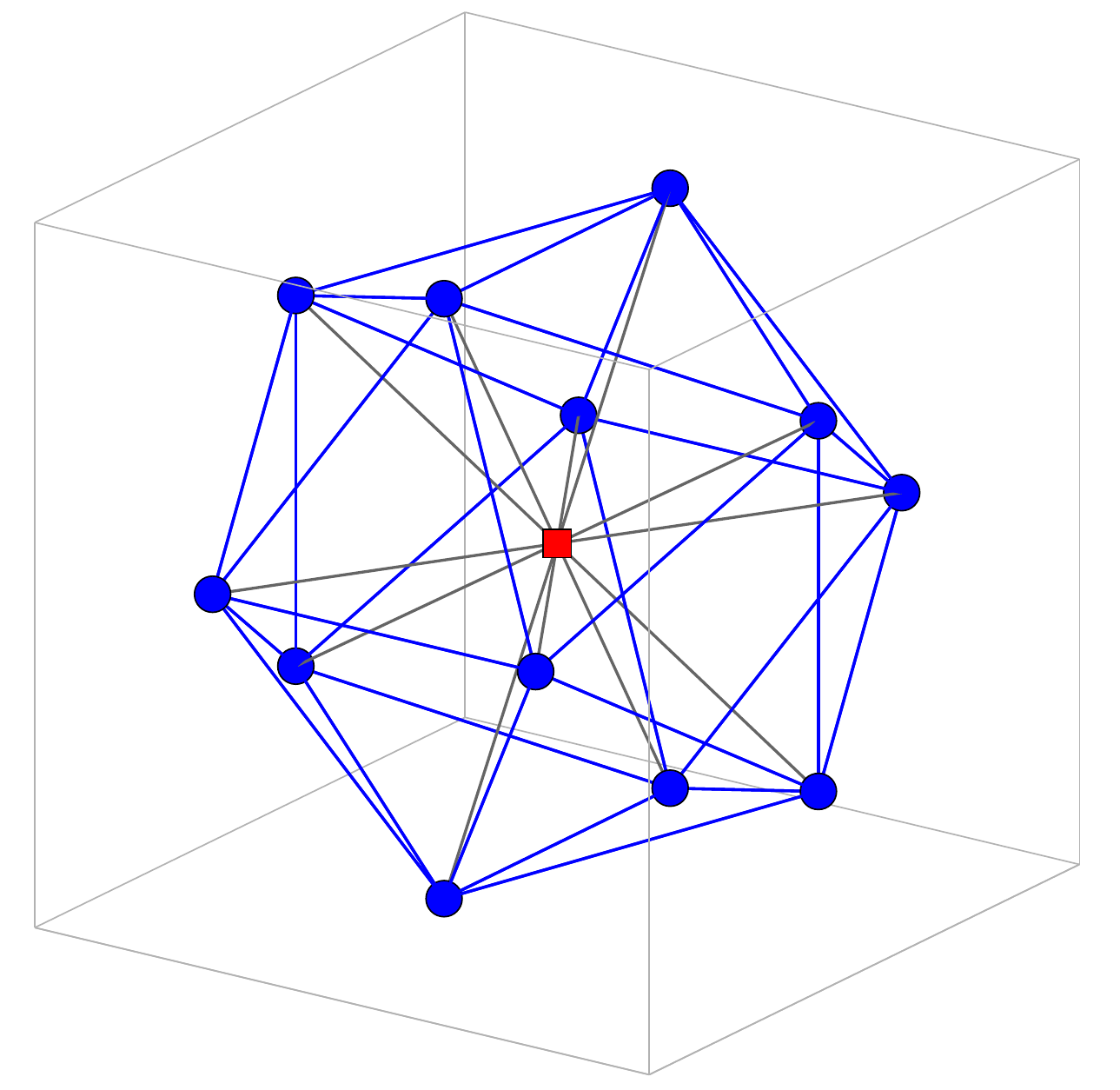}}
  \subfloat[]{\includegraphics[width=0.2\linewidth]{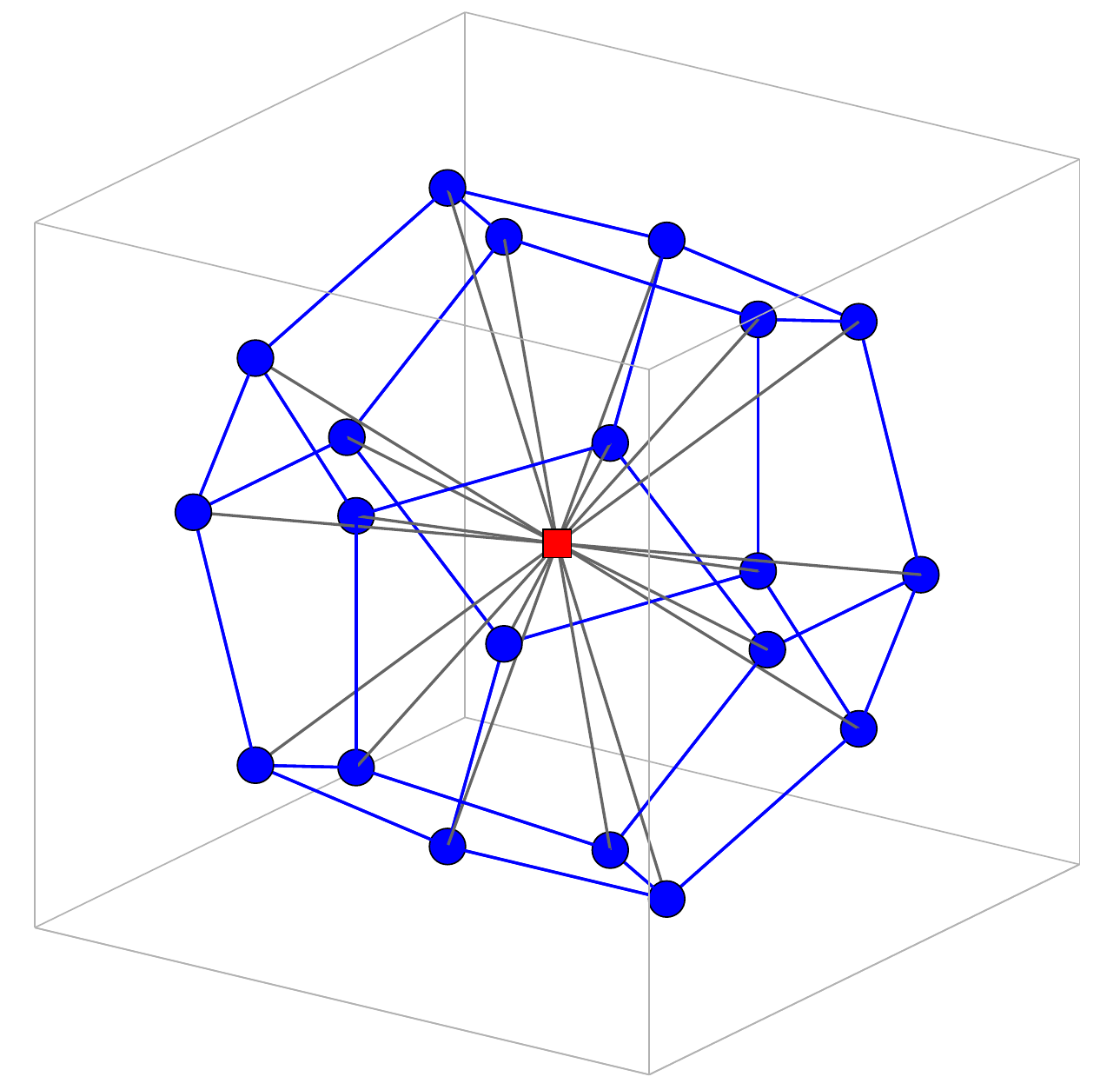}}
  \caption{Examples of 3D equally-weighted optimal placements: Platonic solids. Red square: target; blue dots: sensors. (a) Tetrahedron, $n=4$. (b) Octahedron, $n=6$. (c) Hexahedron, $n=8$. (d) Icosahedron, $n=12$. (e) Dodecahedron, $n=20$.}
  \label{fig_platonicsolids}
\end{figure}

\subsection{Uniqueness}\label{subsection_uniqueness}
Due to placement equivalence, there exist at least an infinite number of equivalent optimal placements minimizing $\|G\|^2$.
It is interesting to ask whether all optimal placements are mutually equivalent.
Now we give the condition under which all optimal placements are mutually equivalent, or in other words, the optimal placement is \emph{unique} up to the equivalence. Again we only consider regular cases.

According to Theorem \ref{theorem_d=n}, it is clear that the optimal placement is unique in the case of $n=d$.
We next show the regular optimal placement is also unique in the case of $n=d+1$ (i.e., three sensors in $\mathbb{R}^2$ or four sensors in $\mathbb{R}^3$).
The uniqueness will be proved by construction, which is inspired by the work in \cite{Goyal01} on unit-norm tight frames.

\begin{theorem}\label{theorem_uniquness}
    In $\mathbb{R}^d$ with $d=2$ or $3$, if $n=d+1$, given a regular coefficient sequence $\{c_i\}_{i=1}^{d+1}$, the regular optimal placement $\{g_i\}_{i=1}^{d+1}$ is unique up to the equivalence in Definition \ref{definition_equivalence}.
\end{theorem}

\begin{proof}
    Suppose $\{g_i\}_{i=1}^{d+1}$ is a regular optimal placement solving \eqref{eq_RegularOptimalPlacement}.
    Denote $\varphi_i=c_ig_i$ and $\Phi=\left[\varphi_1, \dots,\varphi_{d+1}\right]\in\mathbb{R}^{d\times (d+1)}$. Then \eqref{eq_RegularOptimalPlacement} can be written in matrix form as
    $    \Phi\Phi^T={1}/{d}\sum_{i=1}^{d+1} c_i^2 I_d$.
    Hence $\Phi$ has mutually orthogonal rows with row norm as $\sqrt{1/d\sum_{i=1}^{d+1} c_i^2}$.
    Let $x=[x_1,\dots, x_{d+1}]\in\mathbb{R}^{d+1}$ be a vector in the orthogonal complement of the row space of $\Phi$. Assume $\|x\|=\sqrt{1/d\sum_{i=1}^{d+1} c_i^2}$. Adding $x^T$ after the last row of $\Phi$ yields an augmented matrix
    $   \Phi_{\mathrm{aug}}=\left[\Phi^T, x\right]^T \in \mathbb{R}^{(d+1)\times(d+1)}$.
    It is clear that
     $   \Phi_{\mathrm{aug}}\Phi_{\mathrm{aug}}^T=\frac{1}{d}\sum_{i=1}^{d+1} c_i^2 I_{d+1}$.
    Thus $\Phi_{\mathrm{aug}}$ is a scaled orthogonal matrix and its columns are mutually orthogonal.
    The $j$th column of $\Phi_{\mathrm{aug}}$ is $[\varphi_j^T, x_j]^T\in\mathbb{R}^{d+1}$ for all $j\in\{1,\dots,d+1\}$.
    Note the column norm of $\Phi_{\mathrm{aug}}$ is $\sqrt{1/d\sum_{i=1}^{d+1} c_i^2}$.
    Then we have $\|\varphi_j\|^2+x_j^2=1/d\sum_{i=1}^{d+1} c_i^2$ and hence
    \begin{align}\label{eq_hj}
        x_j=\pm\sqrt{\frac{1}{d}\sum_{i=1}^{d+1} c_i^2-c_j^2}.
    \end{align}
    The regularity condition ensures $1/d\sum_{i=1}^{d+1} c_i^2-c_j^2\ge0$.

       By reversing the above proof, we can obtain an explicit construction algorithm for optimal placement with $n=d+1$ as shown in Algorithm \ref{corollary_N=d+1}.
       The rest is to prove the constructed optimal placements are mutually equivalent.
       First, given a vector $x\in\mathbb{R}^{d+1}$ satisfying \eqref{eq_hj}, let $\Phi$ and $\Phi'$ be two different bases of the orthogonal complement of $x$. Due to orthogonality, there exists an orthogonal matrix $U\in\mathbb{R}^{(d+1)\times(d+1)}$ such that
       \begin{align}\label{eq_transformAugmented}
           U
           \left[ \begin{array}{c} \Phi \\ x^T  \end{array} \right]
           =\left[ \begin{array}{c} \Phi' \\ x^T  \end{array} \right].
       \end{align}
       Write $U$ as
       \begin{align}\label{eq_R}
           U=
           \left[ \begin{array}{cc} U_{11} & U_{12} \\ U_{21} & U_{22}  \end{array} \right],
       \end{align}
       where $U_{11}\in\mathbb{R}^{d\times d}, U_{12}\in\mathbb{R}^{d\times 1}, U_{21}\in\mathbb{R}^{1\times d}$, and $U_{22}\in\mathbb{R}$.
       Substituting \eqref{eq_R} into \eqref{eq_transformAugmented} gives $ U_{21}\Phi+(U_{22}-1)x^T=0$.
       Since the rows of $\Phi$ and $x^T$ are linearly independent, we have
       $U_{21}=0$, $U_{22}=1$. Thus $U_{12}=0$ and $U_{11}\Phi=\Phi'$.
       Therefore, the placements described by $\Phi$ and $\Phi'$ are differed only by an orthogonal transformation $U_{11}$. From Definition \ref{definition_equivalence}, the two placements are equivalent.
       Second, let $E\in\mathbb{R}^{(d+1)\times (d+1)}$ be a diagonal matrix with diagonal entries as $1$ or $-1$.
       Given arbitrary $x$ and $x'$ both satisfying \eqref{eq_hj}, there exists an $E$ such that $x'=Ex$. Note $E$ is also an orthogonal matrix. It can be analogously proved that the optimal placements would be differed by an orthogonal transformation and a number of flipping of sensors about the target. From Definition \ref{definition_equivalence}, these placements are also equivalent.
\end{proof}

From the proof of Theorem \ref{theorem_uniquness}, an explicit construction of the unique regular optimal placement in the case of $n=d+1$ can be summarized in Algorithm \ref{corollary_N=d+1}.
The following example illustrates the construction method in Algorithm \ref{corollary_N=d+1}.

\begin{algorithm}[t]
\caption{Construction of the unique regular optimal placement
$\{g_i\}_{i=1}^{d+1}$ with coefficients $\{c_i\}_{i=1}^{d+1}$.}\label{corollary_N=d+1}
    \begin{algorithmic}[1] 
        \State Choose $x=[x_1,\dots, x_{d+1}]\in\mathbb{R}^{d+1}$ with $x_j=\pm\sqrt{1/d\sum_{i=1}^{d+1} c_i^2-c_j^2}$ for $i\in\{1,\dots,d+1\}$.
        \State Use the singular value decomposition (SVD) to numerically compute an orthogonal basis of the orthogonal complement of $x$. Let $x=U\Sigma V^T$ be an SVD of $x$, where $U\in\mathbb{R}^{(d+1)\times(d+1)}$ is an orthogonal matrix.
        \State Let $u_i$ denote the $i$th column of $U$. Then $x=\pm\sqrt{1/d\sum_{i=1}^{d+1} c_i^2} u_1$, and $\Phi$ can be constructed as
        \begin{align}\label{eq_Phi_N=d+1}
            \Phi=\sqrt{\frac{1}{d}\sum_{i=1}^{d+1} c_i^2} \left[u_2,\dots,u_{d+1}\right]^T \in\mathbb{R}^{d\times (d+1)}.
        \end{align}
      \State Compute $g_i=\varphi_i/c_i$ for $i\in\{1,\dots,d+1\}$.
    \end{algorithmic}
\end{algorithm}

\begin{example}
    In $\mathbb{R}^3$, given four bearing-only sensors with sensor-target ranges respectively as $\|r_1\|=20$, $\|r_2\|=21$, $\|r_3\|=22$ and $\|r_4\|=23$.
    The measurement noise variance of the $i$th sensor is $\sigma_i=0.01$ with $i\in\{1,\dots,4\}$. Recall $c_i=1/(\sigma_i\|r_i\|)$ for bearing-only sensors.
    Then $c_1^2=25.00$, $c_2^2=22.68$, $c_3^2=20.66$, $c_4^2=18.90$ and $1/3\sum_{i=1}^4 c_i^2=29.08$.
    The sequence $\{c_i\}_{i=1}^4$ is regular.
    From \eqref{eq_hj}, choose $x=[2.02, 2.53, 2.90, 3.19]^T$. Compute the SVD of $x$ and use \eqref{eq_Phi_N=d+1} to compute $\Phi$ as
    \begin{align*}
        \Phi=
        \left[
          \begin{array}{cccc}
                -2.5307  &  4.5286  & -0.9906  & -1.0891\\
               -2.9016 &  -0.9906 &   4.2568 &  -1.2487\\
               -3.1901  & -1.0891  & -1.2487  &  4.0197\\
          \end{array}
        \right].
    \end{align*}
    It can be verified $\sum_{i=1}^4 c_i^2g_ig_i^T=\Phi\Phi^T=1/3\sum_{i=1}^4 c_i^2 I_3$.
\end{example}

Figure \ref{fig_uniqueplacement_2D} and Figure~\ref{fig_uniqueplacement_3D} show examples of unique optimal placements.
From the last subsection, we know a regular triangle and a regular tetrahedron are equally-weighted optimal.
By Theorem \ref{theorem_uniquness} they are also unique up to equivalence.
Hence the two equivalent placements in Figure \ref{fig_uniqueplacement_2D} consist of all possible forms of the equally-weight optimal placements with $n=3$ in $\mathbb{R}^2$.
The three equivalent placements in Figure \ref{fig_uniqueplacement_3D} consist of all possible forms of the equally-weight optimal placements with $n=4$ in $\mathbb{R}^3$.

When $n>d+1$, the regular optimal placement may not be unique.
In the next subsection, we will give examples to show the optimal placement may not be unique when $n\ge4$ in $\mathbb{R}^2$ or $n\ge6$ in $\mathbb{R}^3$.
Now a question remains: whether the regular optimal placement with $n=5$ in $\mathbb{R}^3$ is unique up to the equivalence.
The answer is negative.
The following is an explanation and an explicit construction of the regular optimal placement with $n=5$ in $\mathbb{R}^3$.

\begin{figure}
  \centering
  \subfloat[]{\label{}\includegraphics[width=0.23\linewidth]{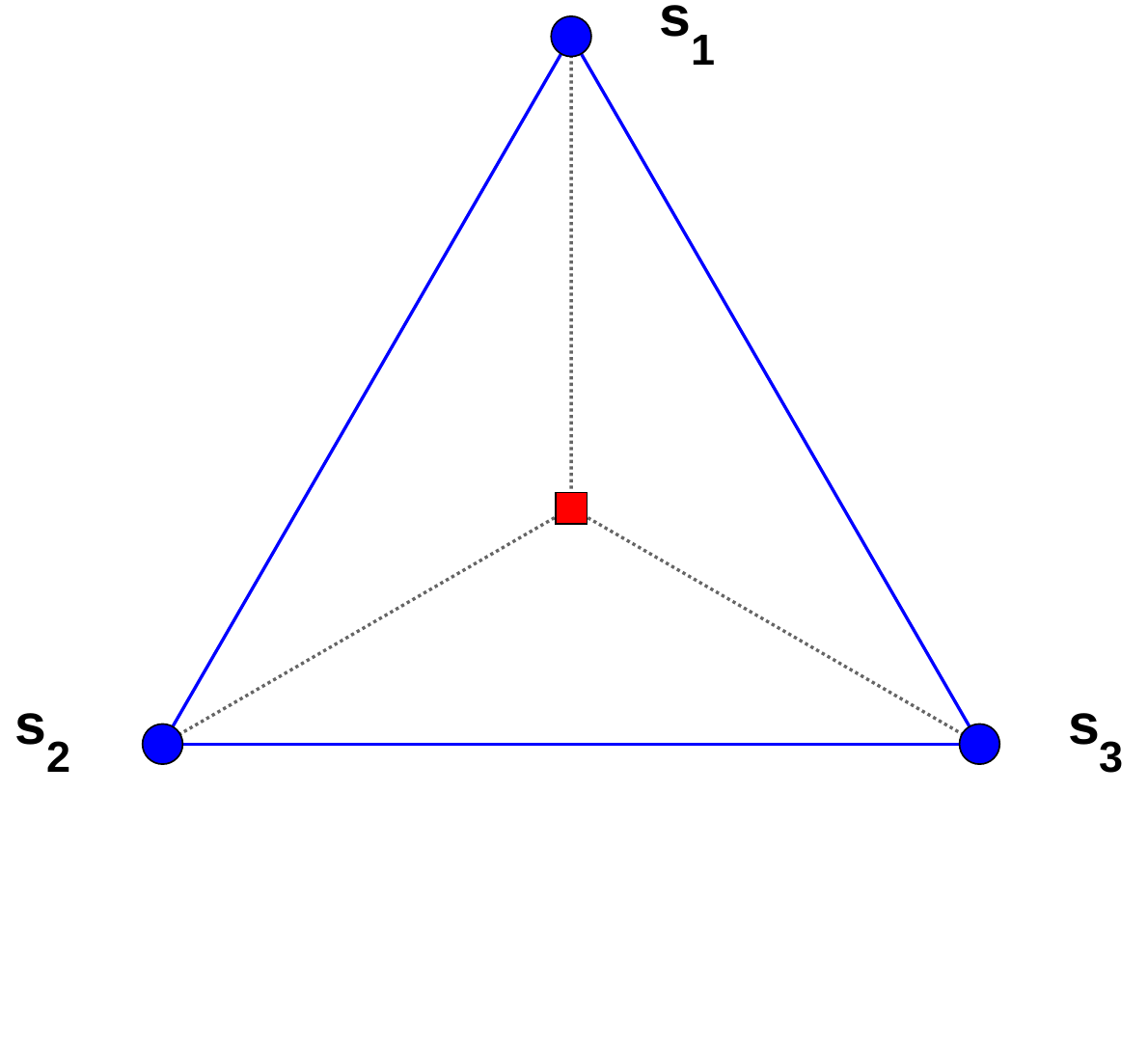}}
  \ \ \
  \subfloat[]{\label{}\includegraphics[width=0.23\linewidth]{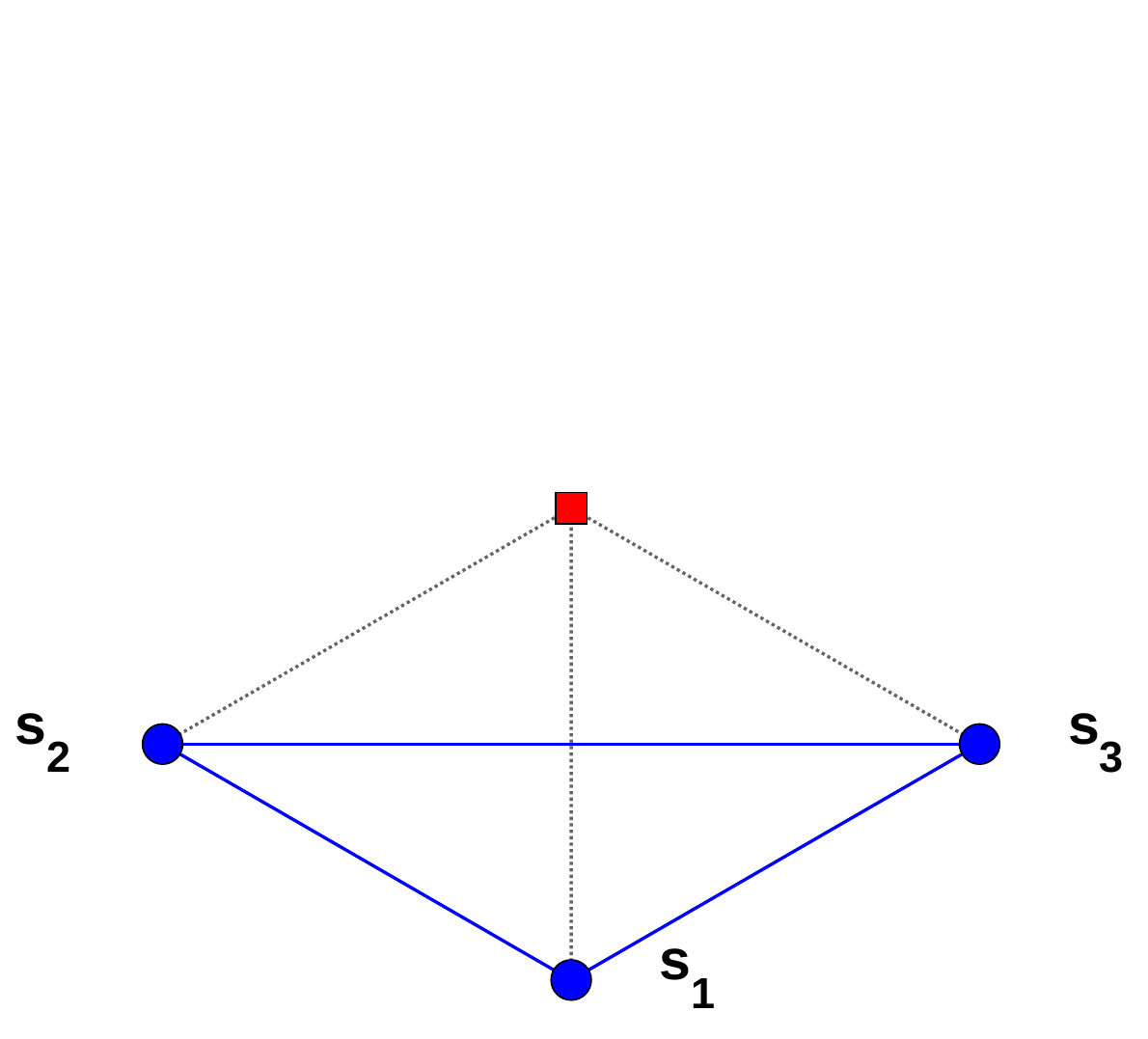}}
  \caption{The unique equally-weighted optimal placements with $n=3$ in $\mathbb{R}^2$. Red square: target; blue dots: sensors.
  (a) Regular triangle. (b) Flip $s_1$ about the target.}
  \label{fig_uniqueplacement_2D}
\end{figure}
\begin{figure}
  \centering
  \subfloat[]{\label{}\includegraphics[width=0.2\linewidth]{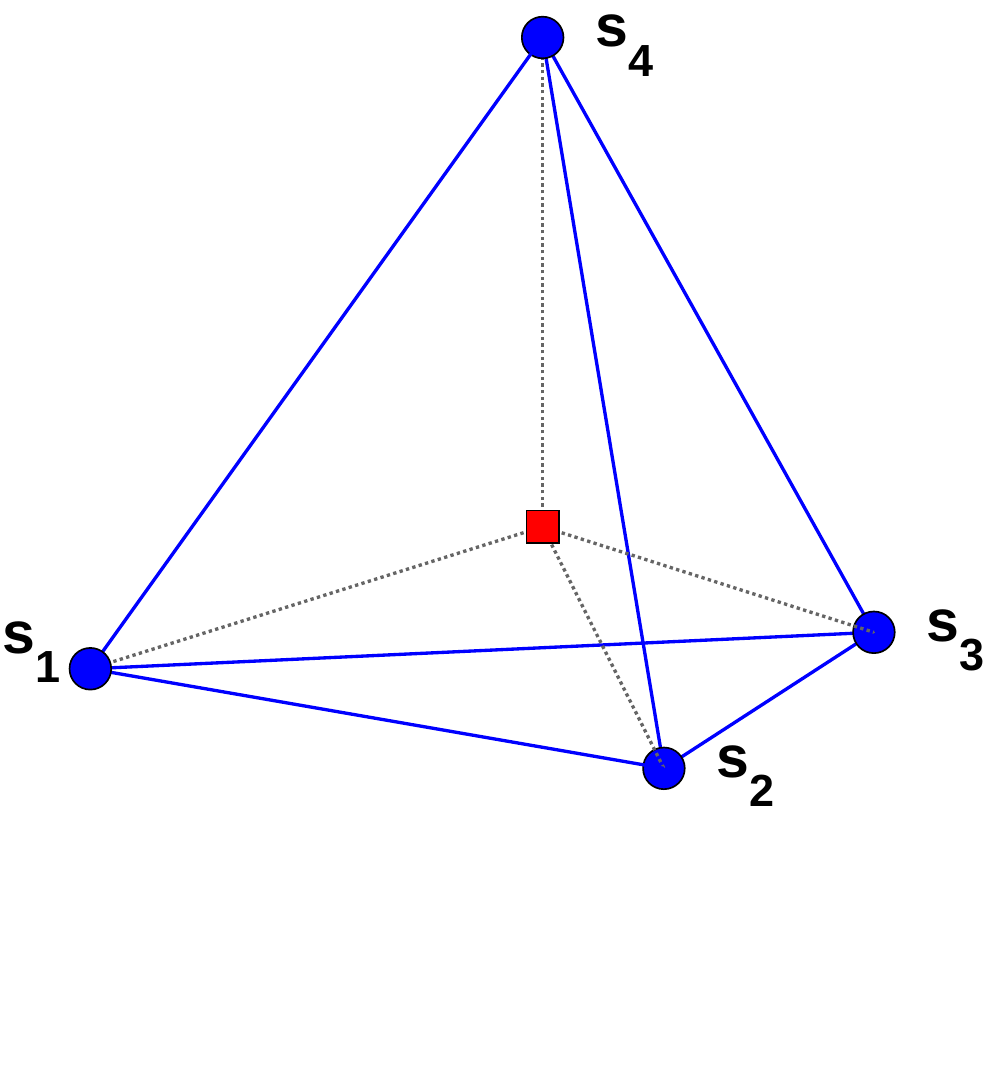}}
  \ \ \
  \subfloat[]{\label{}\includegraphics[width=0.2\linewidth]{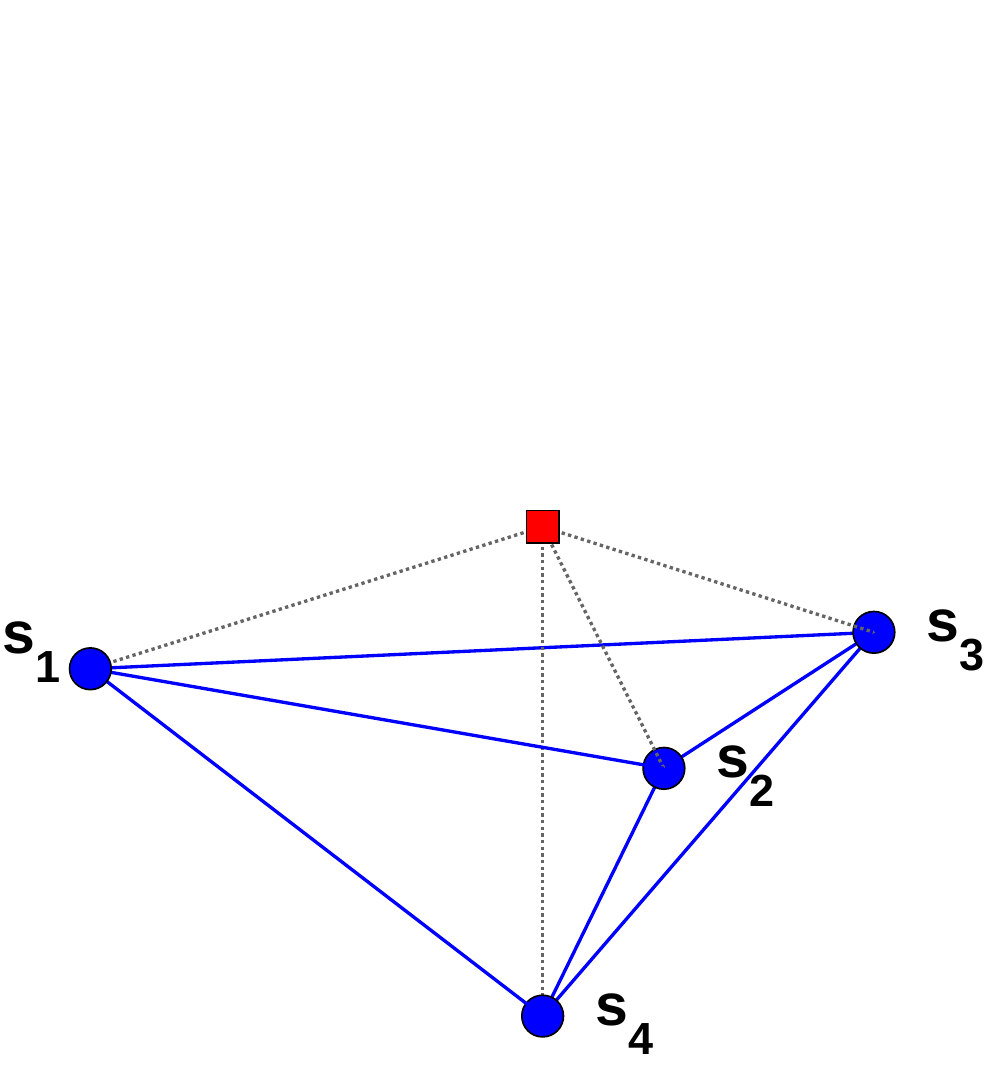}}
  \ \ \
  \subfloat[]{\label{}\includegraphics[width=0.2\linewidth]{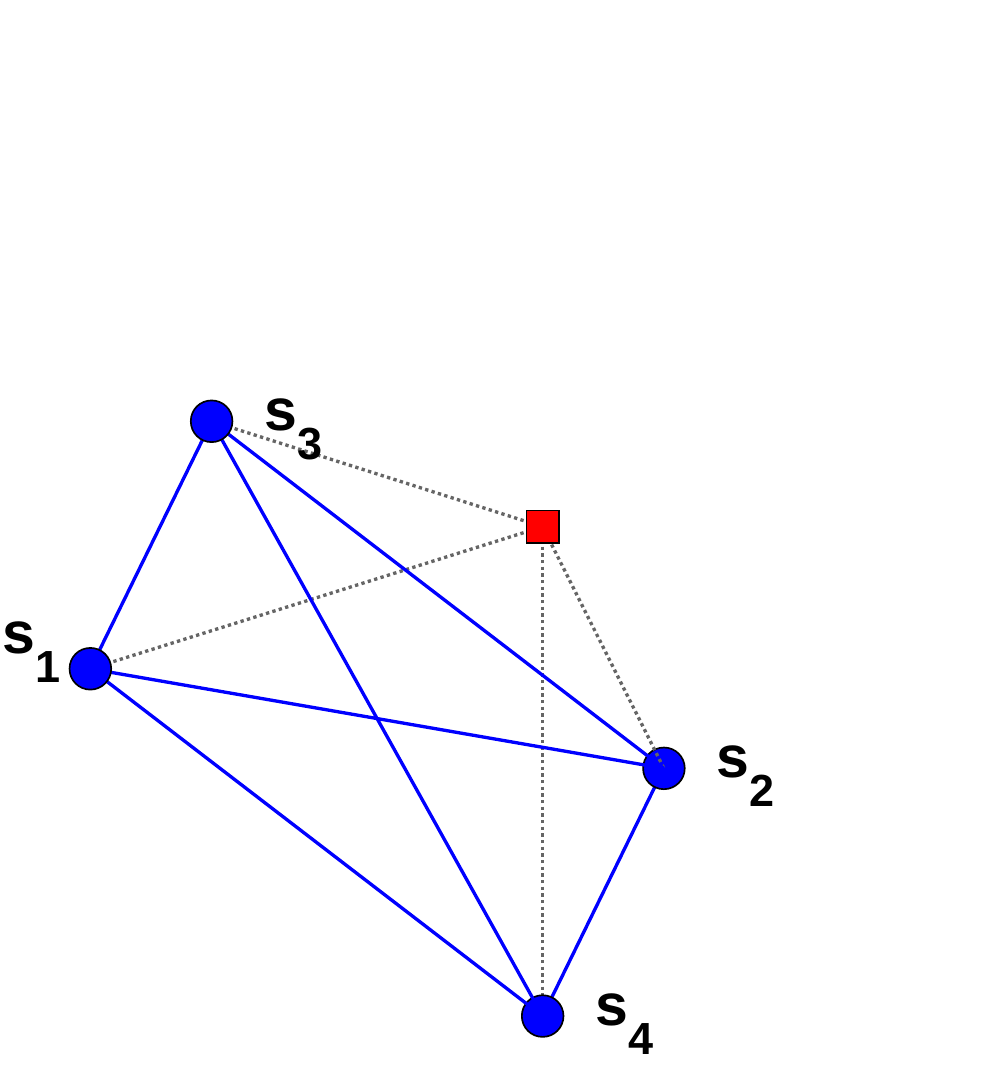}}
  \ \ \
  \caption{The unique equally-weighted optimal placements with $n=4$ in $\mathbb{R}^3$.
  Red square: target; blue dots: sensors.
  (a) Regular tetrahedron. (b) Flip $s_4$ about the target. (c) Flip $s_4$ and $s_3$ about the target.}
  \label{fig_uniqueplacement_3D}
\end{figure}

Suppose the sequence $\{c_i\}_{i=1}^5$ is regular with respect to $d=3$. Denote $\varphi_i=c_ig_i$ and $\Phi=[\varphi_1,\dots,\varphi_5]\in\mathbb{R}^{3\times5}$. Then \eqref{eq_RegularOptimalPlacement} becomes
$\Phi\Phi^T=1/3\sum_{i=1}^5c_i^2 I_3$.
There always exists $\Phi'=[\varphi_1',\dots,\varphi_5']\in\mathbb{R}^{2\times5}$ in the orthogonal complement of the row space of $\Phi$ such that
\begin{align*}
    \left[
      \begin{array}{c}
        \Phi \\
        \Phi' \\
      \end{array}
    \right]
    \left[
      \begin{array}{cc}
        \Phi^T & \Phi'^T \\
      \end{array}
    \right]
    =\frac{1}{3}\sum_{i=1}^5c_i^2 I_5,
\end{align*}
which implies $\|\varphi_j\|^2+\|\varphi_j'\|^2=1/3\sum_{i=1}^5c_i^2$ and $\Phi'\Phi'^T=1/3\sum_{i=1}^5c_i^2 I_2$.
Thus $\{\varphi_j'\}_{j=1}^5$ represents a 2D regular optimal placement with $\|\varphi_j'\|=\sqrt{1/3\sum_{i=1}^5c_i^2-c_j^2}$ for all $j\in\{1,\dots,5\}$ (it can be verified $\{\|\varphi_j'\|\}_{j=1}^5$ is regular with respect to $d=2$).
Therefore, to obtain $\Phi$, we can first construct $\{\varphi_j'\}_{j=1}^5$ using Algorithm \ref{corrollary_construction2D} for example, and then find $\Phi$ in the orthogonal complement of the row space of $\Phi'$.
Since $\{\varphi_i'\}_{i=1}^5$ may have non-equivalent solutions, $\{\varphi_i\}_{i=1}^5$ would not be unique up to the equivalence.

\subsection{Distributed Construction}\label{subsection_distributedconstruction}
When there are a large number of sensors, it might be inconvenient to design the optimal placement involving all sensors. The following property can be applied to construct large-scale optimal placements in a distributed manner. The 2D versions of the following result have been proposed in \cite{bishop10,dogancy08,Bishop09RSS}.

\begin{theorem}\label{theorem_distributedConstruction}
    The union of multiple disjoint regular optimal placements in $\mathbb{R}^d$ ($d=2$ or $3$) is still a regular optimal placement in $\mathbb{R}^d$.
\end{theorem}
\begin{proof}
    In $\mathbb{R}^d$, consider multiple disjoint regular optimal placements: $\{c_i, g_i\}_{i\in\mathcal{I}_k}$ with $\mathcal{I}_k$ as the index set of the $k$th placement ($k=1,\dots,q$).
    The term disjoint as used here means that different placements share no common sensors.
    Define $\mathcal{I}=\bigcup_{k=1}^q \mathcal{I}_k$.
    If $|\cdot|$ denotes the cardinality of a set, then $|\mathcal{I}|=\sum_{k=1}^q |\mathcal{I}_k|$.

    For the $k$th placement, since $\{c_i, g_i\}_{i\in\mathcal{I}_k}$ is regular optimal in $\mathbb{R}^d$, from Theorem \ref{theorem_regularOptimalPlacement} we have
    \begin{align*}
        \sum_{i\in\mathcal{I}_k} c_i^2 g_i g_i^T=\frac{1}{d}\sum_{i\in\mathcal{I}_k} c_i^2 I_d.
    \end{align*}
    For the union placement, we have
    \begin{align*}
        \sum_{j\in\mathcal{I}} c_j^2 g_j g_j^T
        &= \sum_{k=1}^q \sum_{i\in\mathcal{I}_k} c_i^2 g_i g_i^T \nonumber \\
        &= \frac{1}{d} \sum_{k=1}^q \sum_{i\in\mathcal{I}_k} c_i^2 I_d \nonumber \\
        &= \frac{1}{d} \sum_{j\in\mathcal{I}} c_j^2 I_d.
    \end{align*}
    By Theorem \ref{theorem_regularOptimalPlacement}, the union placement is regular optimal in $\mathbb{R}^d$.
\end{proof}

Theorem \ref{theorem_distributedConstruction} implies that a large-scale regular optimal placement can be constructed in a distributed manner:
firstly divide the large-scale placement into a number of disjoint regular sub-placements, secondly construct each regular optimal sub-placement, and finally combine these optimal sub-placements together to obtain a large regular optimal placement.
We call this kind of method \emph{distributed construction}.
Because the combination of the optimal sub-placements can be arbitrary, distributed construction will lead to an infinite number of optimal placements for the large system.
These optimal placements have the same FIM and $\|G\|^2$, but they are generally \emph{non-equivalent}.
From Theorem \ref{theorem_distributedConstruction}, we also know only regular placements can be possibly divided into some regular subsets.

The distributed construction method is very suitable for constructing equally-weighted optimal placements.
That is because any equally-weighted sequence $\{c_i\}_{i=1}^{k}$ is regular if $k \ge d$, and thus $\{c_i\}_{i=1}^n$ can be easily divided into a number of regular subsets.
We next present two instances.
(i)~For any integer $n\ge4$, it is obvious that there exist nonnegative integers $m_1$ and $m_2$ such that $n$ can be decomposed as $n=2m_1+3m_2$.
Thus in $\mathbb{R}^2$ we can always distributedly construct an equally-weighted optimal placement with $n\ge4$ by using the ones with $n=2$ or $3$.
(ii) For any integer $n\ge6$, there exist nonnegative integers $m_1$, $m_2$ and $m_3$ such that $n=3m_1+4m_2+5m_3$.
Thus in $\mathbb{R}^3$ we can always distributedly construct an equally-weighted optimal placement with $n\ge6$ by using the ones with $n=3$, $4$ or $5$.
As a consequence, noticing distributed construction yields an infinite number of non-equivalent optimal placements, equally-weighted placements with $n\ge4$ in $\mathbb{R}^2$ or $n\ge6$ in $\mathbb{R}^3$ always have an infinite number of non-equivalent optimal solutions.

\begin{figure}
  \centering
  \subfloat[]{\label{}\includegraphics[width=0.2\linewidth]{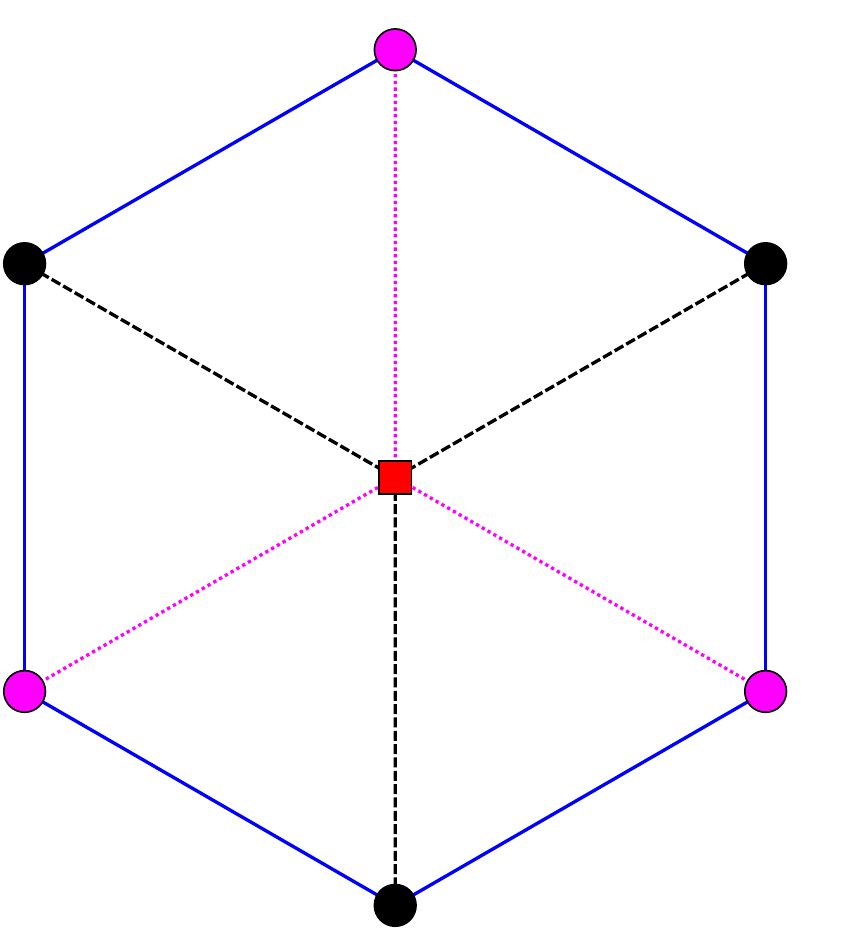}}
  \subfloat[]{\label{}\includegraphics[width=0.2\linewidth]{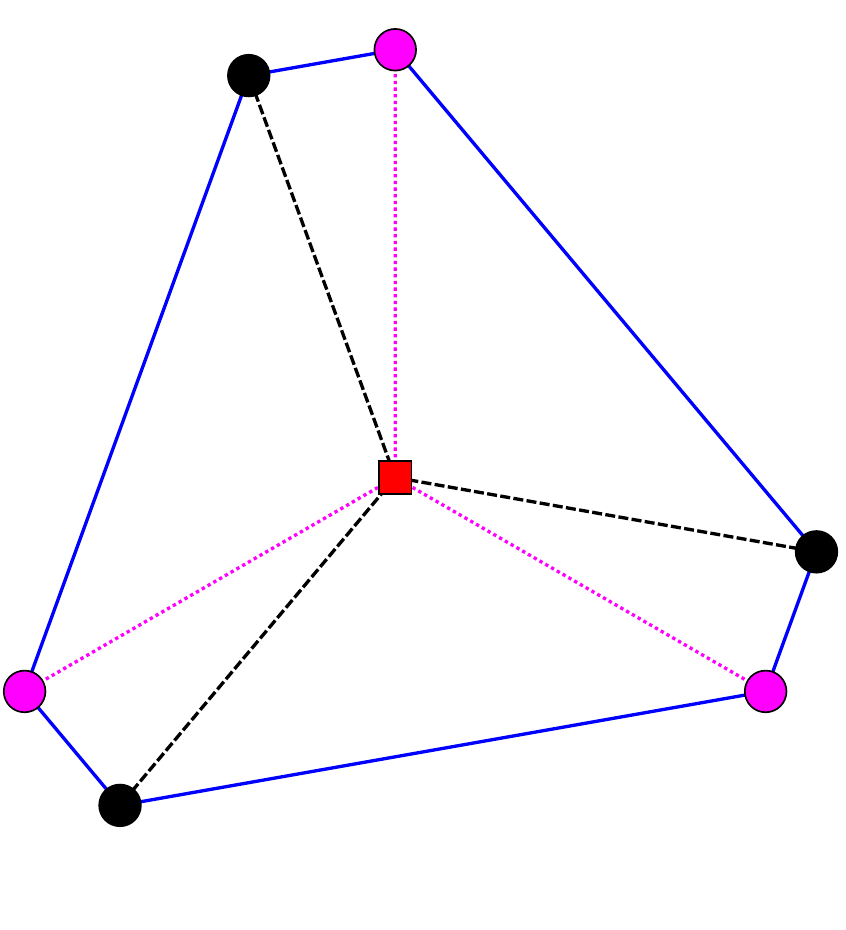}} \ \
  \subfloat[]{\label{}\includegraphics[width=0.2\linewidth]{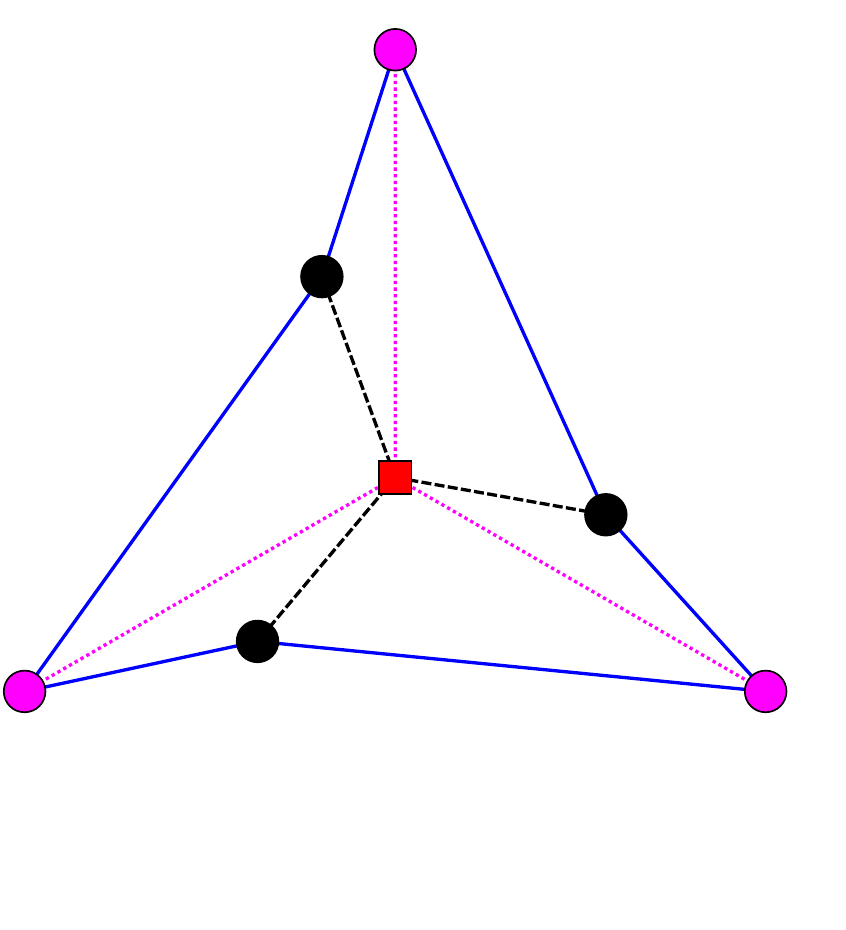}}\\
  \subfloat[]{\includegraphics[width=0.2\linewidth]{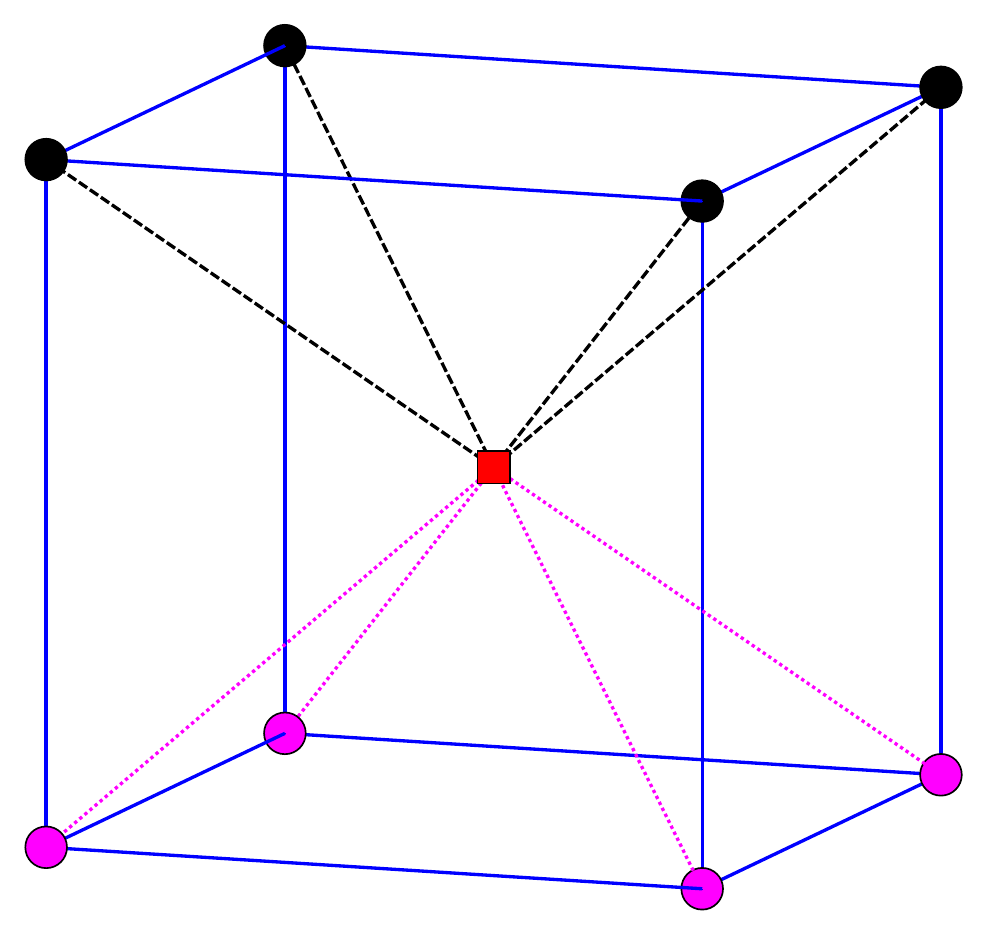}}
  \ \
  \subfloat[]{\includegraphics[width=0.2\linewidth]{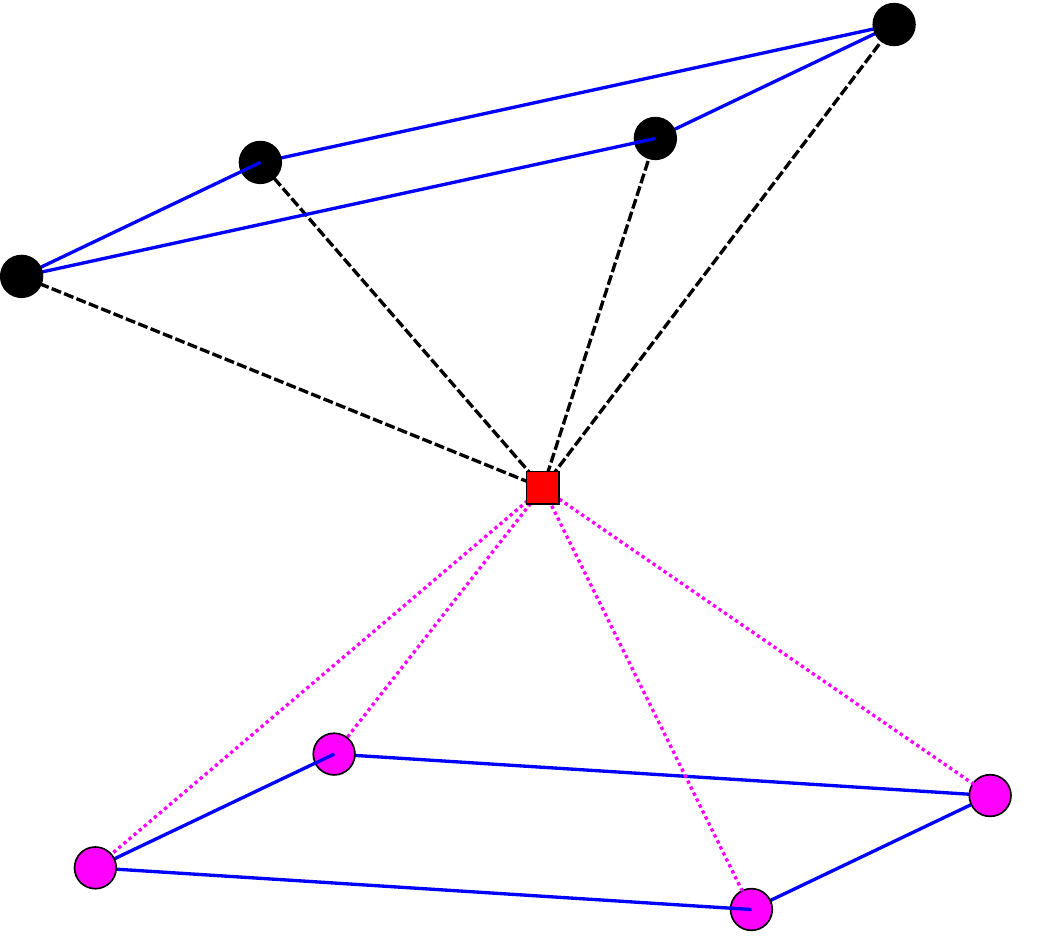}}
  \caption{Examples of optimal placement by distributed construction. Red square: target; dots: sensors.}
  \label{fig_distributed_2D}
\end{figure}

We show some typical examples in Figure \ref{fig_distributed_2D} to illustrate the distributed construction.
From Section \ref{subsection_equallyweightedplacement}, we know a regular triangle or regular tetrahedron placement is equally-weighted optimal.
The placement shown in Figure~\ref{fig_distributed_2D} (a), (b) or (c) is distributedly constructed by two regular optimal placements with $n=3$ (the sensors with the same color form a triangle optimal placement).
The placement shown in Figure \ref{fig_distributed_2D} (d) or (e) is distributedly constructed by the optimal placement with  $n=4$ as shown in Figure~\ref{fig_uniqueplacement_3D} (c), which is equivalent to the regular tetrahedron.
Thus by Theorem \ref{theorem_distributedConstruction} all placements in Figure \ref{fig_distributed_2D} are regular optimal.

\section{Numerical Verification}\label{section_control}
In order to verify our previous analytical analysis, in this section we solve the parameter optimization problem \eqref{eq_problemdefinition2} from a numerical perspective and then present some numerical simulations.
More specifically, we employ Lyapunov approaches to design a centralized gradient control law which can numerically minimize the objective function $\|G\|^2$ given an appropriate initial point.
The control law can be applied to numerically construct generic regular and irregular optimal placements in 2D and 3D.

Assume the motion model of sensor $i$ to be
$ \dot{s}_i=u_i$,
where $u_i\in\mathbb{R}^d$ is the control input.
Then we have $\dot{r}_i=u_i$ because $r_i=s_i-p$ and the target position estimation $p$ is given.
Let $r=[r_1^T,\dots,r_n^T]^T\in\mathbb{R}^{dn}$.
Denote $\beta$ as the constant lower bound of $\|G\|^2$ as shown in Theorem \ref{theorem_regularOptimalPlacement} and \ref{theorem_irregularOptimalPlacement}.
Then the optimal placement set is $\mathcal{E}_0=\{r\in\mathbb{R}^{dn}: \|G\|^2-\beta=0\}$.
Choose the Lyapunov function as $V(r)= 1/4(\|G\|^2-\beta)$. Clearly $V$ is positive definite with respect to $\mathcal{E}_0$.
Denote $\partial V / \partial r_i$ as the Jacobian of $V$ with respect to $r_i$. Then we have
\begin{align*}
    \dot{V}
    = \sum_{i=1}^n \frac{\partial V}{\partial r_i}\dot{r}_i
    = \sum_{i=1}^n \frac{c_i^2}{\|r_i\|}g_i^T G P_i \dot{r}_i,
\end{align*}
where $P_i=I_d-g_ig_i^T$ is an orthogonal projection matrix satisfying $P_i^T=P_i$, $P_i^2=P_i$, and $\mathrm{Null}(P_i)=\mathrm{span}\{g_i\}$. $\mathrm{Null}(\cdot)$ denotes the null space of a matrix.
Design the gradient control law as
\begin{align}\label{eq_closedLoopSystem}
    \dot{r}_i=-P_iGg_i,
\end{align}
such that
\begin{align*}
    \dot{V}= -\sum_{i=1}^n \frac{c_i^2}{\|r_i\|}\|P_i G g_i\|^2 \le 0
\end{align*}
and $\dot{V}=0$ when $P_i G g_i=0$ for all $i\in\{1,\dots,n\}$.

\begin{figure}
  \centering
  \includegraphics[width=0.2\linewidth]{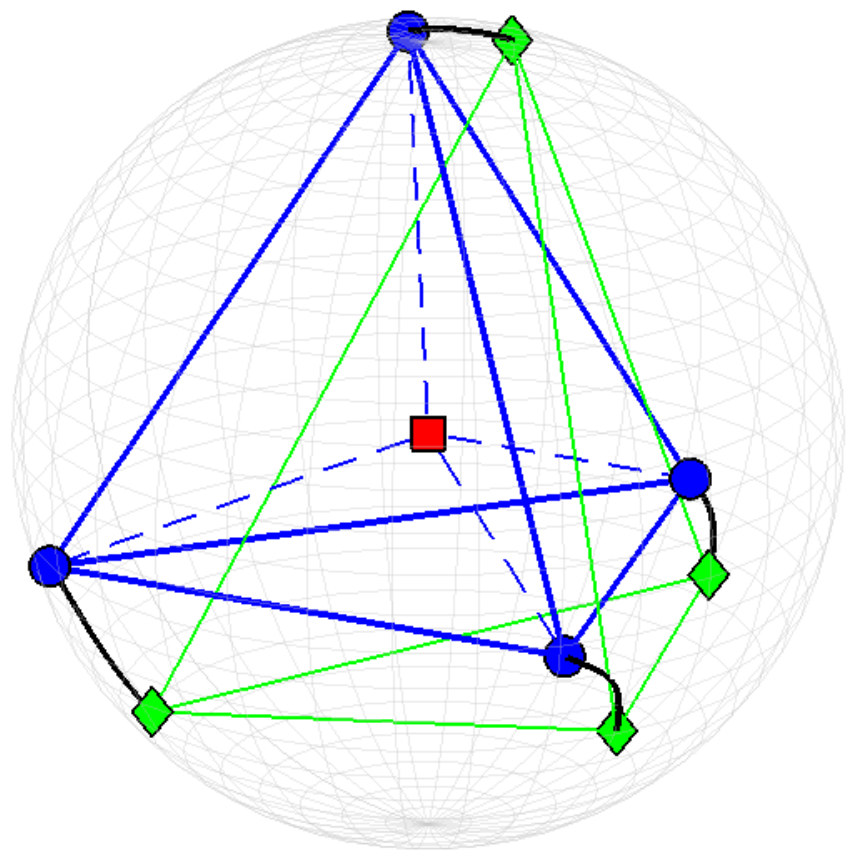}
  \includegraphics[width=0.2\linewidth]{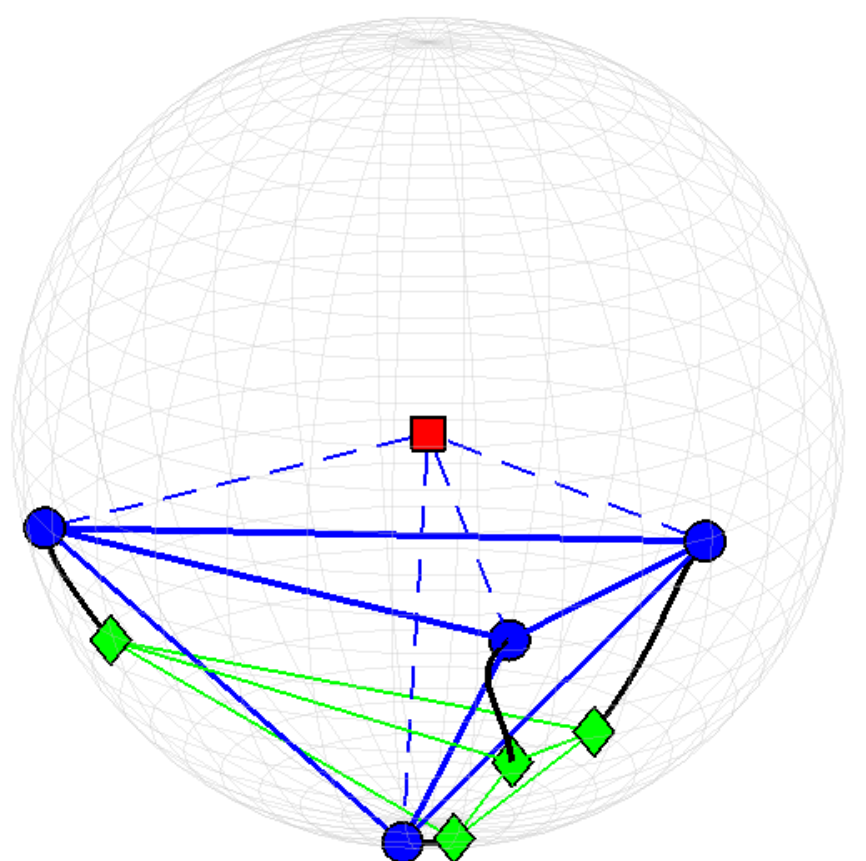}
  \includegraphics[width=0.2\linewidth]{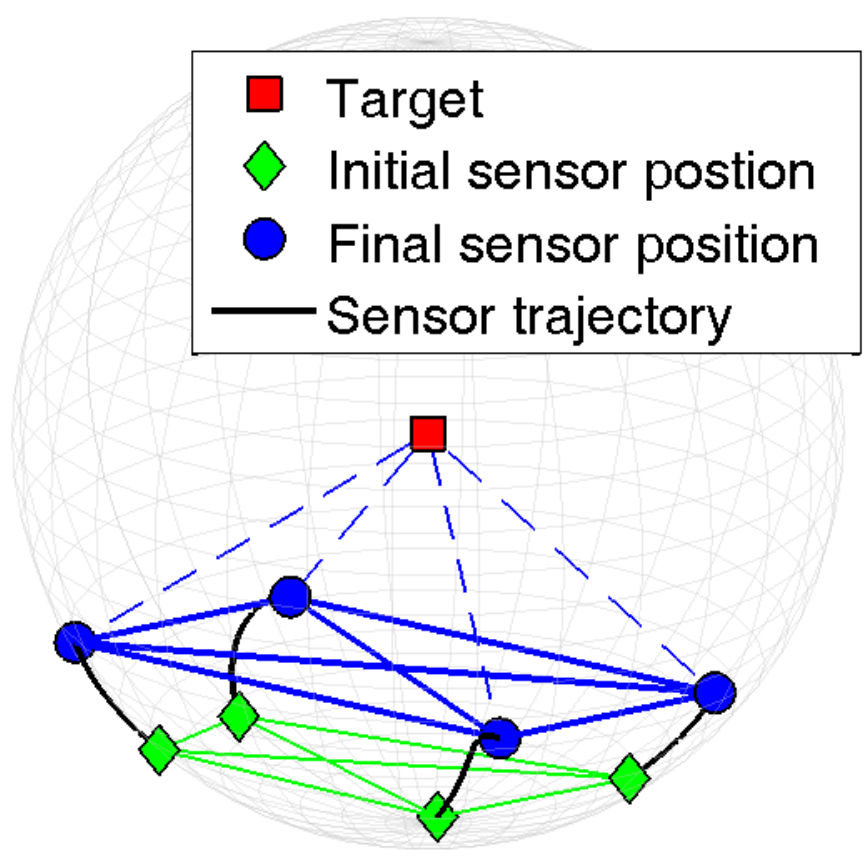}\\
  \includegraphics[width=0.2\linewidth]{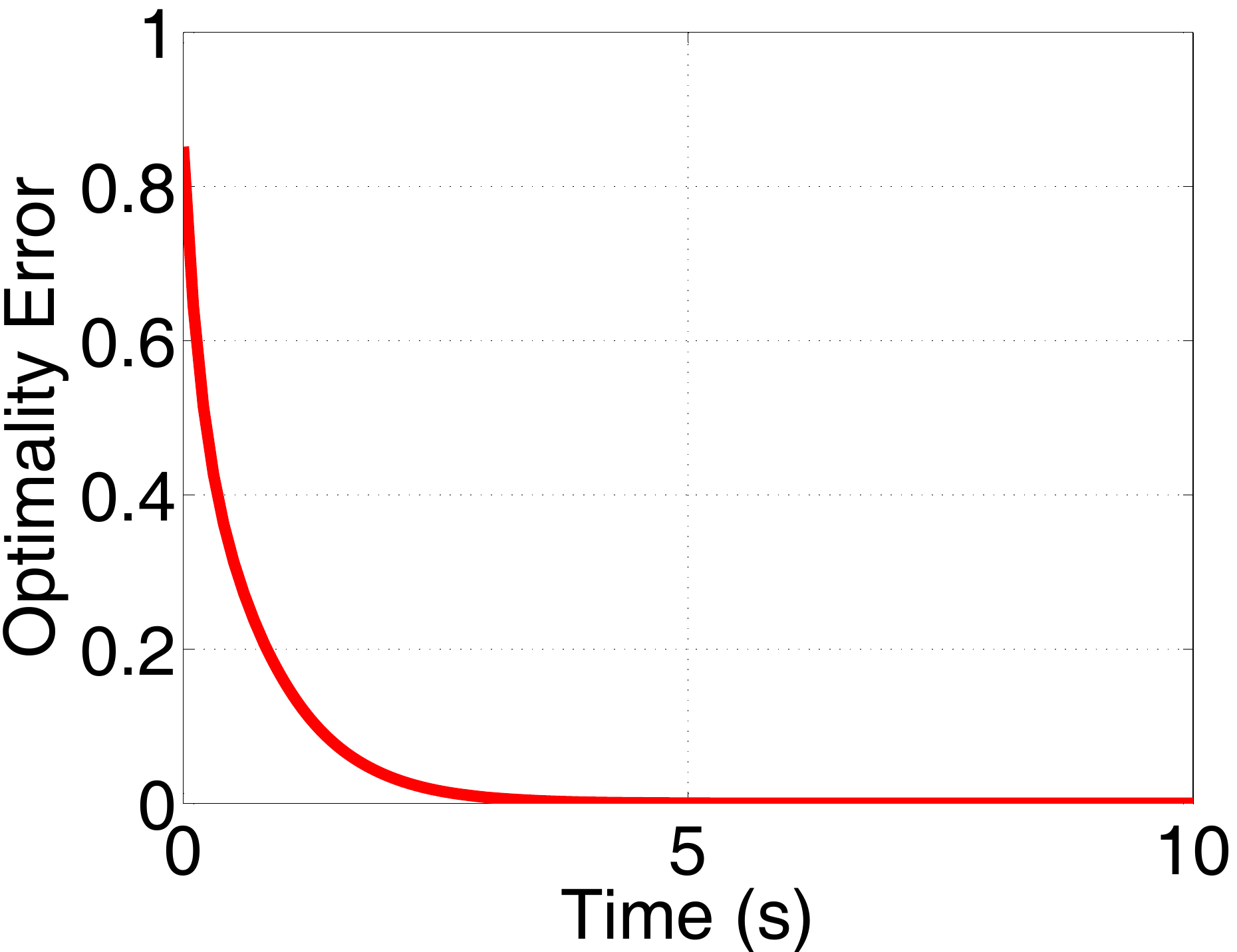}
  \includegraphics[width=0.2\linewidth]{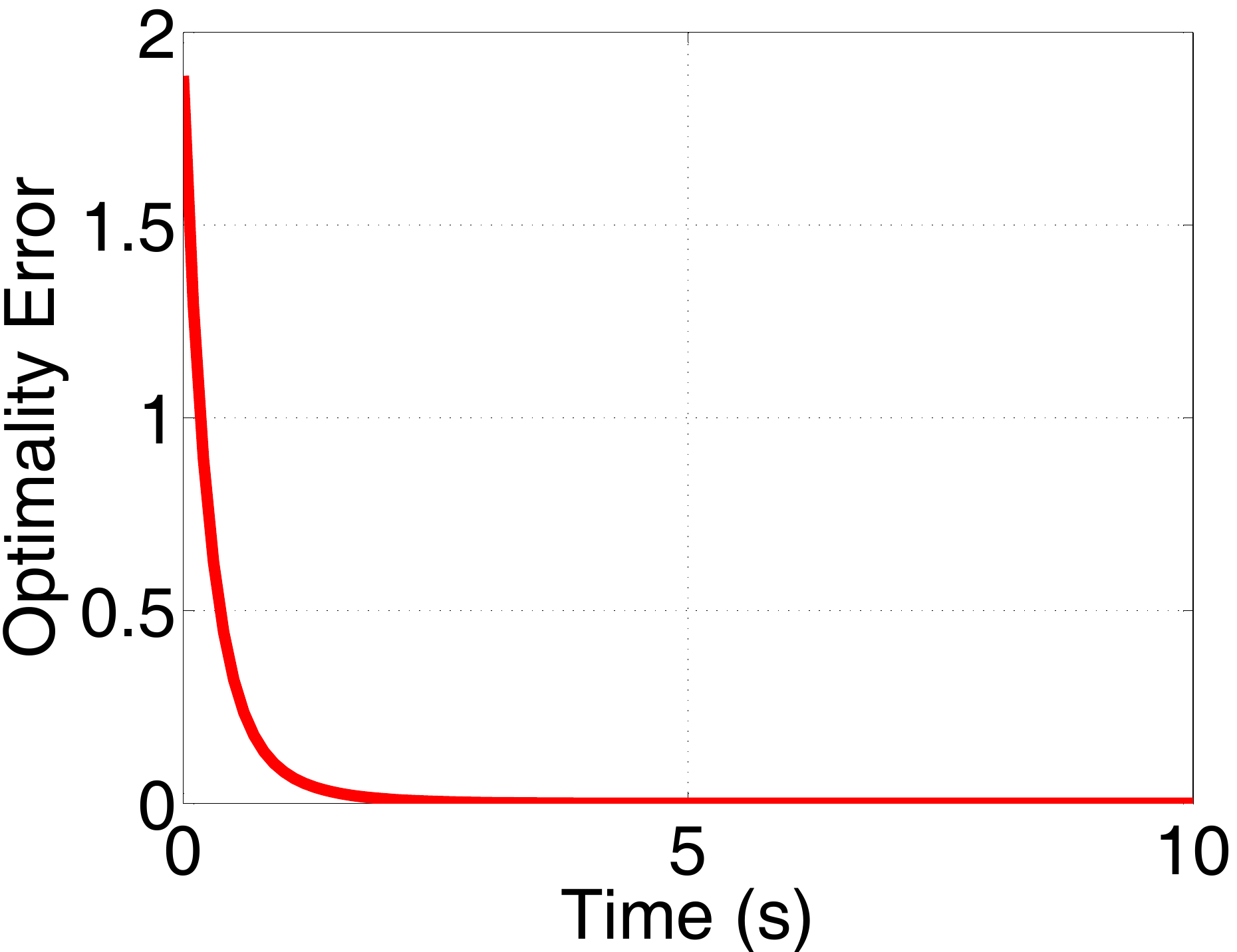}
  \includegraphics[width=0.2\linewidth]{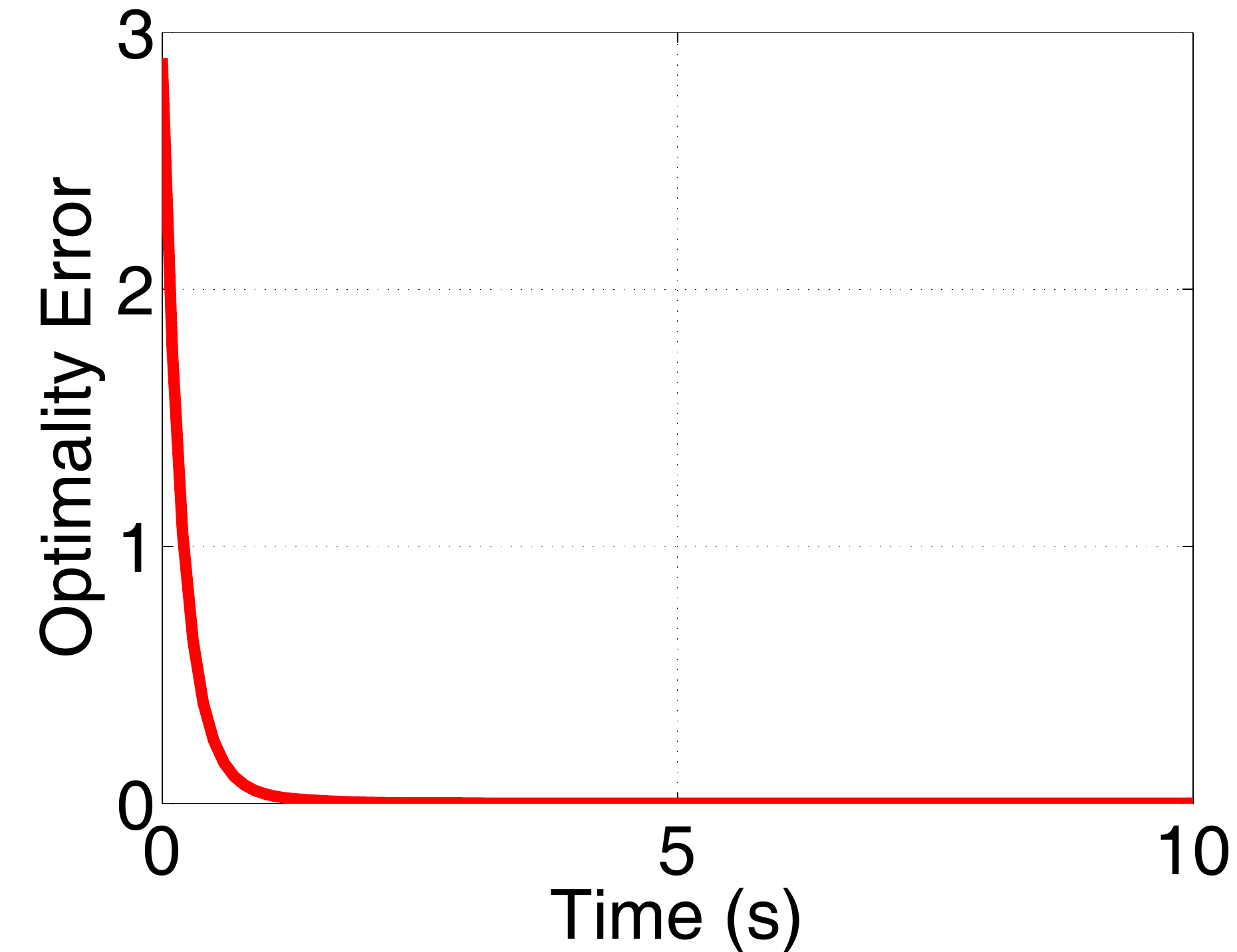}
  \caption{Gradient control of equally-weighted (regular) placements with $n=4$ in $\mathbb{R}^3$.}
  \label{fig_simulation_bearing_3D_regular}
\end{figure}
\begin{figure}
  \centering
  \includegraphics[width=0.2\linewidth]{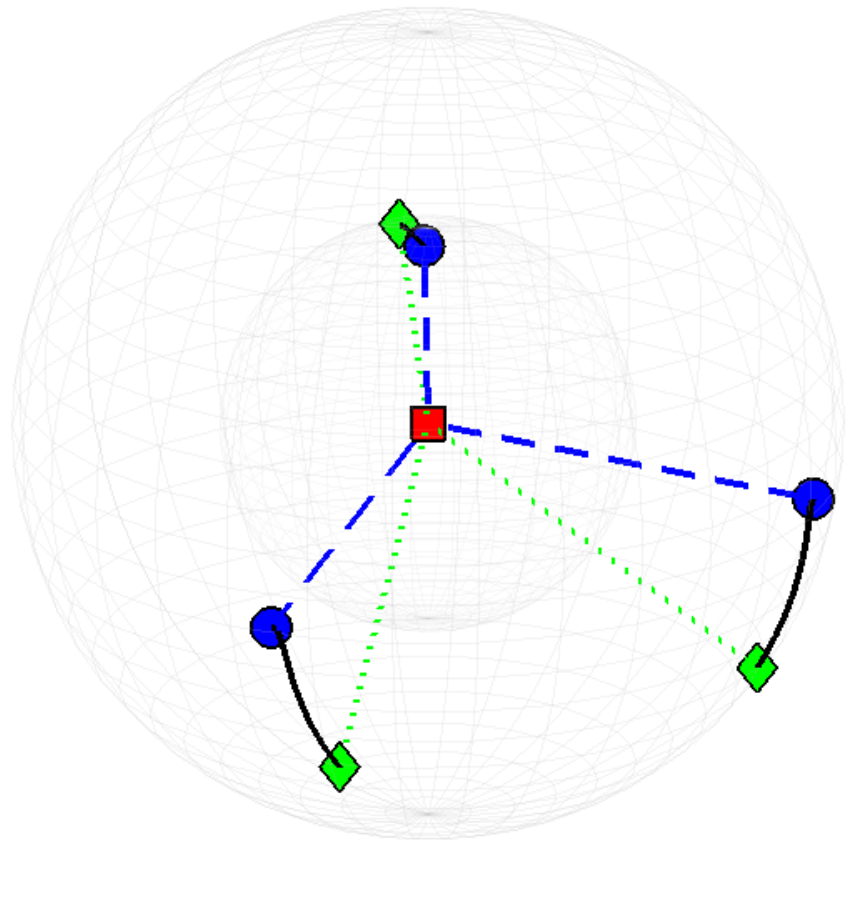}
  \includegraphics[width=0.2\linewidth]{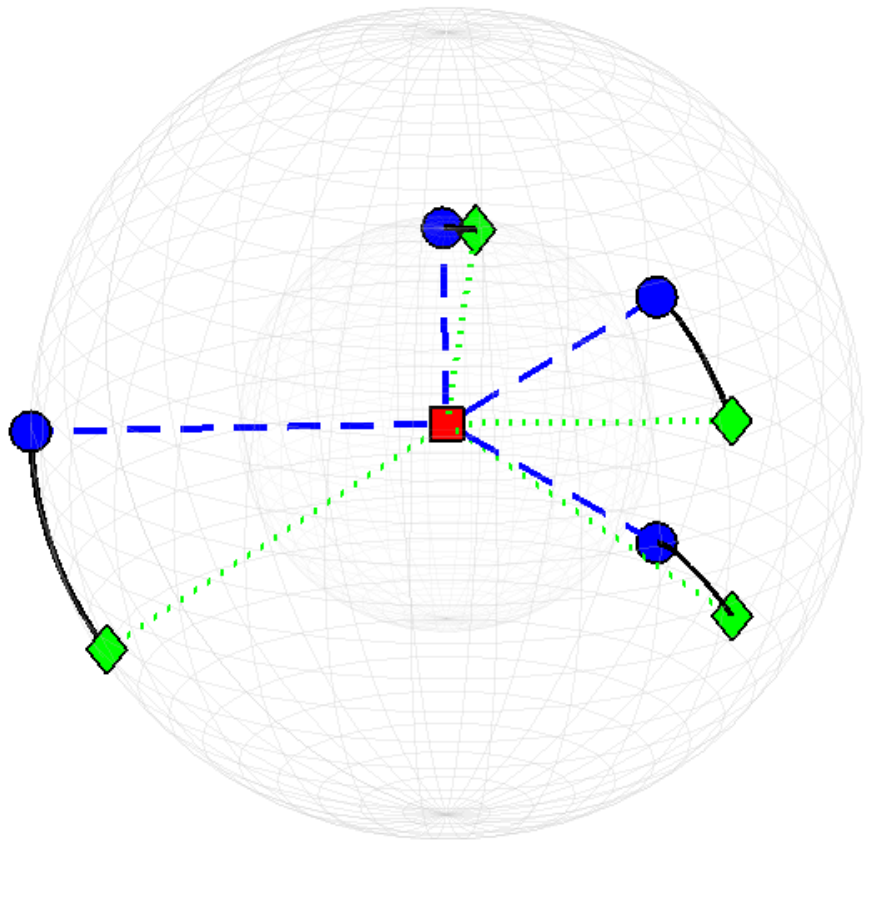}
  \includegraphics[width=0.2\linewidth]{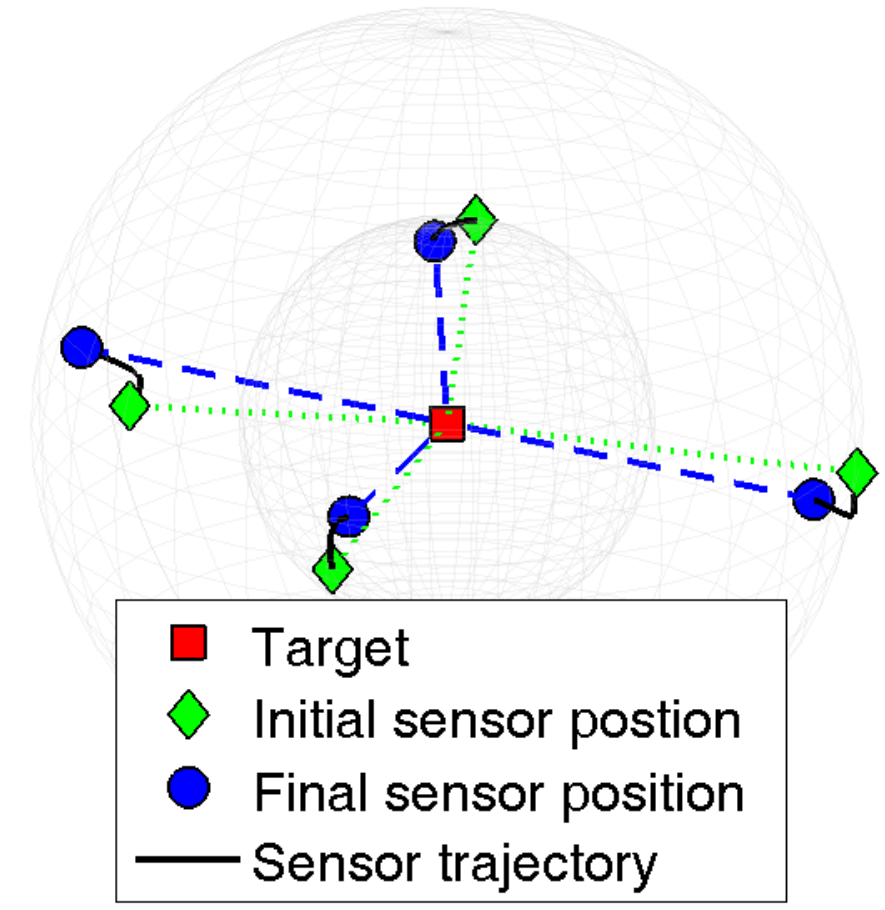}\\
  \subfloat[$n=3$, $k_0=1$]{\includegraphics[width=0.2\linewidth]{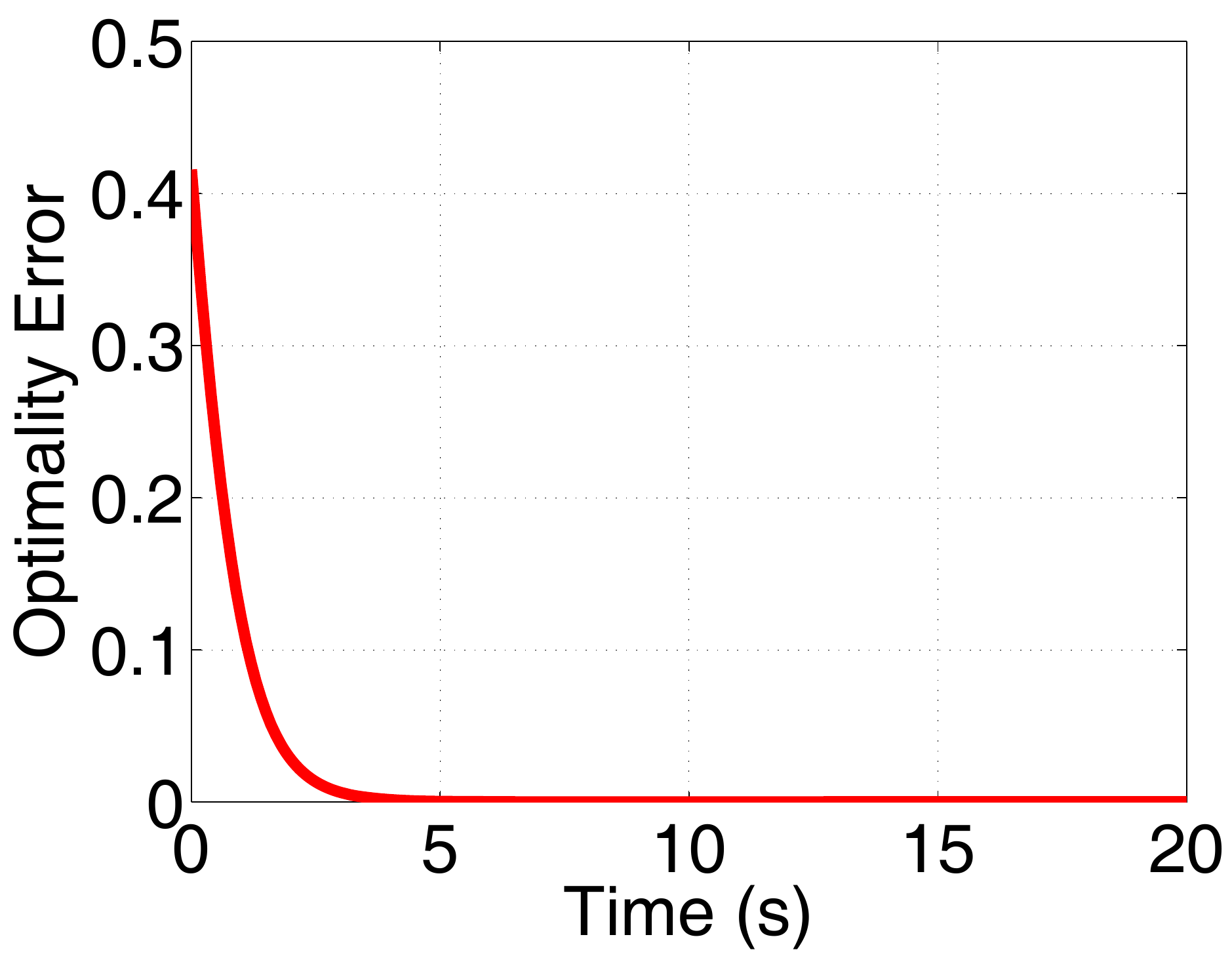}}
  \subfloat[$n=4$, $k_0=1$]{\includegraphics[width=0.2\linewidth]{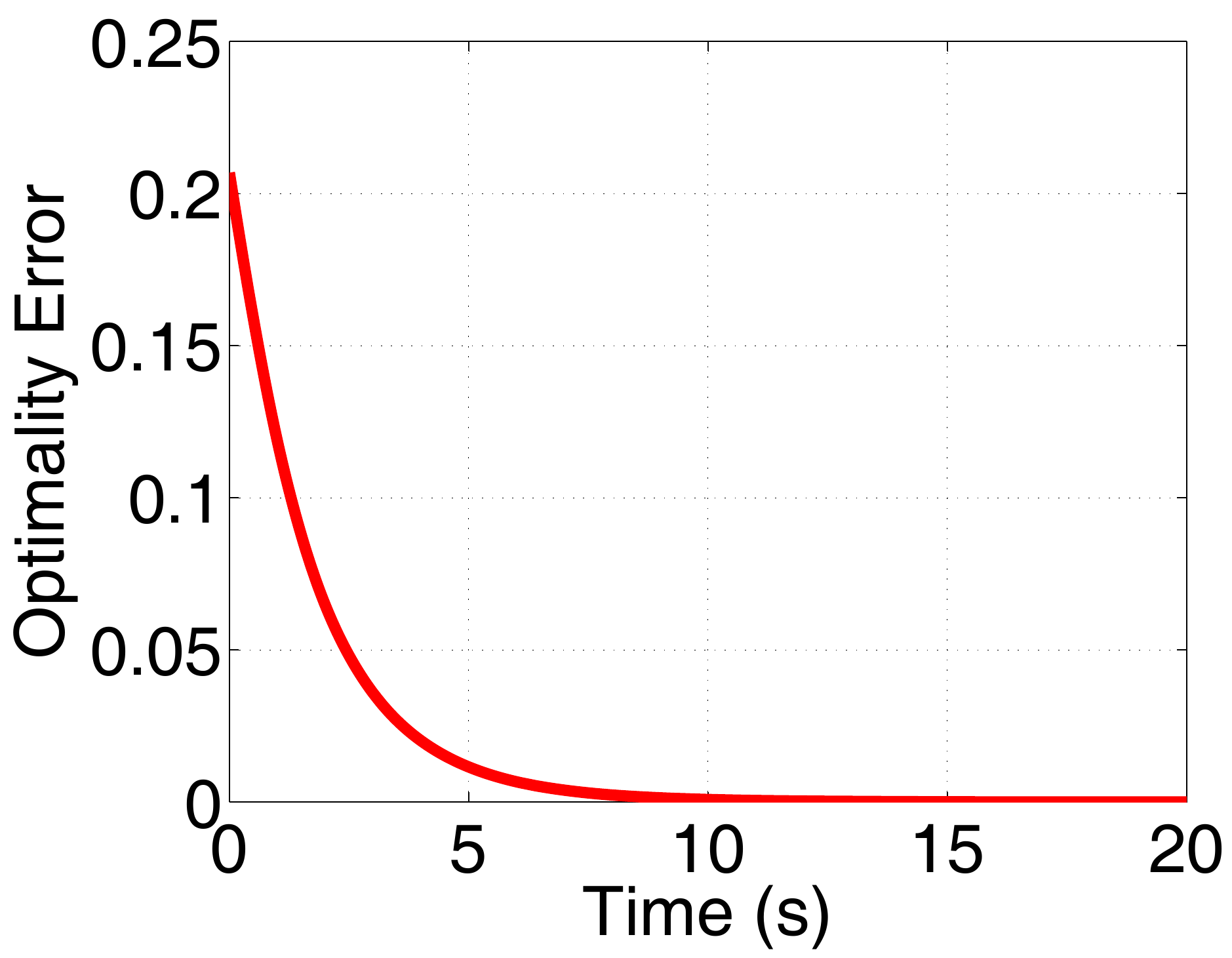}}
  \subfloat[$n=4$, $k_0=2$]{\includegraphics[width=0.2\linewidth]{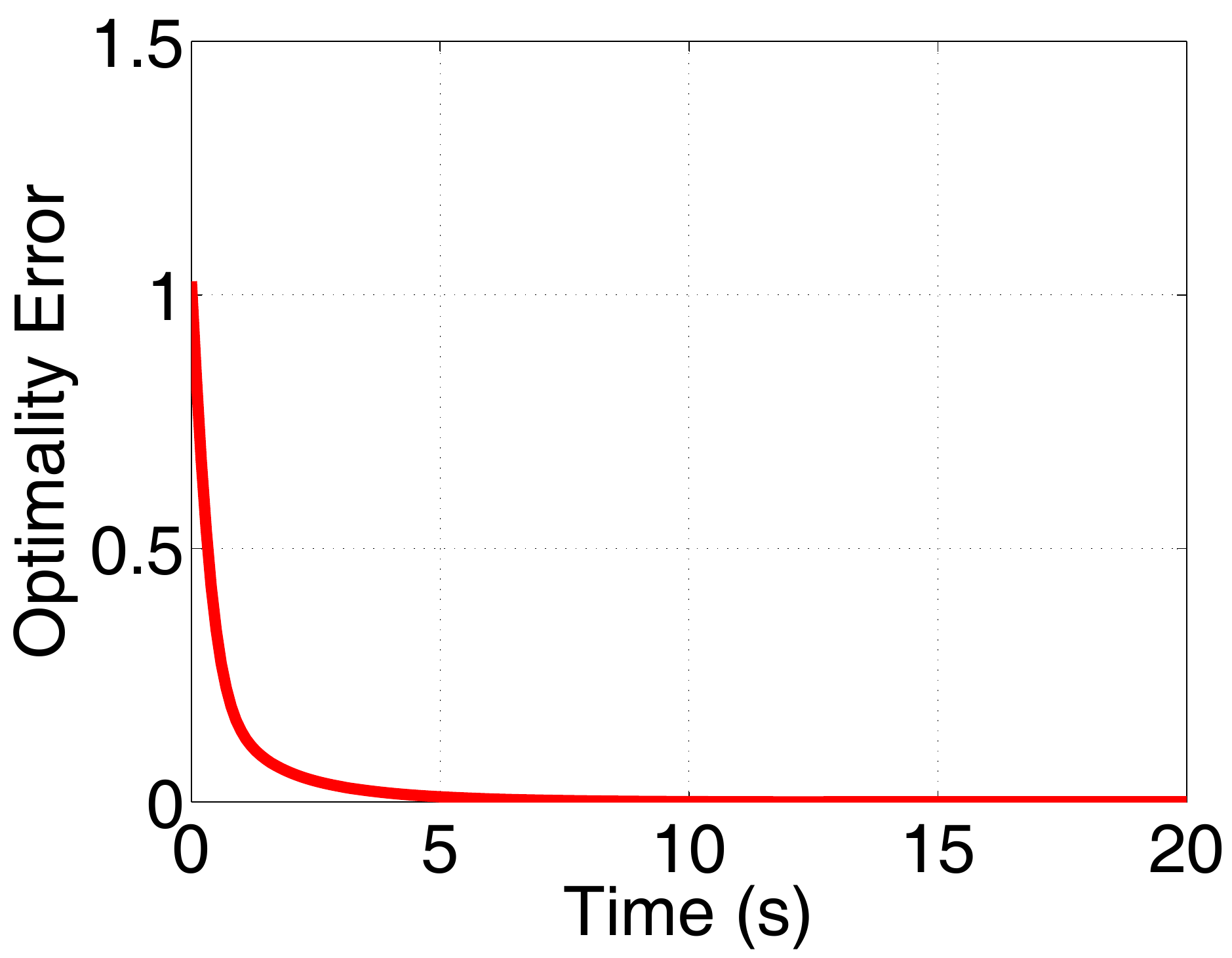}}\\
  \caption{Gradient control of irregular placements in $\mathbb{R}^3$.}
  \label{fig_simulation_bearing_3D_irregular}
\end{figure}

\begin{proposition}\label{theorem_gradientcontrol}
For any initial condition $r(0)\in\mathbb{R}^{dn}$ with $\|r_i(0)\|\neq 0$ for all $i\in\{1,\dots,n\}$, the solution to the nonlinear $r$-dynamics \eqref{eq_closedLoopSystem} asymptotically converges to the set
\begin{align*}
    \mathcal{E}=\left\{r\in\mathcal{S}: P_iGg_i=0, \quad i=1,\dots,n \right\},
\end{align*}
where $\mathcal{S}=\{r\in\mathbb{R}^{dn}: \|r_i\|=\|r_i(0)\|, \  i=1,\dots,n\}$.
\end{proposition}

\begin{proof}
The time derivative of $\|r_i\|$ is
\begin{align}\label{eq_lengthinvariant}
    \frac{\mathrm{d} \|r_i\|}{\mathrm{d} t}=\frac{r_i^T}{\|r_i\|}\dot{r}_i=-g_i^TP_iGg_i=0.
\end{align}
The last equality uses the fact $g_i^TP_i=0$.
By \eqref{eq_lengthinvariant} we have $\|r_i(t)\|\equiv\|r_i(0)\|\ne0$. Hence $\mathcal{S}$ is a positive invariant set with respect to the $r$-dynamics.
The set $\mathcal{S}$ consists of a group of spheres in $\mathbb{R}^d$ centered at the origin.
Thus $\mathcal{S}$ is compact.
Note $\dot{V}=0$ and $\dot{r}=0$ for all points in $\mathcal{E}$.
By the invariance principle \cite{Khalilbook}, every solution starting in $\mathcal{S}$ asymptotically converges to $\mathcal{E}$.
\end{proof}

By Proposition \ref{theorem_gradientcontrol}, the $r$-dynamics converge either to the optimal placement $\mathcal{E}_0$ or the set $\mathcal{E}\setminus\mathcal{E}_0$.
By introducing Lagrange multipliers $\gamma_i$, $i=1,\dots,n$, the constrained optimization problem \eqref{eq_problemdefinition2} is equivalent to minimizing  the Lagrangian function ${L}=\|G\|^2+\sum_{i=1}^n\gamma_i(g_i^Tg_i-1)$.
From $\partial {L} / \partial g_i =0$, a short calculation shows that $\mathcal{E}$ is the \emph{critical point} set,
which consists of not only minimizers of $\|G\|^2$ (i.e., optimal placements) but also saddle points and maximizers of $\|G\|^2$ (i.e., non-optimal placements).
The sets $\mathcal{E}_0$ and $\mathcal{E}$ are equilibrium manifolds.
It is noticed that nonlinear stabilization problems involving equilibrium manifolds also emerge in formation control area recently \cite{Francis2009,Francis2010,Anderson2011}.
It is possible to conduct strict stability analysis including identifying the attractive region of $\mathcal{E}_0$ by using Center manifold theory \cite{Francis2009,Anderson2011} or differential geometry \cite{Francis2010}.
But that will be non-trivial because the geometric structure of $\mathcal{E}_0$ is extremely complicated as shown in \cite{Strawn2006}.

Figure~\ref{fig_simulation_bearing_3D_regular} and Figure~\ref{fig_simulation_bearing_3D_irregular} show a number of optimal placements obtained by the proposed gradient control law.
Due to space limitations, we only show examples in 3D.
The three final converged placements in Figure~\ref{fig_simulation_bearing_3D_regular} are actually the three regular optimal ones shown in Figure~\ref{fig_uniqueplacement_3D}.
The three final placements in Figure~\ref{fig_simulation_bearing_3D_irregular} are the two as illustrated in Figure \ref{fig_irregularPlacements} (b) and (c).
Clearly the numerical results are consistent with our previous analytical analysis.
The \emph{optimality error} refers to the difference between $\|G\|^2$ and its lower bound given in \eqref{eq_lowerboundRegular} or \eqref{eq_lowerboundIrregular}.
The optimality error can be used as a numerical indicator to evaluate the optimality of a placement.
As shown in Figure~\ref{fig_simulation_bearing_3D_regular} and Figure~\ref{fig_simulation_bearing_3D_irregular}, the optimality errors all converge to zero.

\subsection{Simulation}

More simulations will be presented to demonstrate the performance of the gradient control. Our previous analytical characterization of optimal sensor placements will be verified by the simulations.

\subsubsection{Scenario 1 for bearing-only or RSS-based sensors}
For bearing-only or RSS-based sensors, the sensor-target ranges are assumed to be fixed.
The control law \eqref{eq_closedLoopSystem} naturally fulfills this assumption.
Examples of 2D and 3D optimal placements achieved by the gradient control are shown in Figure~\ref{fig_simulation_bearing_2D} and Figure~\ref{fig_simulation_bearing_3D}, respectively.
The noise covariances of all sensors are chosen to be identical in the simulation.

Figure~\ref{fig_simulation_bearing_2D} shows 2D examples.
In (a) or (e), the final angle subtended at the target by the two sensors is 90 degrees.
In (b) and (c), the two optimal placements are equivalent.
In (d), the optimal placement is a distributed construction by two optimal two-sensor sub-placements. The optimal placements in (f), (g), and (h) have non-equal sensor-target ranges. The placement in (g) is irregular.
The sensor with shortest sensor-target range drives the other two sensors to be collinear with the target.

Figure~\ref{fig_simulation_bearing_3D} shows 3D examples.
In (a) or (e), the final angle subtended at the target by any two sensors is 90 degrees. The final placement in (b) forms a tetrahedron. The final placements in (c) and (d) are both equivalent to the one in (b). An optimal placement with irregularity as $k_0=1$ is shown in (f). The sensor with shortest sensor-target range drives the other three sensors to an orthogonal plane, in which the three sensors form a 2D optimal placement like the one in Figure~\ref{fig_simulation_bearing_2D}~(c). An optimal placement with irregularity as $k_0=2$ is shown in (g). These sensors are mutually orthogonal except the two with long sensor-target ranges collinear. A complicated optimal placement with ten sensors is given in (h).

These simulation results are consistent with our previous optimality analysis. It is also observed that sensors can converge to an optimal placement which is close to their initial positions.

\begin{figure*}
  \centering
  \subfloat[$n=2$, $k_0=0$]{\label{subfiga}\includegraphics[width=0.23\linewidth]{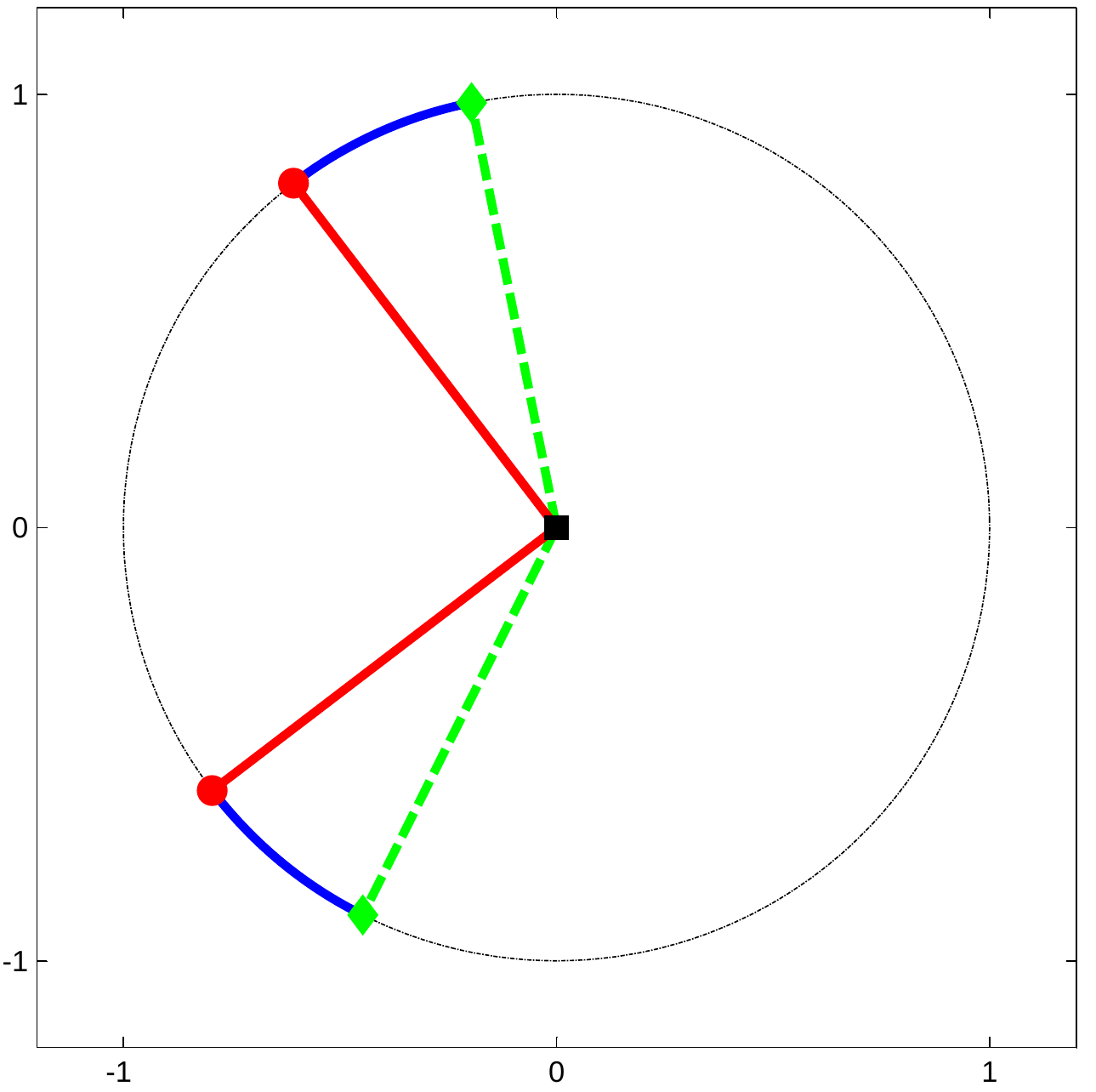}} \ \
  \subfloat[$n=3$, $k_0=0$]{\label{subfigb}\includegraphics[width=0.23\linewidth]{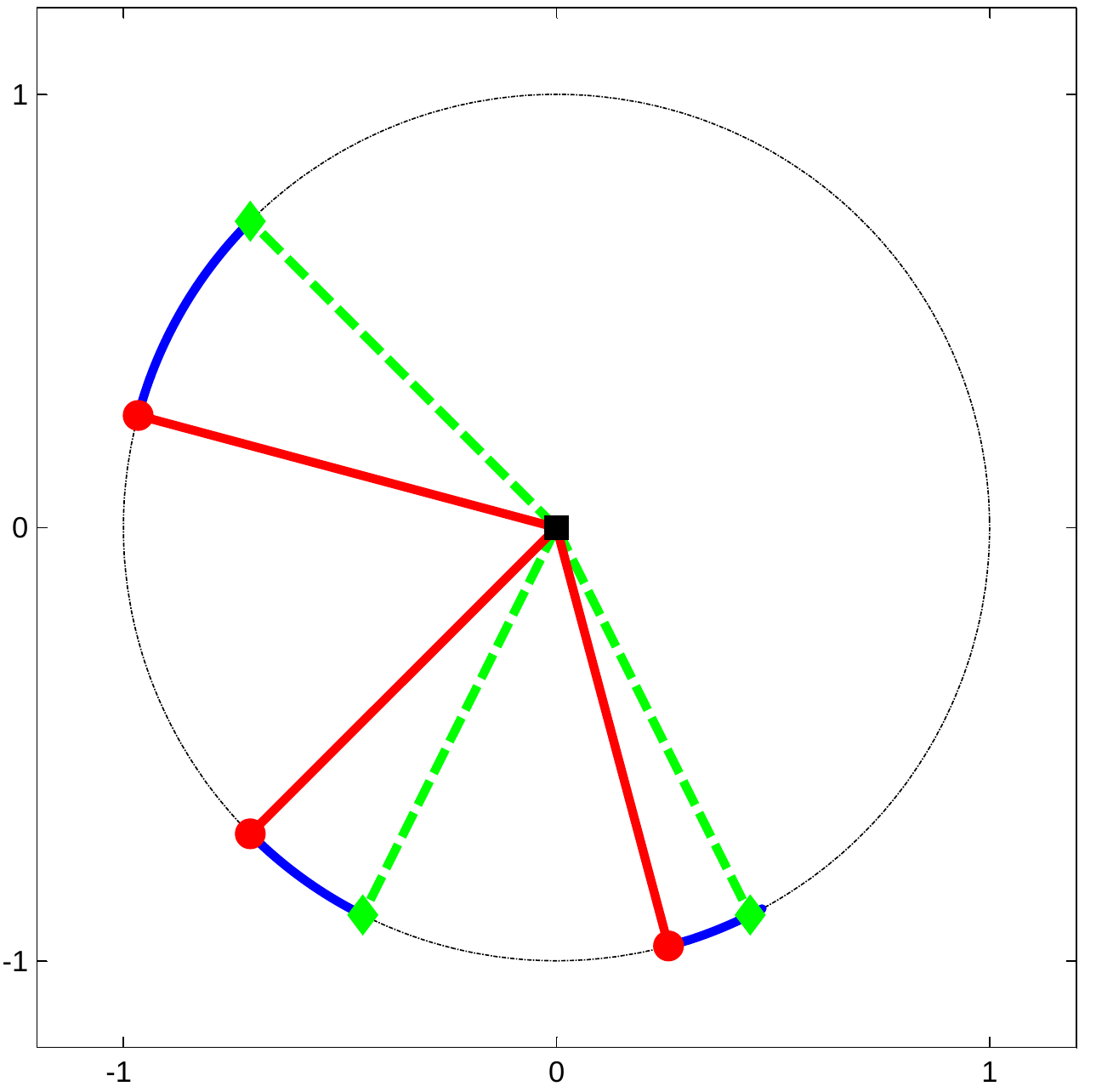}} \ \
  \subfloat[$n=3$, $k_0=0$]{\label{subfigc}\includegraphics[width=0.23\linewidth]{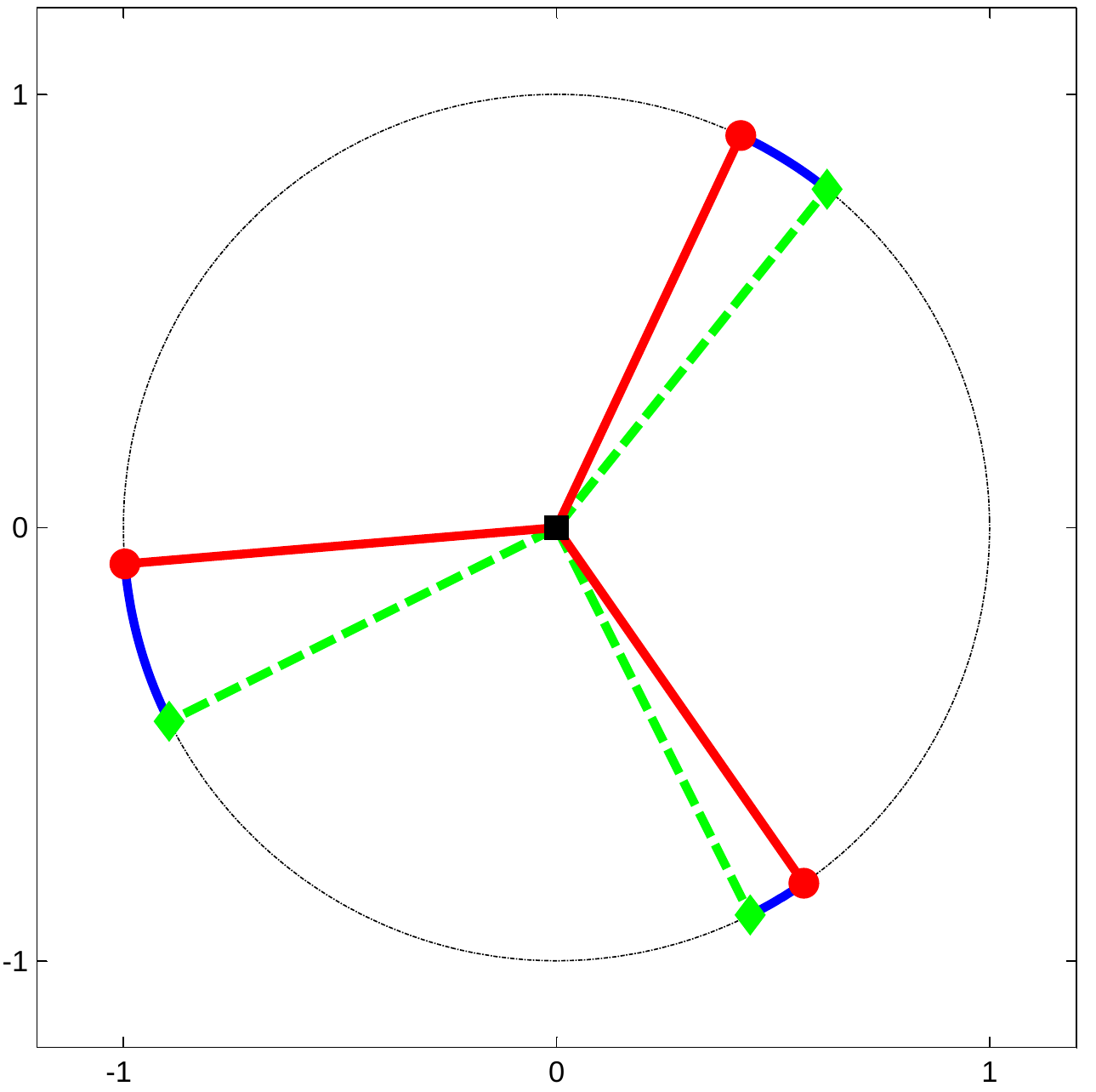}} \ \
  \subfloat[$n=4$, $k_0=0$]{\label{subfigd}\includegraphics[width=0.23\linewidth]{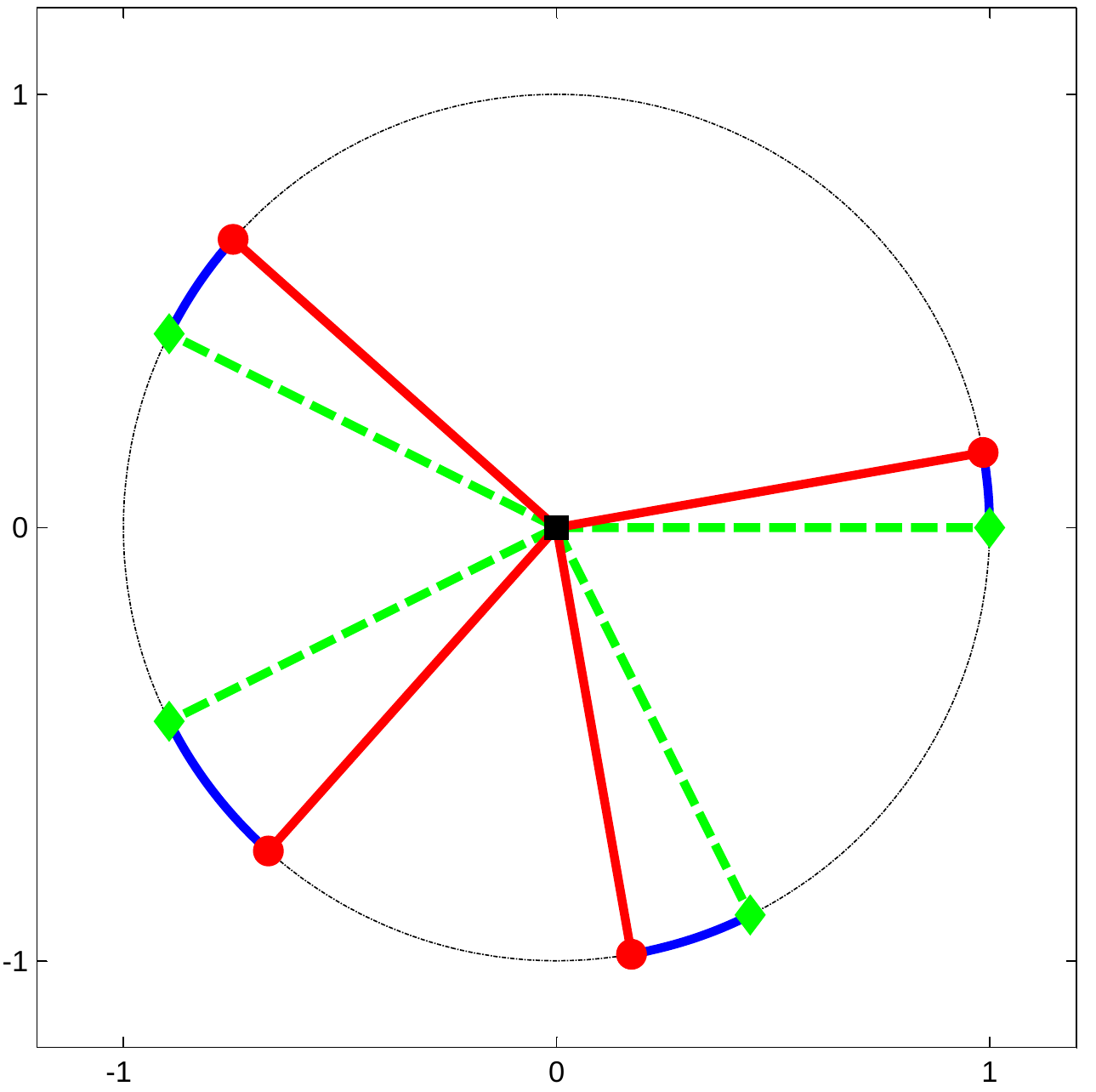}} \\
  \subfloat[$n=2$, $k_0=1$]{\label{subfiga}\includegraphics[width=0.23\linewidth]{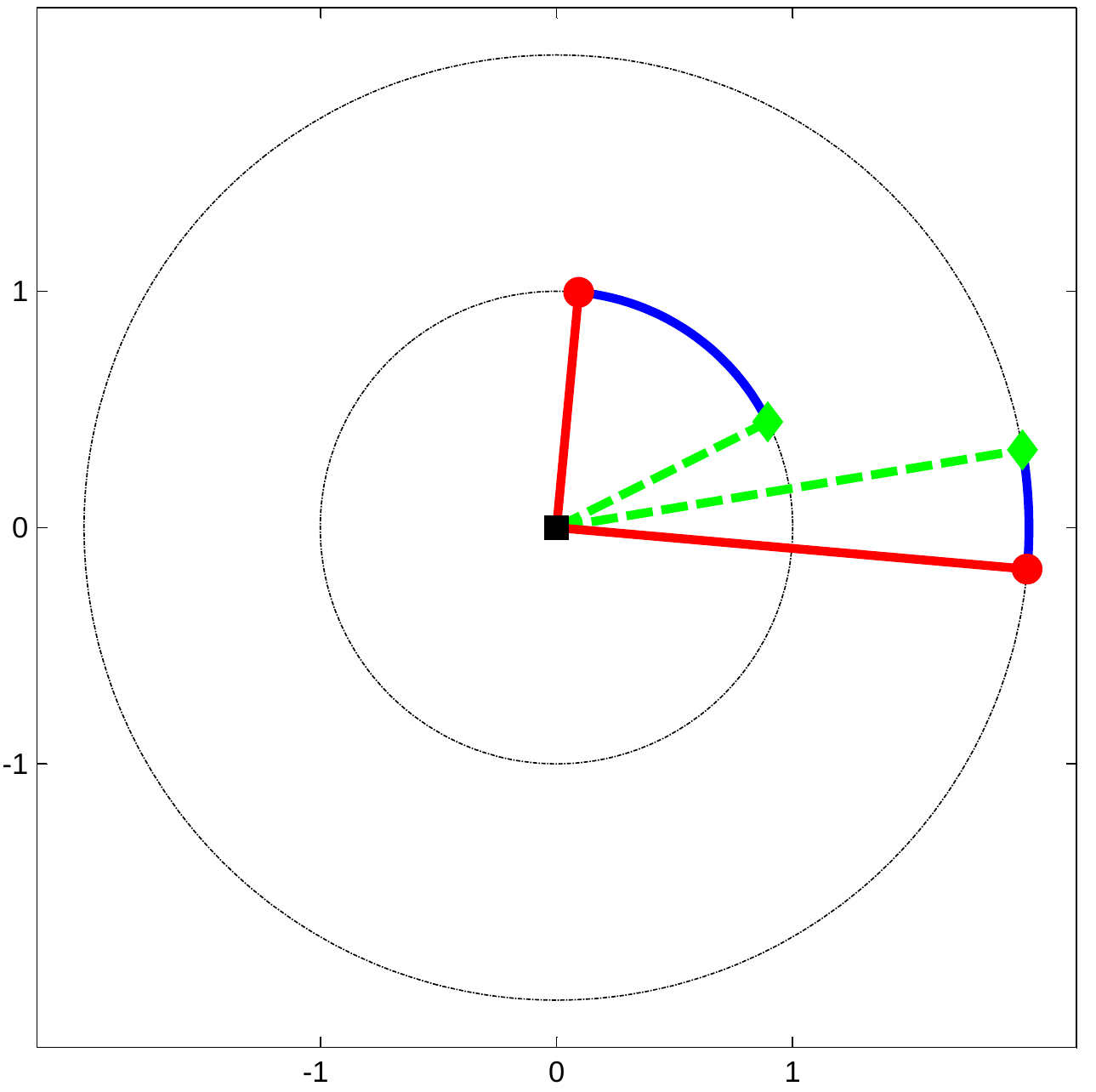}} \ \
  \subfloat[$n=3$, $k_0=0$]{\label{subfigb}\includegraphics[width=0.23\linewidth]{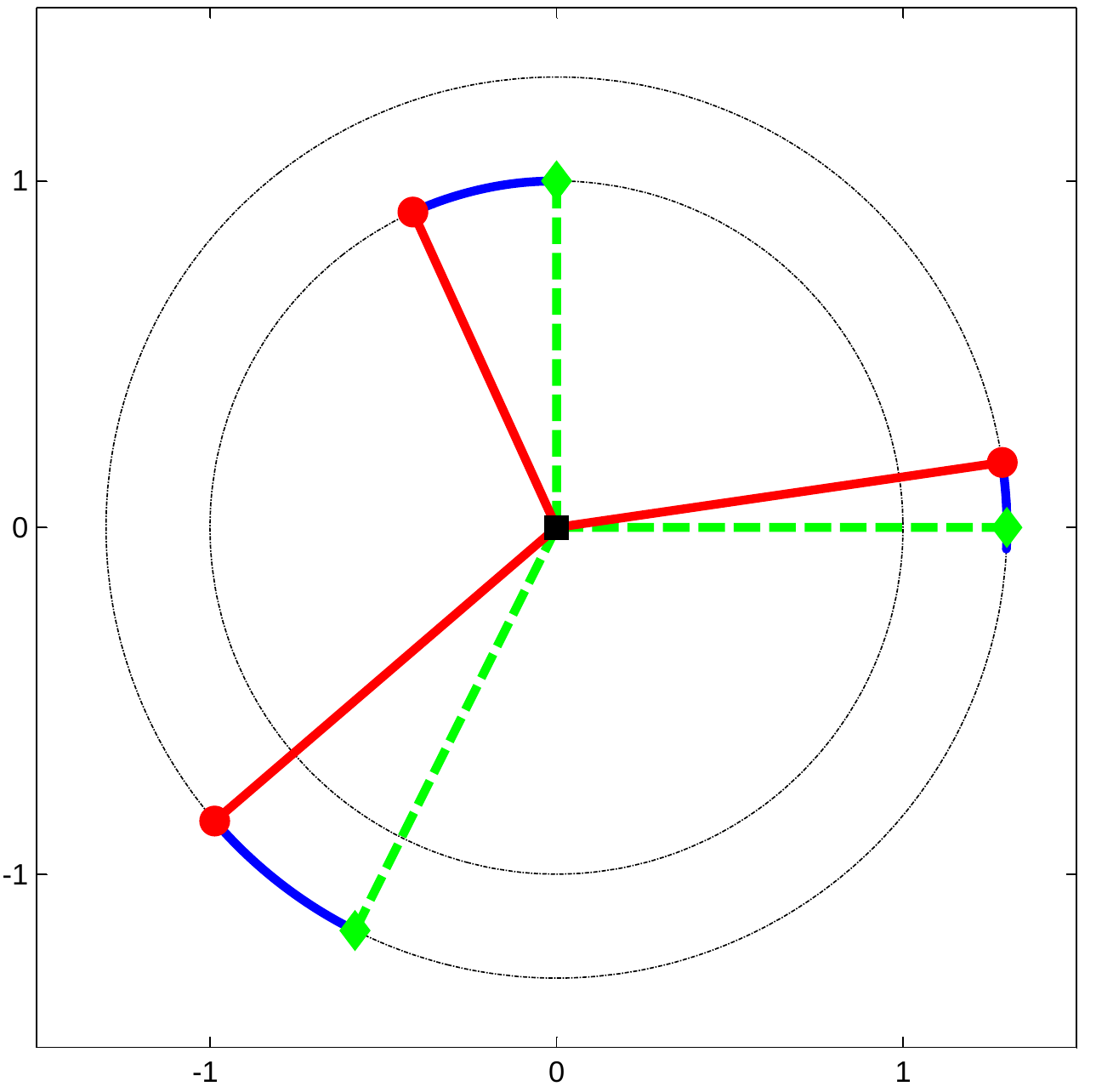}} \ \
  \subfloat[$n=3$, $k_0=1$]{\label{subfigc}\includegraphics[width=0.23\linewidth]{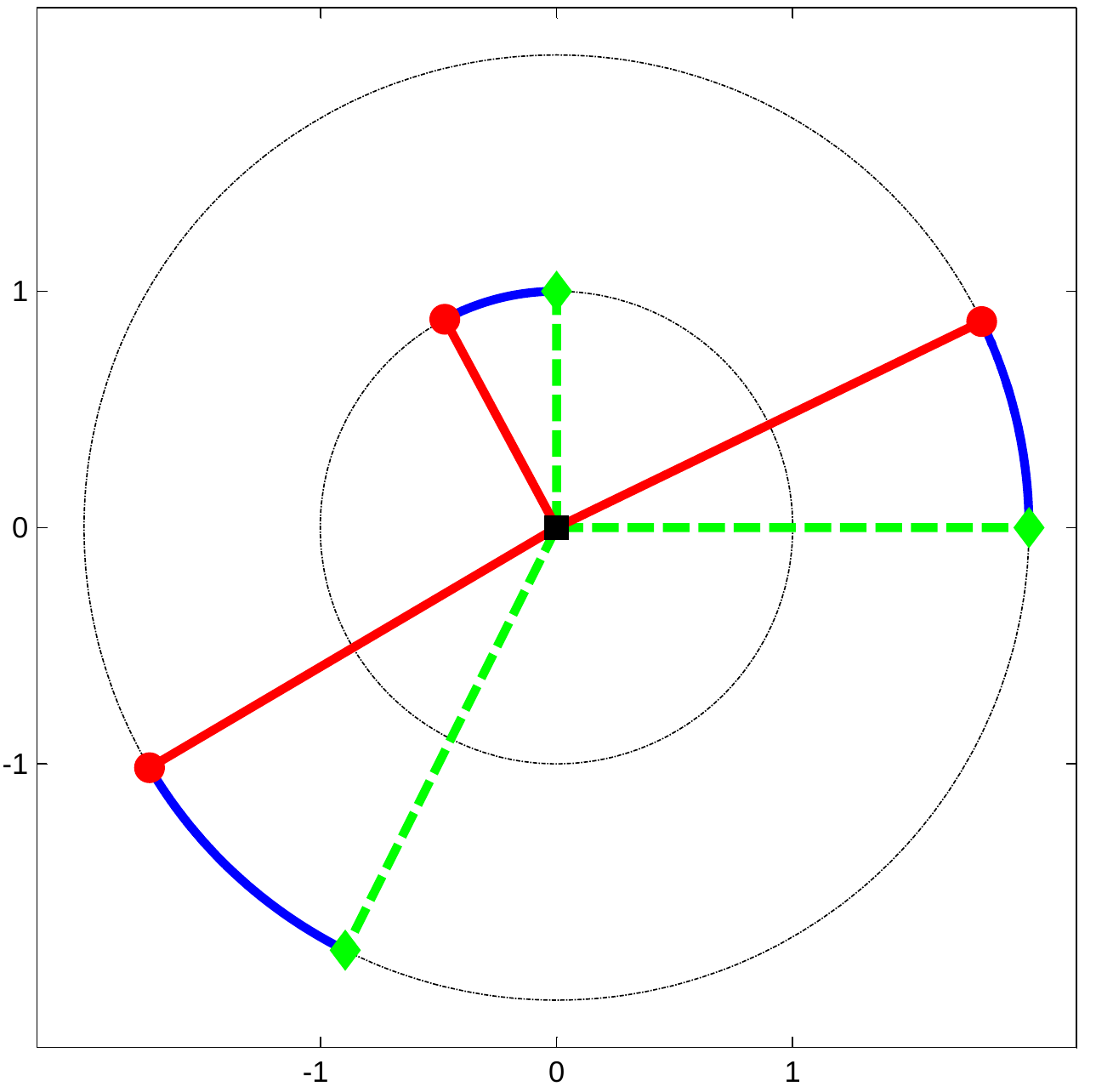}} \ \
  \subfloat[$n=6$, $k_0=0$]{\label{subfigd}\includegraphics[width=0.23\linewidth]{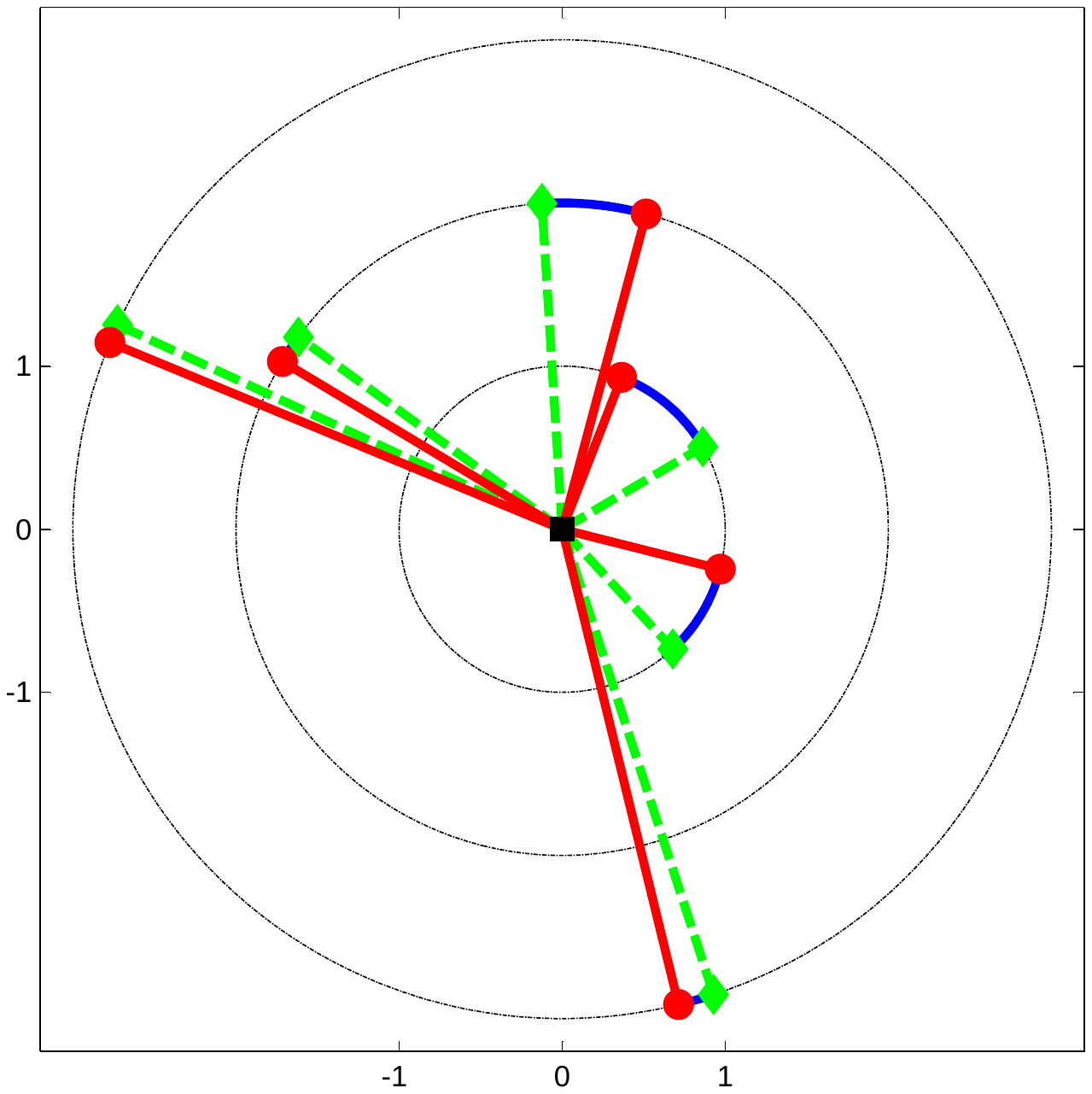}}
  \caption{Examples of 2D optimal placements achieved by the gradient control law. $k_0$: irregularity;
  green diamonds: initial sensor positions; red dots: final sensor positions; blue curves: sensor trajectories; black square: target.}
  \label{fig_simulation_bearing_2D}
\end{figure*}
\begin{figure*}
  \centering
  \subfloat[$n=3$, $k_0=0$]{\label{subfiga}\includegraphics[width=0.25\linewidth]{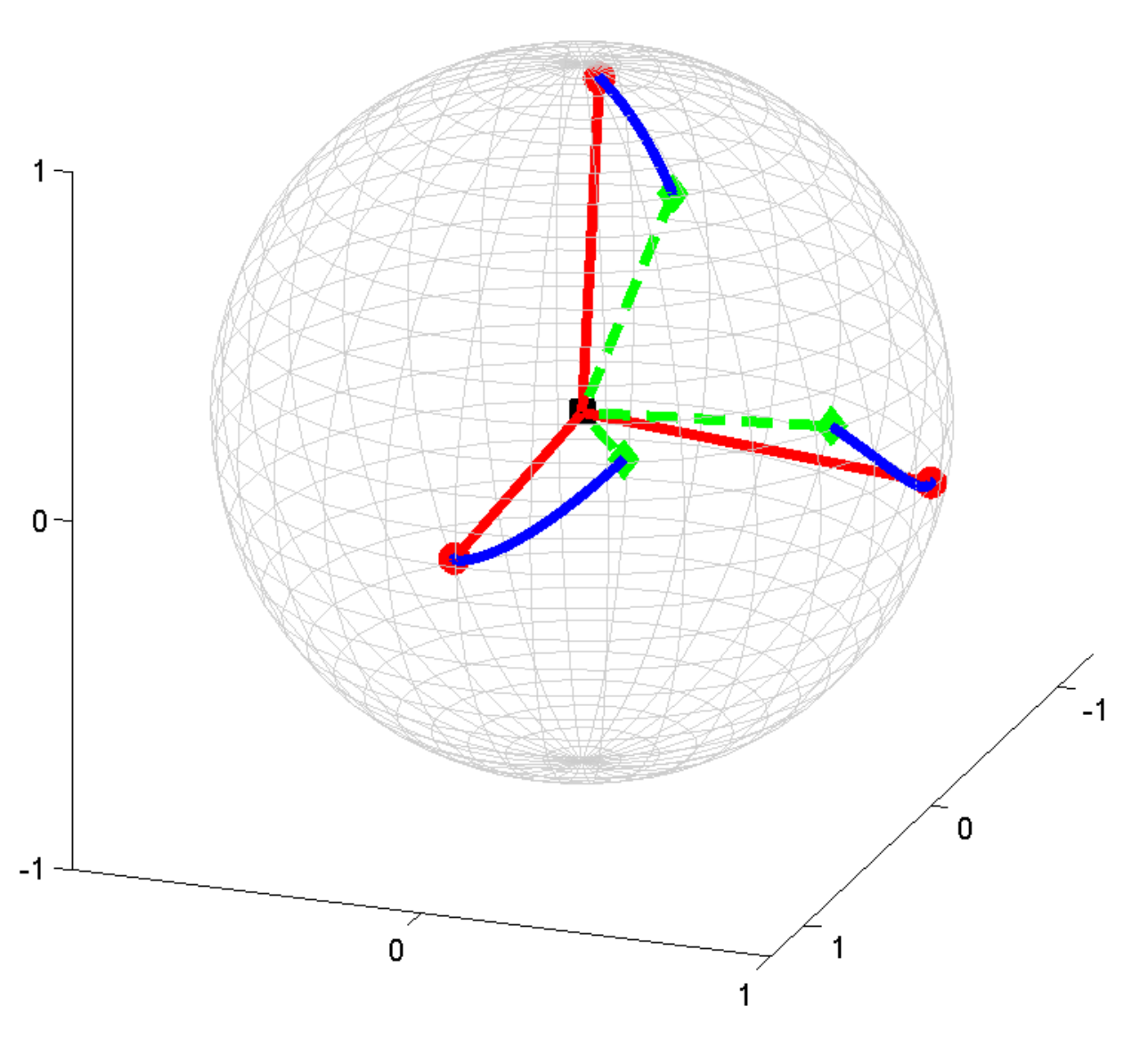}}
  \subfloat[$n=4$, $k_0=0$]{\label{subfigb}\includegraphics[width=0.25\linewidth]{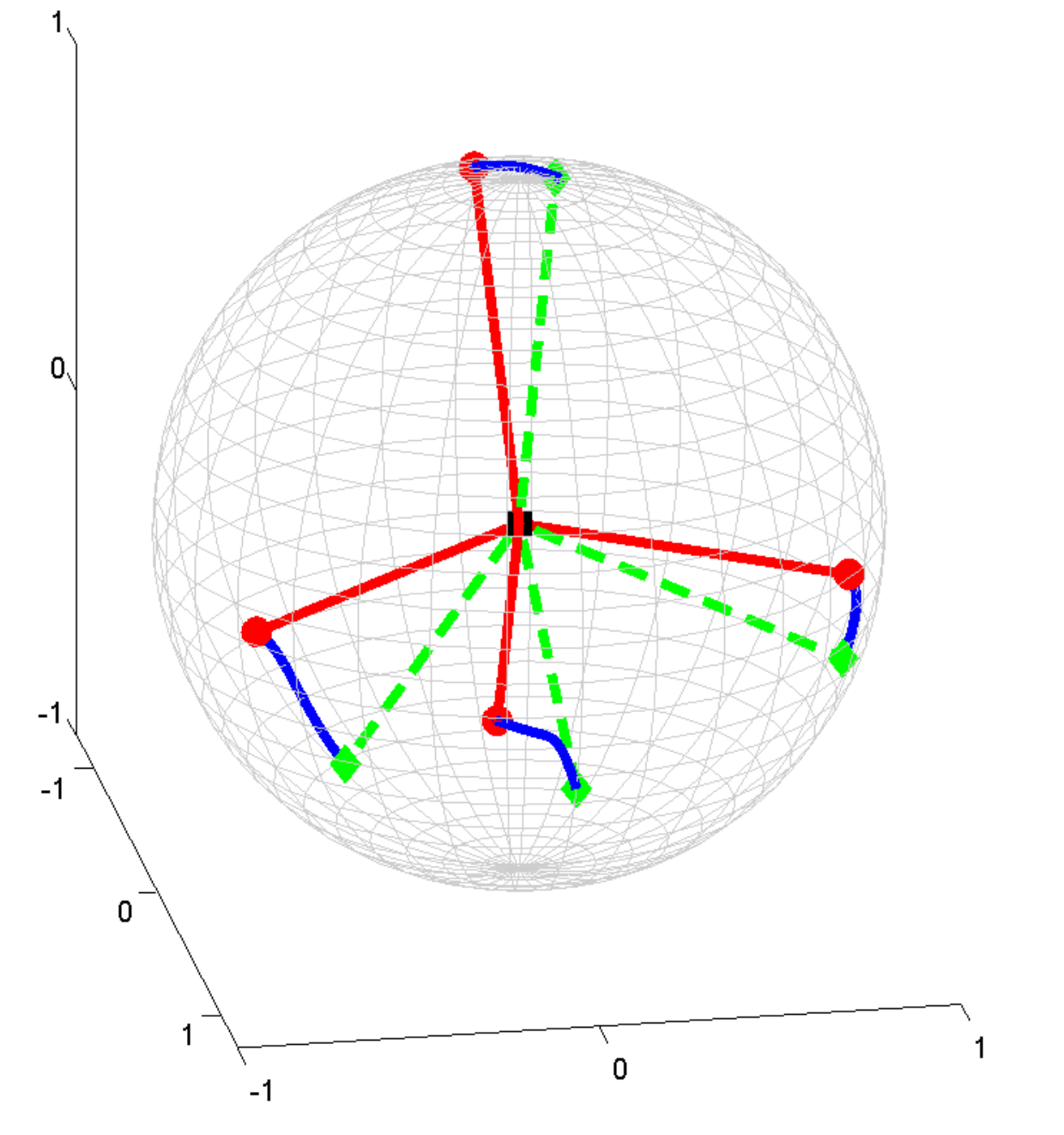}}
  \subfloat[$n=4$, $k_0=0$]{\label{subfigc}\includegraphics[width=0.25\linewidth]{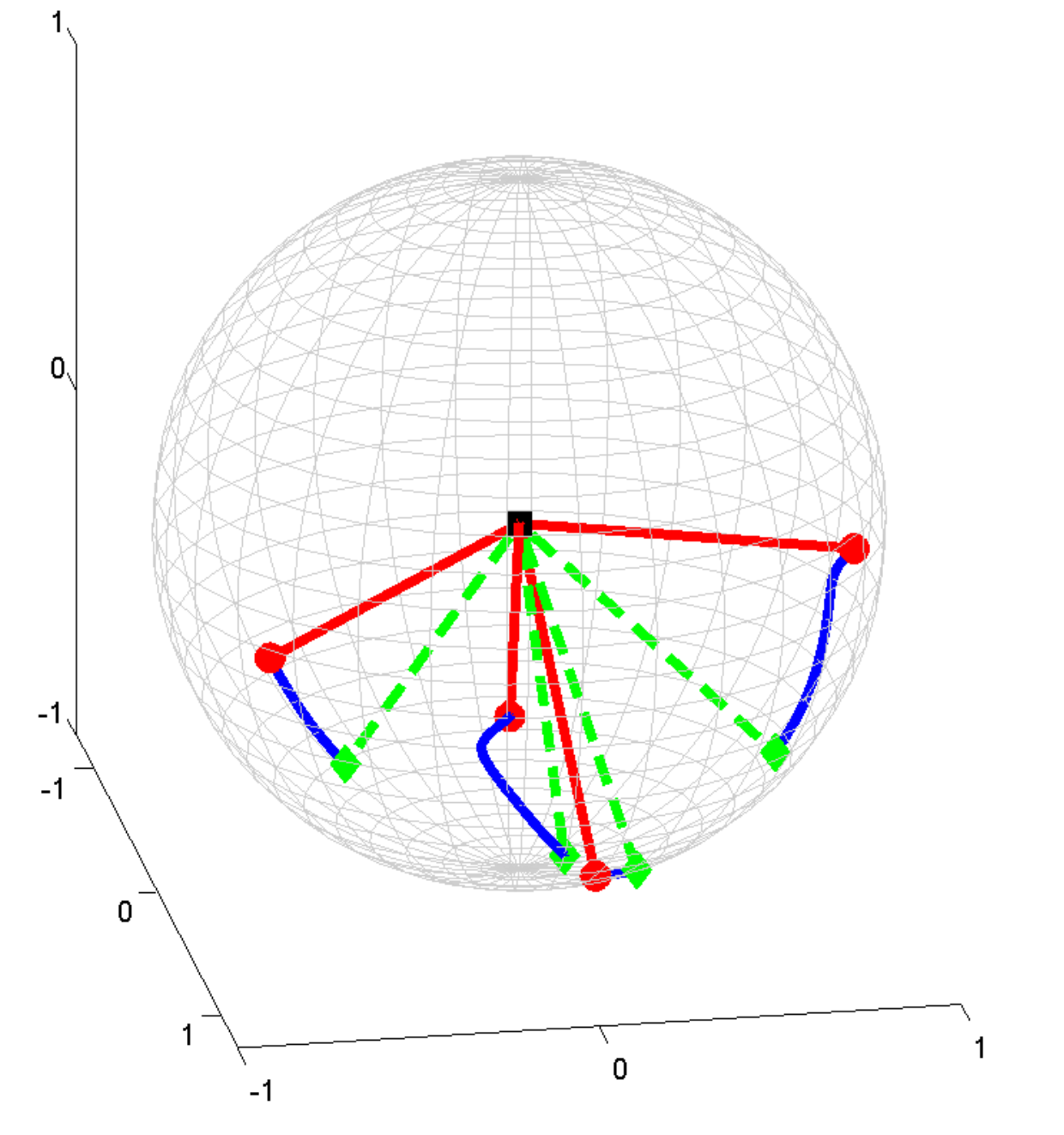}}
  \subfloat[$n=4$, $k_0=0$]{\label{subfigd}\includegraphics[width=0.25\linewidth]{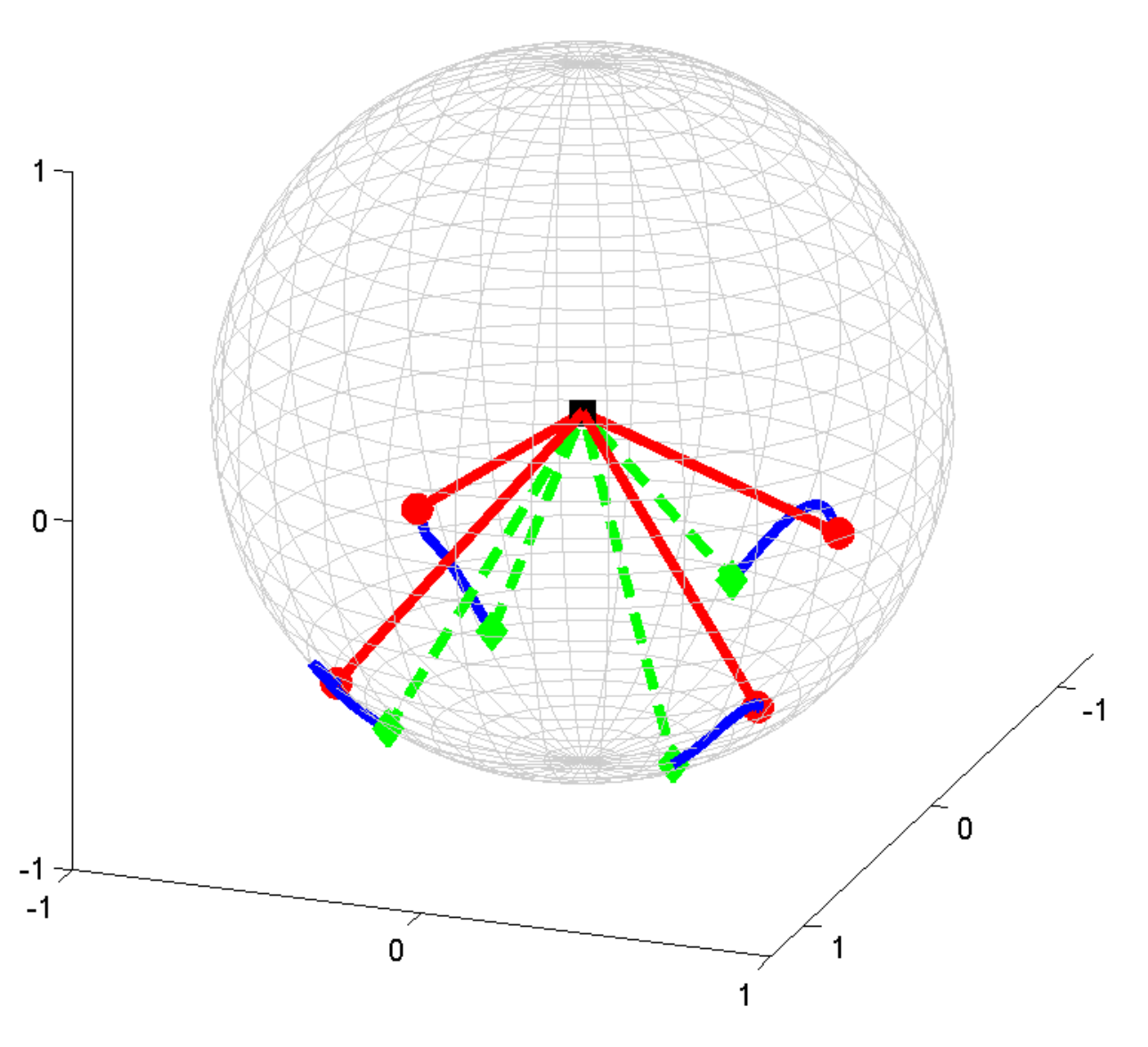}}\\
  \subfloat[$n=3$, $k_0=1$]{\label{subfigb}\includegraphics[width=0.25\linewidth]{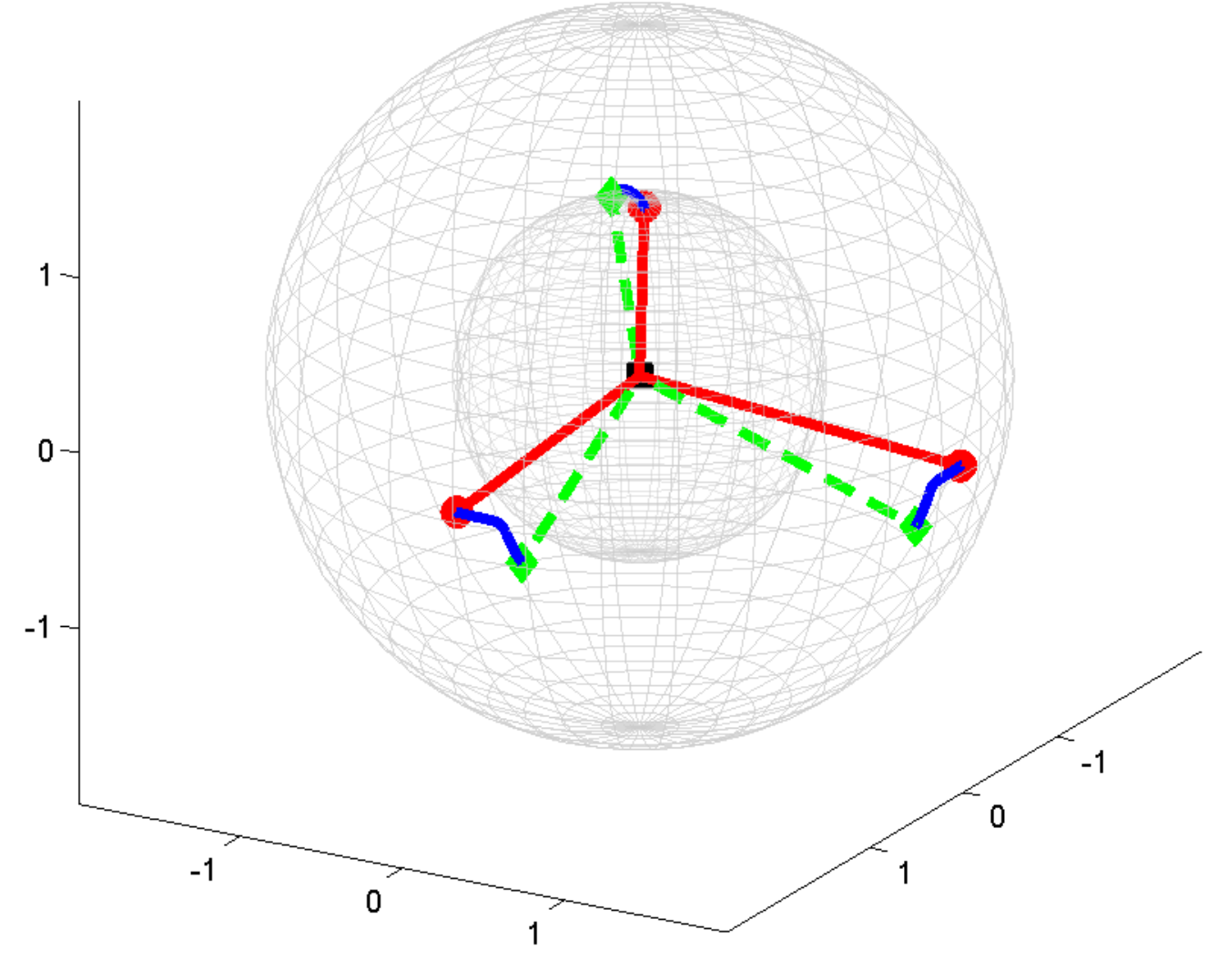}}
  \subfloat[$n=4$, $k_0=1$]{\label{subfigc}\includegraphics[width=0.25\linewidth]{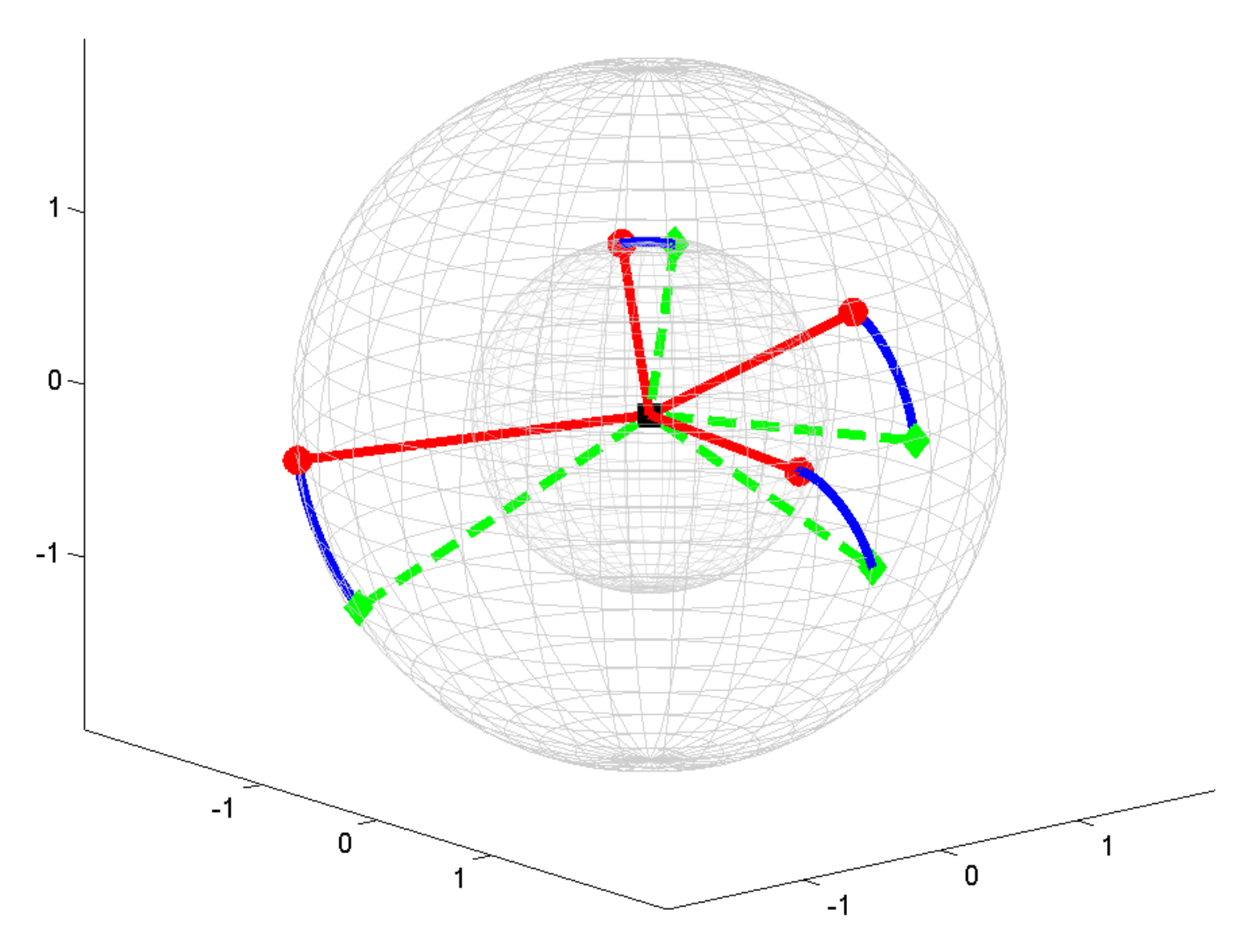}}
  \subfloat[$n=4$, $k_0=2$]{\label{subfigd}\includegraphics[width=0.25\linewidth]{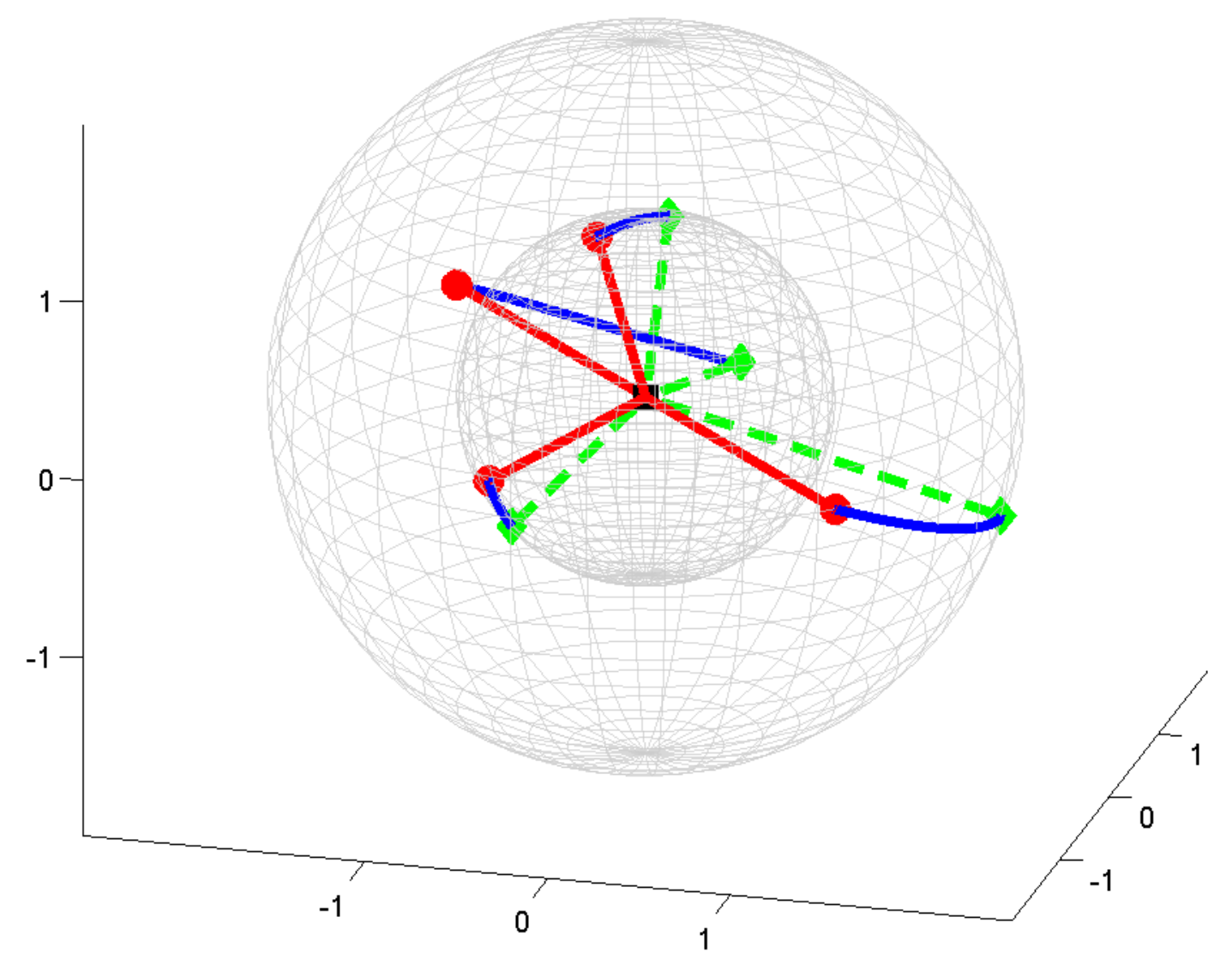}}
  \subfloat[$n=10$, $k_0=0$]{\label{subfiga}\includegraphics[width=0.25\linewidth]{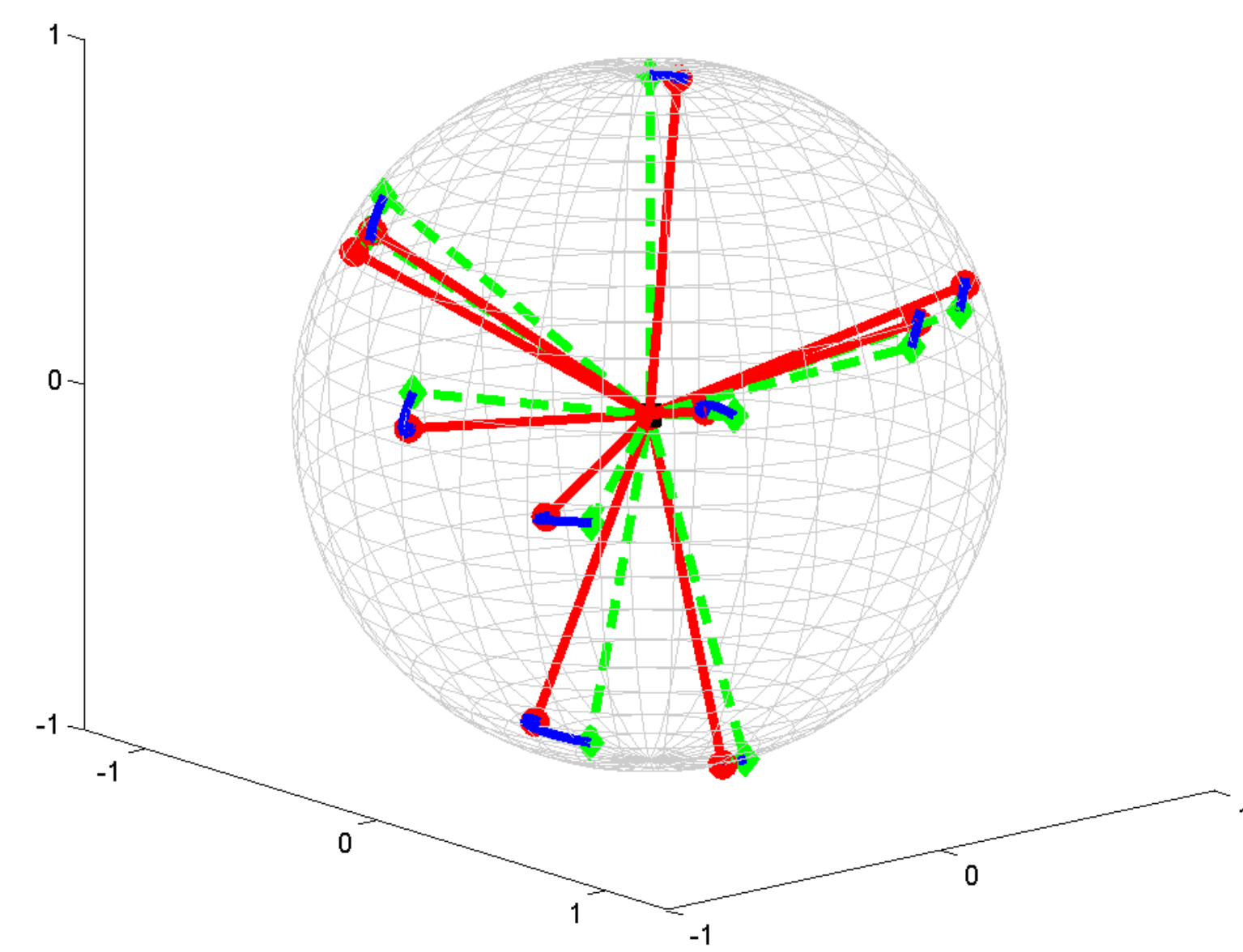}}
  \caption{Examples of 3D optimal placements achieved by the gradient control law. $k_0$: irregularity;
  green diamonds: initial sensor positions; red dots: final sensor positions; blue curves: sensor trajectories; black square: target.}
  \label{fig_simulation_bearing_3D}
\end{figure*}

\subsubsection{Scenario 2 for range-only sensors}

\begin{figure}
  \centering
  \subfloat[3D view]{\label{subfigb}\includegraphics[width=0.5\linewidth]{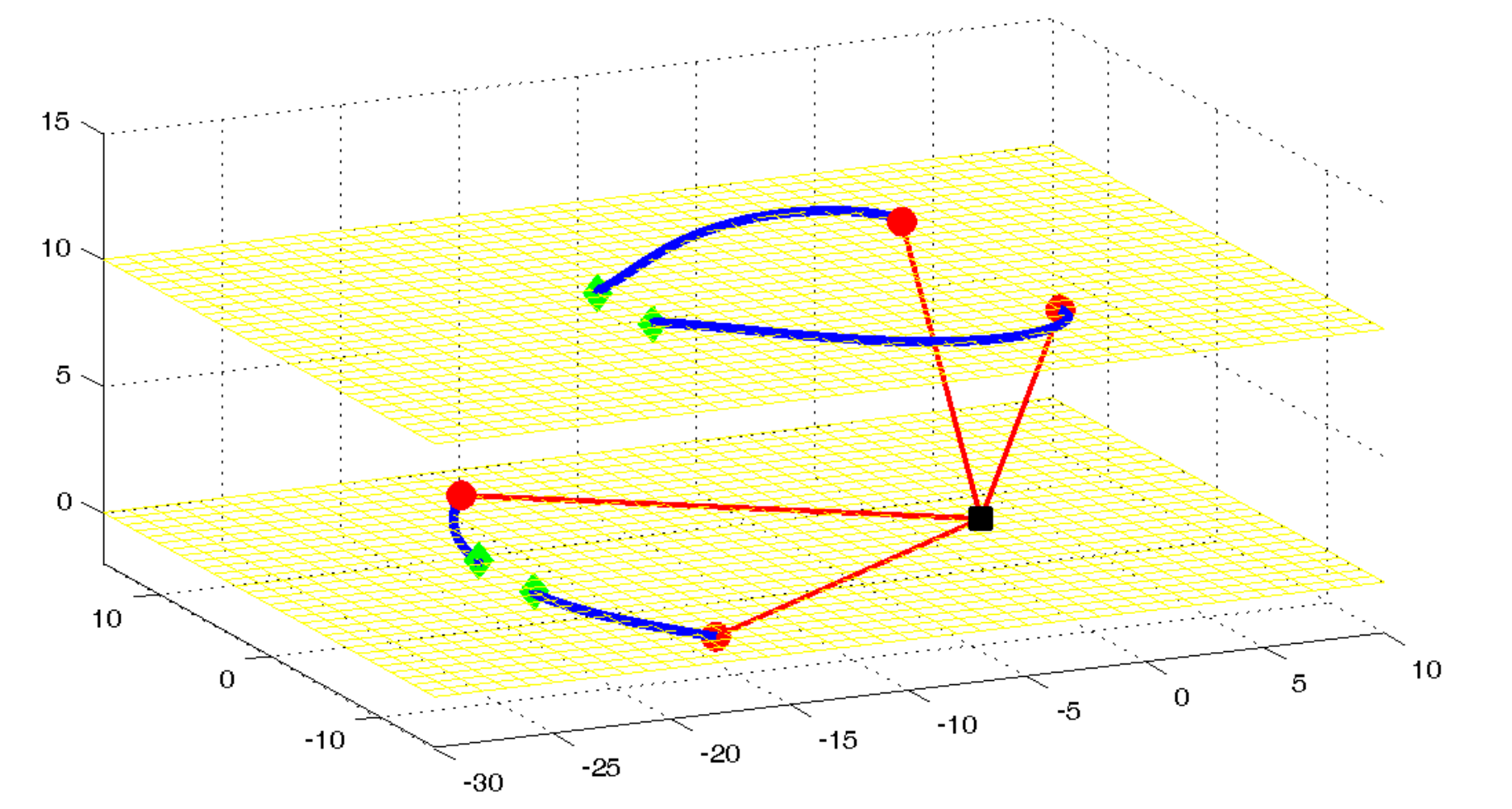}}
  \subfloat[Top view]{\label{subfiga}\includegraphics[width=0.5\linewidth]{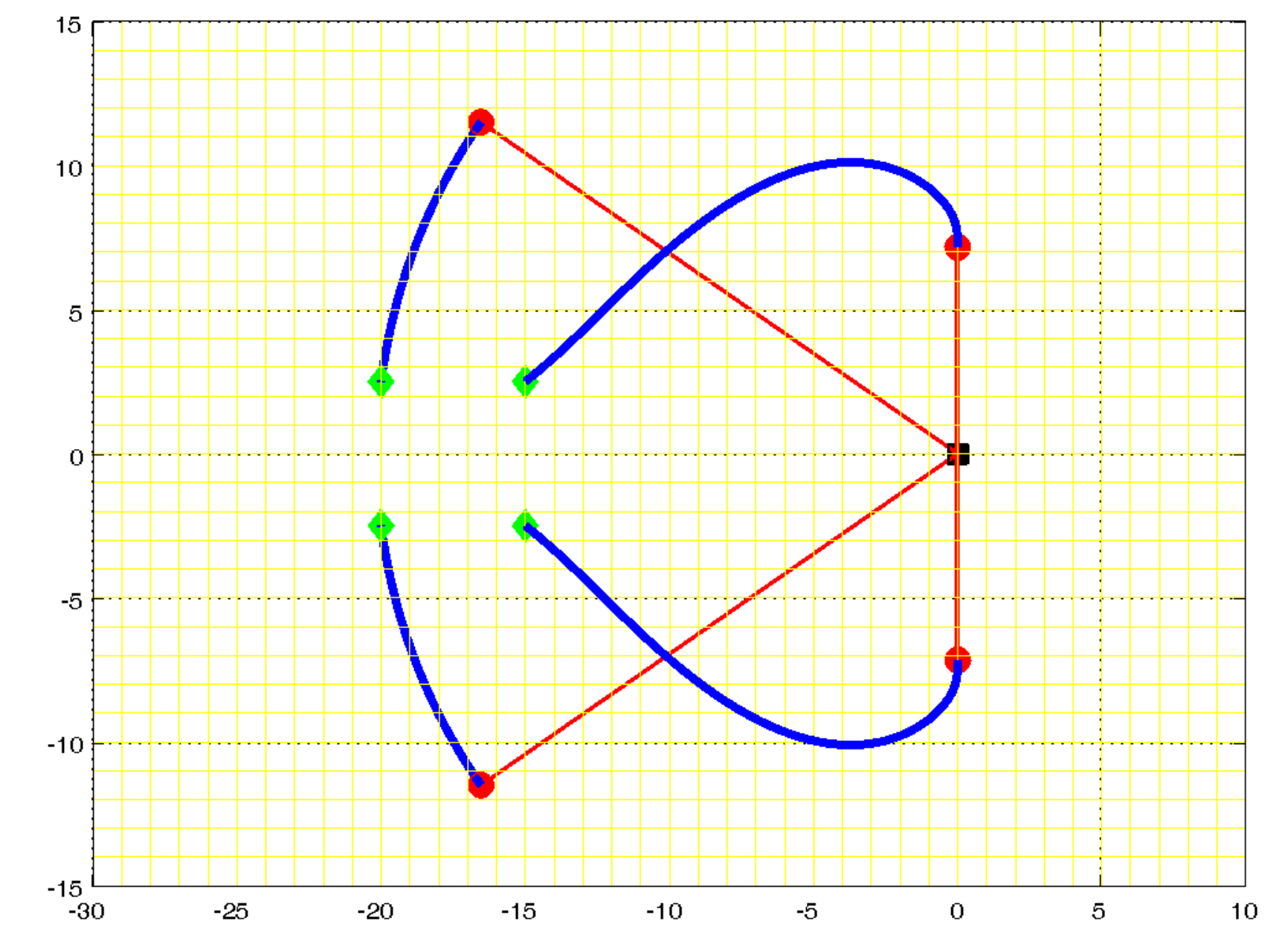}}\\
  \caption{Cooperative air and ground surveillance. Green diamonds: initial vehicle positions; red dots: final vehicle positions; blue curves: vehicle trajectories; black square: ground target.}
  \label{fig_simulation_range_only}
\end{figure}
\begin{figure}
  \centering
  \includegraphics[width=0.5\linewidth]{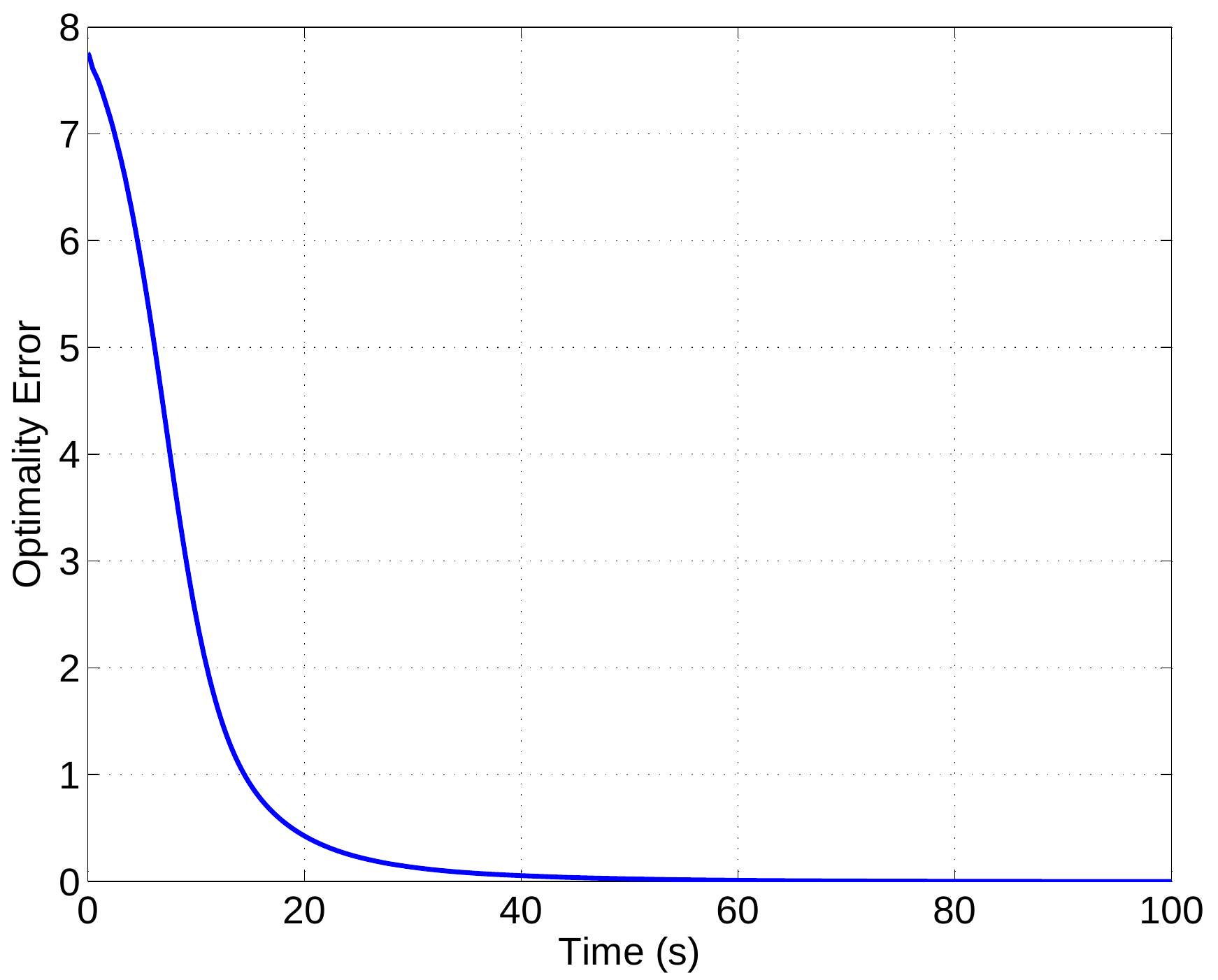}
  \caption{The error between $\|G\|^2$ and its lower bound.}
  \label{fig_simulation_range_error}
\end{figure}

For range-only sensors, the sensor-target ranges are not required to be fixed.
We consider a 3D scenario which might be applied to cooperative air and ground surveillance \cite{Ben2006}.
Consider two unmanned aerial vehicles (UAVs) and two unmanned ground vehicles (UGVs). Each vehicle carries a range-only sensor, which can measure the distance from the ground target to the vehicle. UAVs fly at a fixed altitude and UGVs move on the ground. In addition to the objective function $\|G\|^2$, we introduce the following external potential for sensor $i$:
\begin{align*}
    V_E(r_i)=\left(e_3^T r_i-\ell_i\right)^2,
\end{align*}
where $e_3=[0, 0, 1]^T$ and $\ell_i$ is the required altitude of the vehicles.
Here the external potential is used to control the sensor-target range to fulfill level-plane position constraints. Denote $V_I=\|G\|^2$ as the inter-sensor potential, and $V_E=\sum_{i=1}^n V_E(r_i)$ as the external potential of all sensors.
The total potential function is $V=V_I+V_E$. And the gradient control law becomes
\begin{align*}
    \dot{r}_i=-(\nabla_{r_i}V_I+\nabla_{r_i}V_E),
\end{align*}
where $\nabla_{r_i}$ is the gradient with respect to $r_i$.
The external potential $V_E$ should be \emph{compatible} with the inter-sensor potential $V_I$, otherwise, the external potential will affect the convergence to optimal placements. By compatible, we mean $\nabla_{r_i}V_I+\nabla_{r_i}V_E=0$ if only if $\nabla_{r_i}V_I=0$ and $\nabla_{r_i}V_E=0$.
Since $\nabla_{r_i}V_I\perp r_i$, $V_E$ is compatible with $V_I$ if there is no $r_i$ make $\nabla_{r_i}V_E\perp r_i$.
One compatible example is that all range-only sensors move on the boundary of a convex area containing the target \cite{Sonia06}.

In the simulation, we choose $\ell_1=\ell_2=10$ and $\ell_3=\ell_4=0$.
The noise variances of different sensors are the same.
Figure \ref{fig_simulation_range_only} shows the trajectories and the converged placement of the vehicles.
For complicated placements, it is usually difficult to intuitively judge their optimality.
We can evaluate the optimality by using the error between $\|G\|^2$ and its lower bound given in \eqref{eq_lowerboundIrregular} and \eqref{eq_lowerboundRegular}.
As shown in Figure \ref{fig_simulation_range_error}, the optimality error converges to zero.

\section{Conclusions}\label{section_conclusion}

The main contribution of this paper is that we present a unified way to analytically characterize optimal placements of bearing-only, range-only, or RSS sensors in 2D and 3D.
The necessary and sufficient conditions for optimal placement in 2D and 3D are proved.
The gradient control proposed in this paper is already sufficient for the purpose of numerical verification of the analytical analysis.
But it is a centralized control based on all-to-all communications and ensures local stability of the optimal placement set.
It is meaningful to study distributed or global stability guaranteed control laws or numerical algorithms in the future.


\bibliography{zsybib}
\bibliographystyle{ieeetr}

\end{document}

%% file: figures/figure_tikz_2Dconstruction.tex
\begin{tikzpicture}[scale=\myScale]
    \draw[->, very thin] (-5,0)--(10,0);
    \draw[->, very thin] (0,-4)--(0,5);

    \coordinate (z1) at (0,0);
    \coordinate (z2) at (8,0);
    \coordinate (z3) at (3.5,4);
    \draw [densely dotted, very thick] (z1)--(z2)--(z3)--cycle;
    \draw[] let \p1=(z1), \p2=(z2), \p3=(z3) in
                        (\x1/2+\x2/2,\y1/2+\y2/2) node[above] {$\ell_1$};
    \draw[] let \p1=(z1), \p2=(z2), \p3=(z3) in
                        (\x2/2+\x3/2,\y2/2+\y3/2) node[above right] {$\ell_3$};
    \draw[] let \p1=(z1), \p2=(z2), \p3=(z3) in
                        (\x1/2+\x3/2,\y1/2+\y3/2) node[above left] {$\ell_2$};

    \def\angleRadius{1.5cm}
    \draw[] let \p1=(z1), \p2=(z2), \p3=(z3), \n1={atan2(\x2-\x1,\y2-\y1)}, \n2={atan2(\x3-\x1,\y3-\y1)}  in
                        ($(\p1)!\angleRadius!(\p2)$) arc (\n1:\n2:\angleRadius);
    \draw[] let \p1=(z1), \p2=(z2), \p3=(z3), \n1={atan2(\x2-\x1,\y2-\y1)}, \n2={atan2(\x3-\x1,\y3-\y1)} in
                        (\p1)+(\n1/2+\n2/2:\angleRadius) node[right] {$\alpha_{12}$};
    \draw[] let \p1=(z2), \p2=(z1), \p3=(z3), \n1={atan2(\x2-\x1,\y2-\y1)}, \n2={atan2(\x3-\x1,\y3-\y1)}  in
                        ($(\p1)!\angleRadius!(\p2)$) arc (\n1:\n2:\angleRadius);
    \draw[] let \p1=(z2), \p2=(z1), \p3=(z3), \n1={atan2(\x2-\x1,\y2-\y1)}, \n2={atan2(\x3-\x1,\y3-\y1)} in
                        (\p1)+(\n1/2+\n2/2:\angleRadius) node[left] {$\alpha_{13}$};

    \def\unitlength{3cm}
    \draw[->, >=latex, very thick, red] (z1) -- ($(z1)!\unitlength!(z2)$) node [below] {$\bar{g}_i$ ($i<n_0$)};
    \draw[->, >=latex, very thick, red] let \p1=(z1), \p2=(z2), \p3=($(z2)!\unitlength!(z3)$) in
        (\p1) -- (\x3+\x1-\x2,\y3+\y1-\y2) node [above] {$\bar{g}_i$ ($i>n_0$)};
    \draw[->, >=latex, very thick, red] let \p1=(z1), \p2=(z3), \p3=($(z3)!\unitlength!(z1)$) in
        (\p1) -- (\x3+\x1-\x2,\y3+\y1-\y2) node [below] {$\bar{g}_{n_0}$};

\end{tikzpicture}